\documentclass[11pt,reqno]{amsart}
\usepackage{hyperref}
\usepackage{cleveref}
\allowdisplaybreaks[4]
\usepackage{amssymb}
\usepackage{mathrsfs}
\usepackage{amsthm}
\usepackage{amsmath}
\usepackage{graphicx}
\numberwithin{equation}{section}
\usepackage{bm}
\usepackage[T1]{fontenc}
\usepackage{bbm}
\usepackage{graphicx}
\usepackage{booktabs}
\usepackage{float}
\usepackage{mathrsfs}
\usepackage{subfig}
\usepackage{geometry}
\usepackage{color}
\theoremstyle{plain}
\newtheorem{lemma}{Lemma}[section]
\newtheorem{theorem}[lemma]{Theorem}

\newtheorem{prop}[lemma]{Proposition}

\newtheorem{assp}{Assumption}
\newtheorem{rem}{Remark}

\newcommand{\E}{\mathbb{E}}
\newcommand{\N}{\mathbb{N}}

\newcommand{\PP}{\mathbb{P}}
\newcommand{\RR}{\mathbb{R}}
\newcommand{\I}{\mathbb{I}}

\newcommand{\rd}{\mathrm d}

\def\<{\langle}
\def\>{\rangle}

\def\a{\alpha}
\def\nn{\nonumber}

\def\bc{\begin{center}}       \def\ec{\end{center}}
\def\ba{\begin{array}}        \def\ea{\end{array}}
\def\be{\begin{equation}}     \def\ee{\end{equation}}
\def\bea{\begin{eqnarray}}    \def\eea{\end{eqnarray}}
\def\beaa{\begin{eqnarray*}}  \def\eeaa{\end{eqnarray*}}

\begin{document}
\title[LIL for numerical approximation of Markov process]{The law of iterated logarithm for numerical approximation of  time-homogeneous Markov process}
\subjclass[2020]{60H35; 60J25; 60F17; 37M25}
\author{Chuchu Chen, Xinyu Chen, Jialin Hong}
\address{State Key Laboratory of Mathematical Sciences, Academy of Mathematics and Systems Science, Chinese Academy of Sciences, Beijing 100190, China,
\and 
School of Mathematical Sciences, University of Chinese Academy of Sciences, Beijing 100049, China}
\email{chenchuchu@lsec.cc.ac.cn; chenxinyu@amss.ac.cn; hjl@lsec.cc.ac.cn}
\thanks{This work is funded by the National Key R\&D Program of China under Grant (No. 2024YFA1015900),  by the National Natural Science Foundation of China (No. 12031020, No. 12461160278, and No. 12471386), and by Youth Innovation Promotion Association CAS}
\begin{abstract}
The law of the iterated logarithm (LIL) for the time-homogeneous Markov process with a unique invariant measure characterizes the almost sure maximum possible fluctuation of  time averages around the ergodic limit. Whether a numerical approximation can preserve this asymptotic pathwise behavior remains an open problem. 
In this work, we give a positive answer to this question and establish the LIL for the numerical approximation of  such a process
under verifiable assumptions.
The  Markov process is discretized by a decreasing time-step strategy, which yields the non-homogeneous numerical approximation but
facilitates a martingale-based analysis.
The key ingredient in proving the LIL for such numerical approximation lies in extracting a quasi-uniform  time-grid subsequence from the original non-uniform time grids and
establishing the LIL for a predominant martingale along it, 
while the remainder terms converge to zero.
Finally, we illustrate that our results can be flexibly applied to numerical approximations of a broad class of stochastic systems, including SODEs and SPDEs.
\end{abstract}

\keywords{Markov process $\cdot$ Numerical approximation $\cdot$  Law of the iterated logarithm $\cdot$ Ergodic limit}
\maketitle
\section{Introduction}\label{Tntr}
The  functional law of the iterated logarithm (LIL),  which is also known as Strassen's invariance principle, is first established for sums of independent and identically distributed random variables 
in \cite{Strassen1964}. 
This result is subsequently extended to the martingale setting in \cite{Heyde1973}, where a general LIL criterion for the square-integrable martingale is developed under mild regularity conditions on the increments.
Further progress has been made for the  time-homogeneous Markov process admitting a unique invariant measure~$\mu$.
Specifically, if there exist a class of test functions $f$ and  a limit constant $v>0$ such that
\begin{align*}  
\limsup_{T\to\infty} \frac{1}{\sqrt{2T\log \log T}} \int^T_{0} \big( f(X_t) -\mu(f) \big)\rd t =v, \quad \text{a.s.}
\end{align*}
and
\begin{align*}  
\liminf_{T\to\infty} \frac{1}{\sqrt{2T\log \log T}} \int^T_{0} \big( f(X_t) - \mu(f) \big)\rd t = - v, \quad \text{a.s.},
\end{align*}
then the process $\{X_t\}_{t\geq 0}$ is said to satisfy the LIL. 
In this setting, the LIL characterizes the  utmost fluctuation of the time average
$\frac 1T \int^T_0 f(X_t) \rd t$  around the ergodic limit $\mu(f)$, thereby providing a quantitative assessment of how effectively long-time averages approximate the ergodic limit.
For time-homogeneous Markov processes,  the LILs are established for both the stationary and non-stationary cases under different settings (see \cite{Baxter1976, Bhattacharya1982, Bolt2012, Czapla2021AA}).
As the solution processes of autonomous stochastic dynamical systems
naturally form time-homogeneous Markov processes, recent studies establish the LILs for concrete stochastic systems. 
Authors in \cite{Komorowski2014, Liu2025} establish the LILs for the two-dimensional Navier--Stokes equation 
driven by a non-degenerate Gaussian noise and by a time-periodic force together with a degenerate noise, respectively. 
For stochastic functional differential equations, the LILs are established for both finite and infinite delays \cite{Bao2020, Wu2022}, and also for regime-switching diffusion processes with infinite delay \cite{Wu2024}.

Since explicit solutions of stochastic  systems, such as stochastic ordinary differential equations (SODEs) and stochastic partial differential equations (SPDEs),  are typically inaccessible, numerical schemes serve as indispensable tools for investigating their dynamics and long-term behavior.
Given the established LILs for continuous-time stochastic systems, a natural  question arises: can the numerical approximation preserve the LIL for the exact solution, thereby capturing the maximum fluctuation of the time averages around the ergodic limit? More specifically, can we construct a numerical approximation $\{Y^{\tau}_{t_k}\}_{k\in\N^+}$ with some suitable time-step sequence $\tau:=\{\tau_k\}_{k\in\N^+}$ such that 
$$
\limsup_{n \to \infty} \frac{1}{\sqrt{2t_n\log\log t_n}} \sum^n_{k=1} \tau_k (f(Y^{\tau}_{t_k}) - \mu(f)) = v, \quad \text{a.s.}
$$
and
$$
\liminf_{n \to \infty} \frac{1}{\sqrt{2t_n\log\log t_n}} \sum^n_{k=1} \tau_k (f(Y^{\tau}_{t_k}) - \mu(f)) = -v, \quad \text{a.s.}?
$$
To the best of our knowledge, this question has not been addressed in the existing literature.

In this work, we prove that the LIL of the underlying continuous process is preserved by the numerical approximation with decreasing time steps, under assumptions on moment boundedness,  attractiveness, and convergence (see \Cref{thm_numerLIL}).
The decreasing time-step strategy enables the martingale-based analysis,
which is not feasible in the uniform-step setting. However, it simultaneously introduces non-homogeneity into the numerical approximation due to the non-uniformity of the time steps, making the analysis more technical.
To address this, 
we extract a quasi-uniform  subsequence $\{\tilde{t}_N\}_{N\in\N}$ from the original time grids $\{t_n\}_{n\in\N}$ and  construct a predominant martingale $\{ \tilde{\mathscr{M}}^{x,\tau}_{N}\}_{N\in\N}$ along it,
which we further prove to satisfy the LIL, while the remainder terms are proved to converge to zero. 

For proving the LIL of $\{ \tilde{\mathscr{M}}^{x,\tau}_{N}\}_{N\in\N}$, the 
key requirement is to establish the $L^2$- and almost sure limits of the corresponding martingale difference sequence $\{\tilde{\mathcal Z}^{x,\tau}_{N}\}_{N\in\N}$, which, however, 
requires new  technical analysis due to the lack of the time-shift invariance. 
First, in the non-homogeneous case, a unified functional representation is no longer available, whereas in the homogeneous setting it plays a crucial role in establishing the $L^2$-limit. Here, we couple the time-step sequence with the regularity parameters of the considered processes and show that the sum of the second moments of $\{\tilde{\mathcal Z}^{x,\tau}_N\}_{N\in\N}$ grows linearly with respect to the quasi-uniform time $\tilde{t}_N$ (see \Cref{prop_ZkSquareMoment}).
Second, since the existing tools for establishing the almost sure limit rely on the semigroup framework, which is not applicable here, we instead analyze the transition operator and estimate high-order moments for the deviation between the time average of $\{\tilde{\mathcal Z}^{x,\tau}_{N}\}_{N\in\N}$ and $v^2$, thereby obtaining the almost sure limit  (see \Cref{prop_almostsureZk}). 

Our theoretical results establish the LIL for decreasing-step numerical approximation, demonstrating that the numerical approximation faithfully captures the maximal pathwise fluctuation of the original process $\{X_t\}_{t\geq 0}$ around its ergodic limit and thereby confirming the  effectiveness of numerical time-averaged  estimators. Finally, we apply our theoretical results to specific SODEs and SPDEs, establishing the LILs for both the exact solutions and their numerical approximations with explicit expression of the limit constant $v$. This illustrates that our results can be flexibly applied to numerical approximations of a broad class of stochastic systems.
\subsection{Related works on numerical limit theorems}
Besides the LIL, the strong law of large numbers (SLLN) and the central limit theorem (CLT) are two classical limit theorems that  describe the long-term behavior of stochastic processes, with the SLLN describing the almost sure convergence of time averages and the CLT characterizing the limiting distribution of the normalized time averages. 
There are several studies establishing numerical limit theorems for specific stochastic systems, particularly the numerical SLLN and the numerical CLT. 

For SODEs with globally Lipschitz continuous coefficients, the numerical SLLN has been established for Euler-type methods (see \cite{Brehier2016, Mattingly2010}). 
The  numerical CLT is  proved in \cite{Pages2012} for the Euler--Maruyama (EM) method with decreasing step sizes, covering a broad class of Brownian ergodic diffusions. 
In \cite{Lu2022}, the authors further derive the numerical CLT together with the self-normalized Cramér-type moderate deviation principle for time-averaging estimators based on the EM method. 
For SODEs with non-globally Lipschitz continuous coefficients, the author in \cite{Jin2025} establishes both the  numerical SLLN and CLT for the backward EM method. 
In the infinite-dimensional setting, the numerical SLLN and CLT have been proved for full discretizations of stochastic parabolic equations (see \cite{Chen2023}).
For stochastic reaction–diffusion equations near the sharp interface limit, the numerical CLT for temporal discretizations is obtained in \cite{Cui2023}. In a more general setting, the authors in \cite{Pages2023} design a recursive algorithm to approximate the invariant distribution of Feller processes and prove a discretized version of the CLT presented in \cite{Bhattacharya1982}.
The authors in \cite{Chen2025} establish the numerical SLLN and CLT for numerical discretizations of time-homogeneous Markov processes with low-regularity test functionals.
\subsection{Outline}
The paper is organized as follows.
\Cref{sec_exactLIL} establishes the LIL for the general time-homogeneous Markov process, with the proof deferred to \Cref{Appendix1}.
\Cref{sec_numLIL} presents the LIL for the numerical approximation, which serves as our main theorem in this work, and introduces the extraction of the quasi-uniform subsequence as well as the construction of the predominant martingale along it.
The $L^2$- and almost sure limits of the corresponding martingale difference sequence, along with other useful estimates, are established  in \Cref{sec_martingaleProperties}. 
Detailed proofs of  both the LIL for the predominant martingale and the vanishing results of the remainder terms presented in \Cref{sec_numLIL} are then provided in \Cref{sec_ProofofChap3}.
Finally,  \Cref{sec_examples} provides illustrative examples, including SODEs and SPDEs,   demonstrating the applicability of our results to a broad class of stochastic systems. 
\subsection{General notation}
Let $(E, \|\cdot\|)$ be a real separable Banach space. We denote $\mathfrak B(E)$ the  $\sigma$-algebra of Borel subsets of $E$,  $\mathcal B_b(E)$ the set of all bounded Borel  functions, and $\mathcal P(E)$ the set of all probability measures defined on $(E,\mathfrak B(E))$.  Denote by  $\nu(f):=\int_{E}f(u)\mathrm d\nu(u)$ the integral of the measurable functional $f$ with respect to $\nu\in\mathcal P(E).$  
For fixed  constants $p\geq1$ and $\gamma\in(0,1]$,  let $\mathcal C_{p,\gamma}:=\mathcal C_{p,\gamma}(E;\mathbb R)$ denote the class of measurable functions $f:E\to\mathbb R$
such that
\begin{align}\label{5.2}
	\|f\|_{p,\gamma}:=\sup_{u\in E}\frac{|f(u)|}{1+\|u\|^{\frac{p}{2}}}+\sup_{\substack{u_1,u_2\in E\\u_1\neq u_2}}\frac{|f(u_1)-f(u_2)|}{d_{p,\gamma}(u_1,u_2)} <\infty,
\end{align}
where 
\begin{align}\label{5.1}
	d_{p,\gamma}( u_1,u_2):=(1\wedge\|u_1-u_2\|^{\gamma})(1+\|u_1\|^p+\|u_2\|^p)^{\frac{1}{2}}
\end{align}
is a quasi-metric  on $E$. 
With this definition, it follows that  $\mathcal C_{p_1,\gamma}\subset \mathcal C_{p_2,\gamma}$ whenever $p_2\geq p_1\geq1$ and  
$\mathcal C_{p,\gamma_2}\subset \mathcal C_{p,\gamma_1}$ whenever $0<\gamma_1\leq \gamma_2\leq1$.
The Wasserstein quasi-distance  associated with the quasi-metric $d_{p,\gamma}$ is given by 
\begin{align}\label{5.3}
	\mathbb W_{p,\gamma}(\nu_1,\nu_2):=\inf_{\pi\in\Pi(\nu_1,\nu_2) }\int_{E\times E}d_{p,\gamma}(u_1,u_2)\pi(\mathrm du_1,\mathrm du_2),
\end{align}
where  $\nu_1,\nu_2\in
\mathcal P_{p,\gamma}(E):=\{\nu\in\mathcal P(E):\int_{E\times E}d_{p,\gamma}(u_1,u_2)\mathrm d\nu(u_1)\mathrm d\nu(u_2)<\infty\}$ and $\Pi(\nu_1,\nu_2)$ is the set of all probability measures on $E\times E$ with marginals $\nu_1$ and $\nu_2$. 

Throughout and without any ambiguity,  $\log t$  denotes the natural logarithm $\log_e t$, $\lceil a\rceil$ and $\lfloor a \rfloor$ denote the minimal integer greater than or equal to $a$ and the maximal integer smaller than or equal to $a$, respectively, and $K$ denotes an arbitrary constant which may vary from one place to another. For real numbers $a$ and $b$, we write $a\vee{}b:=\max\{a,b\}$ and $a\wedge{}b:=\min\{a,b\}$.
\section{LIL for the time-homogeneous Markov process} \label{sec_exactLIL}
We consider the $E$-valued  time-homogeneous  Markov process $\{X_t\}_{t\geq0}$ defined on  a filtered probability space $( \Omega,\mathcal{F},\{\mathcal F_t\}_{t\ge 0},\PP )$. The  expectation with respect to $\PP$ is denoted by  $\E$.  For each $x\in E$, denote by $\{X_t^x\}_{t\geq0}$ the process with initial condition $X_0=x$, and for any fixed $s\geq0$, denote by $\{X^{x}_{s,t}\}_{t\geq s}$  the process with initial condition $X_{s}=x$. For any $t\geq0$, denote by $\mu_t^x(\rd y)$  
the probability distribution generated by  $X_t^{x}$, i.e., for any $A\in\mathfrak{B}(E)$,
$\mu_t^x(A)=\PP\big\{\omega\in\Omega: X_{t}^{x}\in A\big\}$. 

Using the time-homogeneity of $\{X_t\}_{t\geq0}$, for any $x\in E$  and $0 \leq s\leq t$, we define the linear operator $P_{s, t}: \mathcal B_b(E) \rightarrow \mathcal B_b(E)$ as
$$
 P_{s, t} f(x):=\E\big[f(X_{s, t}^x) \big]=\int_{E} f(u)\mu^x_{t-s}(\rd u), \qquad f\in\mathcal B_b(E).
$$
In particular, when $s=0$, we denote $P_{t}:= P_{0, t}$ for simplicity. 
With these notations, we impose the following assumption on the time-homogeneous Markov process $\{X^x_t\}_{t\geq 0}$.

 \begin{assp}\label{a1}
 Assume that the time-homogeneous  Markov process $\{X^x_t\}_{t\ge 0}$ satisfies the following conditions.
\begin{itemize}
\item[(\romannumeral 1)] There exist  constants $r\ge 2,\tilde r\ge 1$, and $L_1>0$ such that for any $x\in E$,
\begin{align*}
&\sup_{t\geq0}\E [\|X^{x}_t\|^{r}]\leq L_1(1+\|x\|^{\tilde rr}).
\end{align*}
\item[(\romannumeral 2)] There exist constants $\gamma_1\in(0,1]$, $\beta\in[0,r-1],\, L_2>0$,  and a  continuous function $\rho:[0,+\infty)\to[0,+\infty)$  with $\int_{0}^{\infty} (\rho(t))^{\gamma_1} \mathrm dt<\infty$  and $\sum^\infty_{k=1} (\rho(k))^{\gamma_1} <\infty$ 
such that for any $x,y\in E,$
\begin{align*}
\big(\E[\|X^{x}_{t}-X^{y}_{t}\|^2]\big)^{\frac{1}{2}}\leq L_2\|x-y\|(1+\|x\|^{\beta}+\|y\|^{\beta})\rho(t).
\end{align*}
\item[(\romannumeral 3)] There exist constants $l\in[0, r]$ and $\tilde{l}\in(0, 1]$ such that for $t\geq 0$,
$$
\big(\E[\|X^{x}_{t}-x\|^2]\big)^{\frac{1}{2}}\leq  L_3 (1+\|x\|^l) t^{\tilde{l}}.
$$
\end{itemize}
 \end{assp} 
 Under \Cref{a1},  the existence and uniqueness of the invariant measure are ensured, and the following LIL for $\{X^x_t\}_{t\geq 0}$ holds. The proof is provided in  \Cref{Appendix1}.
 \begin{prop}\label{thm_exactLIL} 
 Let \Cref{a1} hold,  $p\in [1,r]$ 
 such that $2p\tilde{r} +4(1+\beta)\gamma_1\leq r$ and $2\tilde{r}(p\tilde{r}+(2+3\beta)\gamma_1) \leq r$. Then $\{X^x_t\}_{t\ge 0}$ admits a unique invariant measure $\mu\in\mathcal P(E),$ and fulfills the LIL: For any $f\in\mathcal C_{p,\gamma_1},$ it holds that
 \begin{align*}
\limsup_{t\to \infty} \frac{ \int^t_0 (f(X^x_s)-\mu(f)) \rd s}{\sqrt{2t\log\log t} } =v,\quad \text{a.s.} \quad\text{and}\quad 
 \liminf_{t\to \infty} \frac{ \int^t_{0} (f(X^x_s)-\mu(f)) \rd s}{\sqrt{2t\log\log t}  } = -v, \quad   \text{a.s.,}
 \end{align*}
 where $v:= \sqrt{2\mu\big((f-\mu(f))\int_{0}^{\infty}(P_{t}f-\mu(f))\mathrm dt\big) }$. 
 \end{prop}
 Here we provide several important estimates related to the Wasserstein quasi-distance, which will be frequently used throughout this work.
 Note  that  
 if $\nu_1(\|\cdot\|^{p})<\infty $ and $\nu_2(\|\cdot\|^{p})<\infty $ with $p\geq 1$, then for $\gamma\in (0, 1]$, 
 \begin{align*}
 	\mathbb W_{p,\gamma}(\nu_1,\nu_2)\nn
 	&\leq  \big(1+\nu_1(\|\cdot\|^{p})+\nu_2(\|\cdot\|^{p})\big)^{\frac{1}{2}} \Big(\inf_{\pi\in\Pi(\nu_1,\nu_2) }\int_{E\times E}(1\wedge\|u_1-u_2\|^{2\gamma})\pi(\mathrm du_1,\mathrm du_2)\Big)^{\frac{1}{2}}\nn\\
 	&\leq \big(1+\nu_1(\|\cdot\|^{p})+\nu_2(\|\cdot\|^{p})\big)^{\frac{1}{2}}(\mathbb W_{2}(\nu_1,\nu_2))^{\gamma},
 \end{align*}
 where  \eqref{5.1} and the H\"older inequality are used, and $\mathbb{W}_2$ is a bounded-Wasserstein distance defined by 
 \begin{align*}
 	\mathbb{W}_2(\nu_1, \nu_2):=\Big(\inf_{\pi\in\Pi(\nu_1, \nu_2)}\int_{E\times E}(1\wedge\|u_1-u_2\|^2)\pi(\mathrm du_1,\mathrm du_2)\Big)^{\frac{1}{p}}.
 \end{align*}
 This leads to that for any $f\in\mathcal C_{p,\gamma}$ with $p\geq 1$ and $\gamma\in (0, 1]$, 
 \begin{align}\label{distanceOfTwoMeasure}
 	|\nu_1(f)-\nu_2(f)|&=\inf_{\pi\in \Pi(\nu_1,\nu_2)}\Big|\int_{E\times E}(f(u_1)-f(u_2))\pi(\mathrm du_1,\mathrm du_2)\Big|\nn\\
 	&\leq
 	\|f\|_{p,\gamma}\mathbb W_{p,\gamma}(\nu_1,\nu_2)\leq \|f\|_{p,\gamma} \big(1+\nu_1(\|\cdot\|^{p})+\nu_2(\|\cdot\|^{p})\big)^{\frac{1}{2}}(\mathbb W_{2}(\nu_1,\nu_2))^{\gamma}.
 \end{align}
With the estimate between measures given in \eqref{distanceOfTwoMeasure}, we now show that the limit constant $v$ in \Cref{thm_exactLIL} is well-defined. 
Indeed, the conditions in \Cref{thm_exactLIL} imply that $\frac{p}{2} (1+\tilde{r}) +(1+\beta) \gamma_1 \leq r$, which ensures that $v^2$ is bounded  as follows 
	\begin{align*}
		v^2 &=2\mu\Big((f-\mu(f))\int_{0}^{\infty}(P_{t}f-\mu(f))\mathrm dt\Big) \\ 
		&\leq K \mu\Big(\big[\|f\|_{p,\gamma_1} (1+\|\cdot\|^{\frac p2})\big] \int^\infty_{0} \big[\|f\|_{p,\gamma_1} \big(1+ \|\cdot\|^{\frac{p\tilde{r}}{2}+(1+\beta)\gamma_1} \big)(\rho(t))^{\gamma_1} \big] \ \rd t \Big)   \\ 
		&\leq K\|f\|^2_{p,\gamma_1} \mu\big(1+\|\cdot\|^{\frac p2(1+\tilde{r})+(1+\beta)\gamma_1}\big) 
		\leq   K\|f\|^2_{p,\gamma_1}, \quad f\in \mathcal{C}_{p,\gamma_1}.
	\end{align*}
Here, we  make use of the integrability of $(\rho(t))^{\gamma_1}$,   the estimate 
	\begin{align} \label{equ_apprExactPttoMu}
		\big| P_{t}f(x)-\mu(f) \big|  &\leq \|f\|_{p,\gamma_1} \big(1+  \E\big[\|X_{t}^{x}\|^p\big] +\mu(\|\cdot\|^{p})\big)^{\frac{1}{2}}(\mathbb W_{2}(\mu_{t}^{x},\mu))^{\gamma_1}\nn \\
		&\leq  K \|f\|_{p,\gamma_1} \big(1+ \|x\|^{\frac{p\tilde{r}}{2}+(1+\beta)\gamma_1} \big)(\rho(t))^{\gamma_1}, \quad x\in E, 
	\end{align}
and the fact that for any  $\bar{r}\leq r$,
\begin{align} \label{equ_boundOfNormOfmu}
	\mu(\| \cdot \|^{\bar{r}}) \leq \lim_{t\rightarrow\infty}\Big( \int_{E} \| u\|^r \mu_t^{\bf 0} ( \rd u) \Big)^{\bar{r}/r} \leq \Big(\sup_{t\geq 0} \E[ \| X^{\bf 0}_t \|^r] \Big)^{\bar{r}/r} \leq L_4,
\end{align}
as implied by Assumption \ref{a1}.

\section{LIL for  the non-homogeneous approximation}\label{sec_numLIL}
\subsection{Construction of the numerical approximation}
We define the grid points $\{ t_n := \sum^n_{k=0} \tau_k\}_{n\in \mathbb{N}}$  with $t_0=0$, where  $\tau := \{\tau_k\}_{k \in \mathbb{N}}$ is the time-step sequence. 
Assume that,  for all $k \in \mathbb{N}^{+}$, 
\begin{align*}
	0 < \tau_k  \leq \overline{\tau}:= \sup_{k \in \mathbb{N}^{+}} \tau_k, \quad \lim\limits_{k \to  \infty} \tau_k = 0, \quad   \lim\limits_{n \to  \infty}t_n= \lim\limits_{n \to  \infty} \sum^n_{k=1} \tau_k= \infty.
\end{align*}   
For each given $\tau$, let $\{Y^{\tau}_{t_k}\}_{k\in\mathbb N}$ denote the approximation of $\{X_t\}_{ t\geq0}$.  Then $\{Y^{\tau}_{t_k}\}_{k\in\mathbb N}$ is a non-homogeneous Markov chain.
For each $x\in E$, denote by $\{Y^{x, \tau}_{t_k}\}_{k\in\mathbb N}$ the  chain with initial condition $Y^{\tau}_{t_0}=x$ and for any fixed $i\in\N$, denote by $\{Y^{x, \tau}_{t_i,t_j}\}_{j\geq i}$  the chain with initial condition $Y^{\tau}_{t_i}=x$.
For any $x\in E$ and $m,n \in\N$ with $m \leq n$, we denote by $\mu^{x, \tau}_{t_m,t_n}(\rd y)$ the 
 probability distribution generated by $Y^{x,\tau}_{t_m,t_n}$, i.e., for any $A\in \mathfrak{B}(E)$, $	\mu^{x, \tau }_{t_m,t_n}(A)= 	\mathbb{P}\{\omega \in \Omega: Y^{x, \tau}_{t_m, t_{n}} \in  A\}.$
For any $x\in E$  and $m,n\in\N^+$ with $m\leq n$, we define the linear operator $P_{t_m, t_n}^{\tau}: \mathcal{B}_b(E) \rightarrow \mathcal{B}_b(E)$ as
$$
P_{t_m, t_n}^{\tau} f(x):=\E\big[f(Y^{x, \tau}_{t_m,t_n}) \big]=\int_{E} f(u)\mathrm \mu^{x, \tau}_{t_m, t_n}(\rd u), \qquad f\in\mathcal{B}_b(E).
$$
In particular, when $m=0$, we use the notation $\mu^{x,\tau}_{t_n}:= \mu^{x,\tau}_{t_0, t_n}$ and $P_{t_n}^{\tau} := P^\tau_{t_0, t_n}$ for simplicity.

\subsection{LIL for the non-homogeneous approximation}
We impose the following assumptions on the non-homogeneous approximation $\{Y^{x, \tau}_{t_k}\}_{k\in\mathbb N}$. 
 
 \begin{assp}\label{a2}
Assume that $\{Y^{x,\tau}_{t_k}\}_{k\in\mathbb N}$ satisfies  the following conditions.
\begin{itemize}
\item[\textup{(i)}] 
There exist  constants $q\ge 2,\tilde q\ge 1$, and $L_4>0$ such that
\begin{align*}
\sup_{k\geq m}\E [\|Y^{x,\tau}_{t_m, t_{k}}\|^{q}]\leq L_4 (1+\|x\|^{\tilde qq})
\end{align*}
for any fixed $m\in \N$.

\item[\textup{(ii)}] There exist constants 
$\kappa \in[0,q-1], L_5>0$, and a    function $\rho^{\tau}:[0,+\infty)\to[0,+\infty)$ 
 such that for any $x, y\in E$, $m,n \in\N$ with $n\geq m$,
\begin{align*}
\big(\E[\|Y^{x,\tau}_{t_m,t_n}-Y^{y,\tau}_{t_m,t_n}\|^2]\big)^{\frac{1}{2}}&\leq L_5\|x-y\|(1+\|x\|^{\kappa}+\|y\|^{\kappa})\rho^{\tau}(t_n-t_m).
\end{align*}
\end{itemize}
 \end{assp} 

\begin{assp}\label{a3}
Assume that there exist $\a\in\mathbb R_+$ and  $L_6>0$ such that for any $x\in E$ and $m,n \in \mathbb{N}^+$ such that $n\geq m$,  
 $$
 \big(\E[\|X^{x}_{t_m, t_{n}}-Y^{x,\tau}_{t_m, t_n}\|^2]\big)^{\frac{1}{2}}\leq L_6(1+\|x\|^{\tilde r\vee \tilde q}) 
 \tau_n^{\alpha}.$$
\end{assp}

Under Assumptions \ref{a1}--\ref{a3}, we prove that 
the numerical approximation $\{Y^{x,\tau}_{t_k}\}_{k\in\mathbb N}$ preserves 
the LIL of the underlying Markov process, 
characterizing the almost sure maximum possible fluctuation of the numerical time averages around the ergodic limit of  $\{X^{x}_{t}\}_{t\geq 0}$. The result is presented in the following theorem. 
For simplicity, we denote $\tilde{d}:=(1+\beta)\vee\tilde{r}\vee\tilde{q}$ throughout this paper. 
\begin{theorem}\label{thm_numerLIL} 
Let  Assumptions \ref{a1}--\ref{a3} hold. Suppose that there exist constants  $p\in [1,r \wedge \frac q4]$ and $\gamma\in[\gamma_1,1]$ such that  $\frac p2(1+\tilde{r})+(1+\beta)\gamma_1 \leq r$, $4p\tilde{q}+8\tilde{d}\gamma\leq q$, $\frac p2+ ([(\frac{p}{2}+\gamma )(\tilde{r} \vee \tilde{q})]\vee [( \frac{p}{2} + l \gamma)\tilde{r}])\leq r$, $\frac p2+\frac{p\tilde{q}}{2}+\gamma (\tilde{d} \vee \kappa) \leq r \wedge q$, and $2\tilde{q}[\frac p2+\frac{p\tilde{q}}{2}+\gamma (\tilde{d} \vee \kappa)]+ 4\tilde{d}\gamma \leq q$. 
If the step-size sequence $\tau$ satisfies the following conditions:
  \begin{itemize}
	\item[(\romannumeral 1)] $ \sum^{\infty}_{k=1}  \tau_{k}^{1+ \gamma \alpha } <\infty,$ \\ \vspace{-3mm} 
	\item[(\romannumeral 2)] $\sup_{k\in\N} \sum^\infty_{i=k}  \tau_i(\rho(t_i-t_{k}))^\gamma  <\infty$,  \\ \vspace{-3mm}
	\item[(\romannumeral 3)] $\sup_{k\in\N} \sum^\infty_{i=k}  \tau_i(\rho^\tau(t_i-t_{k}))^\gamma  <\infty$,  \\ \vspace{-3mm}
	\item[(\romannumeral 4)] $\lim_{n\to \infty} (\sum^n_{k=1} \tau_k)^{-1}  \sum^{n}_{k=1} ( \tau_{k-1}  ( \sum^\infty_{i=k-1}  \tau_i^{\gamma \alpha} \tau_{i+1} +  \sum^\infty_{i=k} \tau_i^ {1+\tilde{l}\gamma} ) )=0$,
\end{itemize}
then, 
for any $x\in E$ and $f\in \mathcal{C}_{p,\gamma}$,  the numerical approximation 
$\{Y^{x,\tau}_{t_k}\}_{k\in\N}$ obeys the LIL. Specifically, it holds that
	$$
\limsup_{n\to \infty} \frac{\sum^n_{k=1} \tau_k (f(Y^{x,\tau}_{t_k})-\mu(f)) }{\sqrt{2 t_{n} \log \log t_{n}}}  =v, \quad \text{a.s.} $$
and
$$
  \liminf_{n\to \infty} \frac{\sum^n_{k=1} \tau_k (f(Y^{x,\tau}_{t_k})-\mu(f))  }{\sqrt{2 t_{n} \log \log t_{n}}}  = -v, \quad \text{a.s.}
$$
\end{theorem}
\begin{rem}
	If the step-size sequence $\tau$ is strictly decreasing, i.e., $\tau_{k+1}\leq \tau_{k}$ for all $k\in \N^+$, then condition 
	 (\romannumeral 4) can be replaced by 
	\begin{align} \label{equ_equalCondition4}
	\lim_{n\to \infty} (\sum^n_{k=1} \tau_k)^{-1}  \sum^{n}_{k=1} ( \tau_{k-1}  \sum^\infty_{i=k-1}  \tau_i^{(1+\gamma \alpha)\wedge(1+\tilde{l}\gamma)} )=0.
	\end{align}
For example, when $\tau_k=\frac 1k$ for $k\in \N^+$,   \eqref{equ_equalCondition4} follows naturally from the fact that $$ \sum^\infty_{i=k-1}  \tau_i^{(1+\gamma \alpha)\wedge(1+\tilde{l}\gamma)} \leq K \tau_{k-1}^\delta$$ for some $\delta> 0$.
\end{rem}
The proof of \Cref{thm_numerLIL} is based on extracting a quasi-uniform subsequence  from the original time grids that  behave analogously to  integer points
and constructing a predominant martingale along it. 
The argument is then completed by establishing the LIL for this predominant martingale and showing that the remainder terms vanish asymptotically.

For any $k\in\N$,  let $n_{(k)} \in \N$ be the index such that $t_{n_{(k)}} \leq k < t_{n_{(k)}+1}$.  We define  $\{\tilde{t}_k\}_{k\in\N}$ by $\tilde{t}_k= t_{n_{(k)}}$ for each $k$ and define the corresponding time-step sequence by $\tilde{\tau}:=\{ \tilde{\tau}_k\}_{k\in\N}$ with $\tilde{\tau}_0=0$ and $\tilde{\tau}_k= \tilde{t}_k- \tilde{t}_{k-1}$ for $k\in\N^+$. It is clear that $\{\tilde{t}_k\}_{k\in\N}$  is a subsequence of $\{t_k\}_{k\in\N}$ and
 $\tilde{\tau}_k=\sum^{n_{(k)}}_{i=n_{(k-1)}+1} \tau_i$ for $k\in\N^+$. 
With this definition, we also have 
\begin{align} \label{equ_tildetOneStepBound}
	\sup_{k\in\N} | t_{n_{ (k)}-1} -t_{n_{(k-1)}-1}|
	\leq \sup_{k\in\N} ( k -(k-1 -2\bar{\tau} ) )= 1+2\bar{\tau}.
\end{align}
Moreover, for  any $k_0\in\N^+$ such that $k_0>\bar{\tau}$ and any  $\varepsilon>0$, we obtain the following result: 
\begin{align} \label{equ_sumofTildeT}
	\sum^\infty_{k=k_0} (\tilde{t}_k)^{-(1+\varepsilon)} \leq \sum^\infty_{k=k_0} (k-\bar{\tau})^{-(1+\varepsilon)} <\infty.
\end{align}
In fact, the subsequence  $\{\tilde{t}_k\}_{k\in \N}$ defined  above  is a quasi-uniform time-grid subsequence.

With this extraction of the subsequence of time grids, we next construct a predominant martingale along it.
Namely, we decompose the sum $	\sum_{i=1}^{k}\tau_i (f(Y^{x,\tau}_{t_i})-\mu(f))$ into three parts: the difference between the original sequence and the extracted subsequence, 
a predominant martingale along the subsequence, and a remainder  term along the subsequence.
To locate the nearest point in the extracted subsequence smaller than $t_k$ for $k\in\N$, we denote by $\tilde{k} \in \N$  the index such that  $\tilde{t}_{\tilde{k}} \leq t_k < \tilde{t}_{\tilde{k}+1}$, or equivalently, $ t_{n_{(\tilde{k})}} \leq t_k < t_{n_{(\tilde{k}+1)}}$.  With this notation, we introduce the decomposition  for $k\in\N^+$,
\begin{align*} 
&\quad 	\sum_{i=1}^{k}\tau_i \big(f(Y^{x,\tau}_{t_i})-\mu(f)\big)  \\
	&= \Big(	\sum_{i=1}^{k}\tau_i \big(f(Y^{x,\tau}_{t_i})-\mu(f)\big) -\sum^{n_{(\tilde{k})}}_{i=1} \tau_i  (f(Y^{x,\tau}_{t_i}) - \mu(f)) \Big)+ \sum^{n_{(\tilde{k})}}_{i=1} \tau_i  (f(Y^{x,\tau}_{t_i}) - \mu(f)) \\
	&=R^{x,\tau}_{k}+\tilde{\mathscr M}^{x,\tau}_{\tilde{k}} +\tilde{\mathscr R}^{x,\tau}_{\tilde{k}}, 
\end{align*}
  where  
  \begin{align} \label{equ_defR2}
  	R^{x,\tau}_{k}:= \sum^{k}_{i=n_{(\tilde{k})}+1} \tau_i  (f(Y^{x,\tau}_{t_i}) - \mu(f))
  \end{align}
  represents the remainder term between the original sequence and the extracted subsequence,
  $	\tilde{	\mathscr M}^{x,\tau}_{\tilde{k}}:=	\mathscr M^{x,\tau}_{n_{(\tilde{k})}}$ represents the predominant martingale along the subsequence $\{\tilde{t}_k\}_{k\in\N}$ with
\begin{equation}
	\begin{aligned}\label{mar_eq}
	\mathscr M^{x,\tau}_{k}&:= 
		\sum_{i=0}^{\infty} \tau_i \Big(\E\big[f(Y_{t_i}^{x,\tau})|\mathcal F_{t_{k}}\big]-\mu(f)\Big)-\sum_{i=0}^{\infty}\tau_i \Big(\E\big[f(Y_{t_i}^{x,\tau})|\mathcal F_{0}\big]-\mu(f)\Big),
	\end{aligned}
\end{equation}  and
$	\tilde{	\mathscr R}^{x,\tau}_{\tilde{k}}:=	\mathscr R^{x,\tau}_{n_{(\tilde{k})}}$ represents the remainder term along the subsequence $\{\tilde{t}_k\}_{k\in\N}$ with
\begin{equation}
	\begin{aligned} \label{equ_defR}
		\mathscr R^{x,\tau}_{k}& :=-\sum_{i=k+1}^{\infty}\tau_i \Big(\E\big[f(Y_{t_i}^{x,\tau})|\mathcal F_{t_{k}}\big]-\mu(f)\Big)+\sum_{i=0}^{\infty}\tau_i\Big(\E\big[f(Y_{t_i}^{x,\tau})|\mathcal F_{0}\big]-\mu(f)\Big).
	\end{aligned}
\end{equation}

We establish the LIL for the martingale $\{\tilde{\mathscr M}^{x,\tau}_{N}\}_{N\in\N}$, as well as the decay properties of the remainder terms, namely  $\lim_{N\to \infty} (\tilde{t}_N)^{-\frac 12} \tilde{\mathscr R}^{x, \tau}_N =0$ and $\lim_{n\to \infty} (t_n)^{-\frac 12} R^{x, \tau}_n =0$, in the following propositions, whose proofs are postponed to \Cref{sec_ProofofChap3}. 
 \begin{prop}\label{prop_LILforMartingale} 
 	Under the assumptions  in \Cref{thm_numerLIL},
 the LIL holds for  the martingale $\{ \tilde{\mathscr M}^{x,\tau}_{N}\}_{N\in\N}$ corresponding to the test function $f\in\mathcal{C}_{p, \gamma}$ and initial value $x\in E$. 
 	Specifically, we have  
 	$$
 	\limsup_{N \to \infty} \frac{\tilde{\mathscr M}^{x,\tau}_{N}}{\sqrt{2 \tilde{t}_{N} \log \log \tilde{t}_{N}  }}  = v, \quad \text{a.s.} \qquad \text{and} \qquad   \liminf_{N\to \infty} \frac{\tilde{\mathscr M}^{x,\tau}_{N} }{\sqrt{2 \tilde{t}_{N} \log \log \tilde{t}_{N}  }}  = -v, \quad \text{a.s.}
 	$$
 \end{prop}
 \begin{prop}\label{prop_RemainTermLimit}
 	Let Assumptions \ref{a1}--\ref{a3} hold. Suppose that there exists constants $c>2$, $p\in [1, r \wedge q]$ and $\gamma\in(0, 1]$ such that 
 	$c(\frac{p\tilde{q}}{2}+\tilde{d}\gamma) \leq q$. If the step-size sequence $\tau$ satisfy  $\sum^\infty_{i=1} \tau_i^{1+\gamma\alpha} <\infty$ and $\sup_{k\in\N} \sum^\infty_{i=k} \tau_i(\rho(t_i-t_k))^\gamma<\infty$,  
 	then for any $x\in E$ and $f\in C_{p,\gamma}$, we have
 	\begin{align*}
 		\lim_{N\to \infty}	\frac{1}{( \tilde t_{N})^{\frac 12}} \tilde{\mathscr R}^{x,\tau}_{N} =0, \quad {\text{a.s.}}
 	\end{align*}
 \end{prop}
 \begin{prop} \label{prop_RemainderLimit2} 
 	Let Assumptions \ref{a1}--\ref{a3} hold. Suppose that there exists constants $c>2$, $p\geq 1$ and $\gamma\in(0, 1]$ such that 
 	$\frac{cp}{2} \leq q$.
 	Then for any $x\in E$ and $f\in C_{p,\gamma}$, we have
 	\begin{align*}
 		\lim_{k\to \infty}	\frac{1}{(t_k)^{\frac 12}} R^{x,\tau}_{k} =0, \quad {\text{a.s.}}
 	\end{align*}
 \end{prop}
Building on Propositions \ref{prop_LILforMartingale}--\ref{prop_RemainderLimit2}, 
we now present the proof of \Cref{thm_numerLIL}.

 \begin{proof}[Proof of  \Cref{thm_numerLIL}]
 
 Under the assumptions of   \Cref{thm_numerLIL}, 
 we know that Proposition \ref{prop_RemainTermLimit} and Proposition \ref{prop_RemainderLimit2} hold for $c=4$.
 	Combining Propositions \ref{prop_LILforMartingale}--\ref{prop_RemainderLimit2},  we arrive at 
 	\begin{align}  \label{equ_MainthmUpperbound}
 	&\quad \limsup_{k\to \infty} \frac{	\sum_{i=1}^{k}\tau_i \big(f(Y^{x,\tau}_{t_i})-\mu(f)\big)}{\sqrt{2t_k\log\log t_k}} \nn\\
 	&= \limsup_{k\to \infty} \frac{\tilde{\mathscr M}^{x,\tau}_{\tilde{k}} +\tilde{\mathscr R}^{x,\tau}_{\tilde{k}}  }{\sqrt{2t_k\log\log t_k}}+ \lim_{k\to \infty} \frac{ R^{x,\tau}_{k}  }{\sqrt{2t_k\log\log t_k}} \nn \\
 		&\leq \limsup_{\tilde{k}\to \infty} \frac{\tilde{\mathscr M}^{x,\tau}_{\tilde{k}}  }{\sqrt{2t_{n_{(\tilde{k})}}\log\log t_{n_{(\tilde{k})}}}}+ \lim_{\tilde{k}\to \infty} \frac{\tilde{\mathscr R}^{x,\tau}_{\tilde{k}}  }{\sqrt{2t_{n_{(\tilde{k})}}\log\log t_{n_{(\tilde{k})}}}}   + \lim_{k\to \infty} \frac{ R^{x,\tau}_{k}  }{\sqrt{2t_k\log\log t_k}}= v, \quad \text{a.s.}
 	\end{align}
 	and
 	\begin{align}  \label{equ_MainthmLowerbound}
 		&\quad \limsup_{k\to \infty} \frac{\sum_{i=1}^{k}\tau_i \big(f(Y^{x,\tau}_{t_i})-\mu(f)\big)}{\sqrt{2t_k\log\log t_k}} \nn \\
 		&\geq  \limsup_{\tilde{k}\to \infty} \frac{\tilde{\mathscr M}^{x,\tau}_{\tilde{k}}  }{\sqrt{2t_{n_{(\tilde{k+1})}}\log\log t_{n_{(\tilde{k+1})}}}}+ \lim_{\tilde{k}\to \infty} \frac{\tilde{\mathscr R}^{x,\tau}_{\tilde{k}}  }{\sqrt{2t_{n_{(\tilde{k+1})}}\log\log t_{n_{(\tilde{k+1})}}}}  + \lim_{k\to \infty} \frac{ R^{x,\tau}_{k}  }{\sqrt{2t_k\log\log t_k}} \nn \\
 		&\geq \limsup_{\tilde{k}\to \infty} \frac{\tilde{\mathscr M}^{x,\tau}_{\tilde{k}}   }{\sqrt{2t_{n_{(\tilde{k})}}\log\log t_{n_{(\tilde{k})}}}} \frac{\sqrt{2t_{n_{(\tilde{k})}}\log\log t_{n_{(\tilde{k})}}}} {\sqrt{2t_{n_{(\tilde{k}+1)}}\log\log t_{n_{(\tilde{k}+1)}}}}  \nn \\
 		&\quad +\lim_{\tilde{k}\to \infty} \frac{\tilde{\mathscr R}^{x,\tau}_{\tilde{k}}   }{\sqrt{2t_{n_{(\tilde{k})}}\log\log t_{n_{(\tilde{k})}}}} \frac{\sqrt{2t_{n_{(\tilde{k})}}\log\log t_{n_{(\tilde{k})}}}} {\sqrt{2t_{n_{(\tilde{k}+1)}}\log\log t_{n_{(\tilde{k}+1)}}}}= v, \quad \text{a.s.}
	\end{align}
 	Here we use  $	\lim_{\tilde{k}\to \infty} t_{n_{(\tilde{k})}}/t_{n_{(\tilde{k}+1)}}= 1$, which follows from
 \begin{align*}
 	\lim_{\tilde{k}\to \infty} \frac{t_{n_{(\tilde{k})}}}{t_{n_{(\tilde{k}+1)}}} \geq	\lim_{\tilde{k}\to \infty} \frac{\tilde{k}-\bar{\tau}}{\tilde{k}+1} =1
 \end{align*}
 and
 $		\lim_{\tilde{k}\to \infty} t_{n_{(\tilde{k})}}/t_{n_{(\tilde{k}+1)}}\leq 1$. 
 Combining \eqref{equ_MainthmUpperbound} and \eqref{equ_MainthmLowerbound}  yields the $\limsup$ result in \Cref{thm_numerLIL}.
 	By taking $f$ as $-f$, we similarly obtain 
 	\begin{align*}
 		\liminf_{k\to \infty} \frac{ \sum^{k}_{j=1}\tau_j \big( f(Y^{x,\tau}_{t_j}) - \mu(f) \big)}{\sqrt{2t_k\log\log t_k}} = -v, \quad \text{a.s.},
 	\end{align*}
 	which gives the  $\liminf$ result in \Cref{thm_numerLIL}.
 \end{proof}
\section{Estimates of the  martingale difference sequence}\label{sec_martingaleProperties} 
In this section, we provide estimates of the martingale difference sequence. 
We first show that $\{\mathscr M^{x,\tau}_{k}\}_{k\in\mathbb N}$ 
 forms  a martingale (see \Cref{prop_martingaleIntegrability}), and then prove the moment boundedness for the corresponding martingale difference sequence (see \Cref{prop_ZkIntegrable}). Building on these estimates, we further derive the $L^2$-limit and the almost sure limit of the martingale difference sequence associated with the predominant martingale (see \Cref{prop_ZkSquareMoment} and \Cref{prop_almostsureZk}), which play a key role in the proof of \Cref{prop_LILforMartingale}. 

We first give some estimates about the measures.
Let $\xi$ denote the $E$-valued random variable such that $\mu=\mathbb{P}\circ\xi^{-1}$.
Using Assumptions \ref{a1}--\ref{a3} and the definition of the Wasserstein distance, for any $m ,k \in\N^+$ with $k\geq m$, we obtain the following estimates:
\begin{align*}
\mathbb{W}_2(\mu^{x}_{t_m, t_k},  \mu_{t_m, t_k}^{x,\tau})
&\leq \big(\E \big[\| X^{x}_{t_m, t_k} - Y^{x, \tau}_{t_m, t_k} \|^2\big]\big)^{\frac12} 
\leq K \big( 1+\|x\|^{\tilde{r}\vee\tilde{q}} \big) \tau_k^{\alpha} 
\end{align*}
and
\begin{align*}  
	\mathbb{W}_2(\mu^{x}_{t_m, t_k},  \mu) &\leq \E\Big[\big(\E \big[\| X^{x}_{t_m, t_k} - X^{u}_{t_m, t_k} \|^2\big]\big)^{\frac12}  \Big|_{u=\xi} \Big] \nn \\
	&\leq K \E \big[ \|\xi-x\| (1+\|\xi\|^\beta+\|x\|^{\beta} )\big] \rho(t_k-t_m)  \nn \\
	&\leq K(1+\|x\|^{1+\beta}) \rho(t_k-t_m).
\end{align*}
These estimates yield
\begin{align}  \label{equ_apprIM}
\mathbb{W}_2(\mu^{x, \tau}_{t_m, t_k},  \mu) &\leq \mathbb{W}_2(\mu^{x}_{t_m, t_k},  \mu_{t_m, t_k}^{x,\tau}) + \mathbb{W}_2(\mu^{x}_{t_m, t_k},  \mu)\nn \\
&\leq K(1+\|x\|^{(1+\beta)\vee\tilde{r}\vee\tilde{q}}) \max\{ \tau_k^{\alpha},  \rho(t_k-t_m)\}.
\end{align}
Using  \eqref{distanceOfTwoMeasure},  \eqref{equ_boundOfNormOfmu}, and  \eqref{equ_apprIM},  for any $k,m\in \N$ such that $k\geq m$, $u \in E$, and $f\in \mathcal{C}_{p,\gamma}$ with $p \in [1,r \wedge q]$ and $\gamma \in (0,1]$, we have 
	\begin{align} \label{equ_PtkApprmuFromtm}
		| P^\tau_{t_m, t_k} f(u) -\mu(f) |  &\leq \|f\|_{p,\gamma} \big(1+  \E\big[\|Y_{t_m, t_k}^{u,\tau}\|^p\big] +\mu(\|\cdot\|^{p})\big)^{\frac{1}{2}}(\mathbb W_{2}(\mu_{t_m, t_k}^{u,\tau},\mu))^{\gamma} \nn \\
		&\leq K  \|f\|_{p,\gamma}  (1+\|u\|^{\frac{p\tilde{q}}{2}+\tilde{d}\gamma})  \max\big\{ \tau_k^{\gamma \alpha}, (\rho(t_k-t_m))^\gamma \big\}.
	\end{align}
Specially,  when $m=0$, \eqref{equ_PtkApprmuFromtm} turns into 
\begin{equation} \label{equ_PtkApprmu}
	\begin{aligned}
		| P^\tau_{t_k} f(u) -\mu(f) |  
		&\leq K  \|f\|_{p,\gamma}  (1+\|u\|^{\frac{p\tilde{q}}{2}+\tilde{d}\gamma})  \max\big\{ \tau_k^{\gamma \alpha}, (\rho(t_k))^\gamma \big\}.
	\end{aligned}
\end{equation}
 Here we recall that $\tilde{d}= (1+\beta)\vee\tilde{r}\vee\tilde{q}$.

With these preparations, we first show that $\{\mathscr M^{x,\tau}_{k}\}_{k\in\mathbb N}$, as defined in \eqref{mar_eq}, is a martingale.

\begin{prop}\label{prop_martingaleIntegrability}
Let Assumptions \ref{a1}--\ref{a3} hold.   
 Suppose that these exist constants $p\in [1, r \wedge q]$ and $\gamma\in(0, 1]$  such that $\frac{p\tilde{q}}{2}+\tilde{d}\gamma\leq q$.
 If the step-size sequence $\tau$ satisfies  $\sum_{i=1}^{\infty} \tau_i^{1+\gamma \alpha}  <\infty$ and $ \sup_{k\in\N} \sum_{i=k}^{\infty} \tau_i   (\rho(t_i- t_k))^\gamma  <\infty$, 
then, for any $x\in E$ and $f\in\mathcal C_{p,\gamma}$,  the sequence $\{\mathscr M^{x,\tau}_{k}\}_{k\in\mathbb N}$ is an $\{\mathcal F_{t_k}\}_{k\in \mathbb N}$-adapted martingale with  $\mathscr M^{x,\tau}_{0}=0$. 
\end{prop}

\begin{proof}
	The fact $\mathscr M^{x,\tau}_{0}=0$ is obvious. 
Combining \eqref{equ_PtkApprmuFromtm} and \eqref{equ_PtkApprmu}  with the expression of $\mathscr M^{x,\tau}_{k}$, we prove that
	\begin{align*}
|\mathscr M^{x,\tau}_{k}|
&\leq \Big| \sum_{i=0}^{k-1} \tau_i \big(f(Y_{t_i}^{x,\tau})-\mu(f)\big)\Big| 
 + \sum_{i=k}^{\infty} \tau_i \Big| P^\tau_{t_k,t_i}f(Y_{t_k}^{x,\tau})-\mu(f)\Big| + \sum_{i=0}^{\infty}\tau_i \Big| P^\tau_{t_i}f(x)-\mu(f)\Big| \\
	&\leq \sum_{i=0}^{k-1} \tau_i  \big|f(Y_{t_i}^{x,\tau})-\mu(f)\big| \\
	&\quad+K\|f\|_{p,\gamma}(1+\|Y_{t_k}^{x,\tau}\|^{\frac{p\tilde{q}}{2}+\tilde{d}\gamma})\sum_{i=k}^{\infty}\big(\tau_i  \max\big\{ \tau_i^{\gamma \alpha}, (\rho(t_i-t_k))^\gamma \big\} \big)\\
		&\quad +K\|f\|_{p,\gamma}(1+\|x\|^{\frac{p\tilde{q}}{2}+\tilde{d}\gamma} ) 
		\sum_{i=0}^{\infty} \big(\tau_i   \max\big\{ \tau_i^{\gamma \alpha}, (\rho(t_i))^\gamma \big\} \big).
	\end{align*}
	Using the conditions  $\frac{p\tilde{q}}{2}+\tilde{d}\gamma \leq q$,   $\sum_{i=0}^{\infty} \tau_i^{1+\gamma \alpha}  <\infty$, 
	$\sum_{i=k}^{\infty} \tau_i   (\rho (t_i- t_k))^\gamma  <\infty$,	
	and  Assumption \ref{a2} leads to
	\begin{align*}
		\E\Big[ \big|\mathscr M^{x,\tau}_{k} \big| \Big]
		&\leq \sum^{k-1}_{i=0} \tau_i \|f\|_{p,\gamma}(1+\E[\|Y^{x,\tau}_{t_i}\|^{\frac p2}]) \\
		&\quad + K\|f\|_{p,\gamma} \Big(1+\E [\|Y_{t_k}^{x,\tau}\|^{\frac{p\tilde{q}}{2}+\tilde{d}\gamma}] \Big)\sum_{i=k}^{\infty}\big(\tau_i  \max\big\{ \tau_i^{\gamma \alpha}, (\rho(t_i-t_k))^\gamma \big\} \big)\\
		&\quad +K\|f\|_{p,\gamma}\Big(1+\|x\|^{\frac{p\tilde{q}}{2}+\tilde{d}\gamma} \Big) 
		\sum_{i=0}^{\infty} \big(\tau_i   \max\big\{ \tau_i^{\gamma \alpha}, (\rho(t_i))^\gamma \big\} \big) \\
		&\leq Kt_{k-1}\|f\|_{p,\gamma}\Big(1+ \|x\|^{\frac{p\tilde{q}^2}{2}+\tilde{q}\tilde{d}\gamma}\Big)<\infty.
	\end{align*}
	For any $s\in \N$ such that $s\leq k$, we see that
	\begin{align*} 
	&\quad 	\E \Big[ \mathscr M^{x,\tau}_{k} \Big\vert \mathcal{F}_{t_s} \Big] \\ 
	&=   \E \bigg[  \sum_{i=0}^{\infty} \tau_i \Big(\E\big[f(Y_{t_i}^{x,\tau})|\mathcal F_{t_{k}}\big] -\mu(f)\Big)-\sum_{i=0}^{\infty}\tau_i \Big(\E\big[f(Y_{t_i}^{x,\tau})|\mathcal F_{0}\big]- \mu(f)\Big) \bigg\vert  \mathcal{F}_{t_s} \bigg] \\
		&=   \sum_{i=0}^{\infty} \tau_i \Big(\E\big[f(Y_{t_i}^{x,\tau})|\mathcal F_{t_{s}}\big]-\mu(f)\Big)-\sum_{i=0}^{\infty}\tau_i \Big(\E\big[f(Y_{t_i}^{x,\tau})|\mathcal F_{0}\big]-\mu(f)\Big) = \mathscr M^{x,\tau}_{s},
	\end{align*}
 which finishes the proof.
\end{proof}
Recalling that $\{\tilde{\mathscr M}^{x,\tau}_{k}\}_{k\in\N}$ is a subsequence of $\{\mathscr M^{x,\tau}_{k}\}_{k\in\N}$, \Cref{prop_martingaleIntegrability} implies that $\{\tilde{\mathscr M}^{x,\tau}_{k}\}_{k\in\N}$ is  a $\{\mathcal F_{t_{n_{(k)}}}\}_{k\in\N}$-adapted martingale.  Since $\{\mathscr M^{x,\tau}_{k}\}_{k\in\N}$ and $\{\tilde{\mathscr M}^{x,\tau}_{k}\}_{k\in\N}$ are martingales, we  define  the  corresponding martingale difference sequences
by $\mathcal Z^{x,\tau}_{0}= \tilde{\mathcal Z}^{x,\tau}_{0}=0$,  and 
\begin{align*}
	\mathcal Z^{x,\tau}_{k}:=\mathscr M^{x,\tau}_{k}-\mathscr M^{x,\tau}_{k-1}, \qquad 
	\tilde{\mathcal Z}^{x,\tau}_{k}:=\tilde{\mathscr M}^{x,\tau}_{k}-\tilde{\mathscr M}^{x,\tau}_{k-1}  
\end{align*}
for all $k\in\N^+$, respectively. 
The following proposition gives the moment boundness of  the  martingale difference sequences $\{\mathcal{Z}^{x,\tau}_k\}_{k\in\N}$ and $\{\tilde{\mathcal Z}^{x,\tau}_{k}\}_{k\in\N}$. 
\begin{prop}\label{prop_ZkIntegrable}
Let Assumptions \ref{a1}--\ref{a3} hold and $c\geq 1$. Suppose that there exist  $p\in[1, r \wedge q]$ and $\gamma\in (0 ,1]$ such that 
$c\big(\frac{p\tilde{q}}{2}+\tilde{d}\gamma\big)\leq q$.  
If the step-size sequence $\tau$ satisfies  $\sum_{i=1}^{\infty} \tau_i^{1+\gamma \alpha}  <\infty$ and $ \sup_{k\in\N} \sum_{i=k}^{\infty} \tau_i   (\rho(t_i- t_k))^\gamma  <\infty$, then for  $f\in \mathcal{C}_{p,\gamma}$, we have: 
	\begin{align} \label{equ_Momentboundedness1}
			\sup_{k\geq 0}\E[|\mathcal Z^{x,\tau}_{k}|^c]
		&\leq  K\|f\|^c_{p,\gamma} \Big(1+ \|x\|^{c(\frac{p\tilde{q}^2}{2}+\tilde{q}\tilde{d}\gamma)} \Big).
	\end{align}
Moreover, for  $n,m\in\N$ with $n<m$ such that $\sum^{m-1}_{k=n}  \tau_{k-1} \leq \tilde{K}$ with a positive constant $\tilde{K}$ independent of $n$ and $m$,  we have
	 \begin{align} \label{equ_Momentboundedness2}
	 	\E \Big[\big|\sum^{m-1}_{k=n} \mathcal Z^{x,\tau}_{k} \big|^c \Big]
	 	&\leq  K \tilde K^c \|f\|^c_{p,\gamma} \Big(1+  \|x\|^{c(\frac{p\tilde{q}^2}{2}+\tilde{q}\tilde{d}\gamma )} \Big).
	 \end{align}
\end{prop}
\begin{rem} \label{rem_momentboundedness}
	Note that 
	$ \tilde{\mathcal Z}^{x,\tau}_{k}=\sum^{n_{(k)}}_{j=n_{(k-1)}+1} 	\mathcal Z^{x,\tau}_{j}$.
	Using \eqref{equ_tildetOneStepBound}, we obtain
	\begin{align*}
		\sup_{N\in\N} \sum^{n_{(N)}}_{n_{(N-1)}+1} \tau_{k-1} 
		=	\sup_{N\in\N} | t_{n_{(N)}-1}-t_{n_{(N-1)}-1}| \leq 1+2\bar{\tau},
	\end{align*}
	which combining with Proposition \ref{prop_ZkIntegrable} shows that 
	 \begin{align} \label{equ_MomentboundednessTildeZ}
		\E \big[\big|\tilde{\mathcal Z}^{x,\tau}_{k} \big|^c \big]
		&\leq  K (1+2\bar{\tau})^c \|f\|^c_{p,\gamma} \Big(1+  \|x\|^{c(\frac{p\tilde{q}^2}{2}+\tilde{q}\tilde{d}\gamma )} \Big).
	\end{align}
\end{rem}
\begin{proof}[Proof of \Cref{prop_ZkIntegrable}]
First, we give the proof of \eqref{equ_Momentboundedness1}. 
	We split the martingale difference sequence $\{\mathcal{Z}^{x, \tau}_k\}_{k\in\N}$ as
		\begin{align} \label{equ_splitOfZ}
			\mathcal{Z}^{x, \tau}_k  
			&= \sum_{i=0}^{\infty} \tau_i \Big(\E\big[f(Y_{t_i}^{x,\tau})|\mathcal F_{t_{k}}\big] -\mu(f)\Big) -\sum_{i=0}^{\infty} \tau_i \Big(\E\big[f(Y_{t_i}^{x,\tau})|\mathcal F_{t_{k-1}}\big]-\mu(f)\Big) \nn \\
			&= \tau_{k-1} \big( f(Y_{t_{k-1}}^{x, \tau})-\mu(f)\big) +  \sum_{i=k}^{\infty} \tau_i \Big(\E\big[f(Y_{t_i}^{x,\tau})|\mathcal F_{t_{k}}\big] -\mu(f) \Big) \nn \\
			&\quad -  \sum_{i=k-1}^{\infty} \tau_i \Big(\E\big[f(Y_{t_i}^{x,\tau})|\mathcal F_{t_{k-1}}\big] -\mu(f)\Big).
		\end{align}
	It follows from 
 \eqref{equ_PtkApprmuFromtm}, 	\eqref{equ_PtkApprmu}, and the conditions of $\tau$ that
		\begin{align}\label{equ_integrableZk}
			|\mathcal Z^{x,\tau}_{k}| &\leq\tau_{k-1} \big|f(Y_{t_{k-1}}^{x,\tau})-\mu(f)\big|+\sum_{i=k}^{\infty}\tau_i \big|P^{\tau}_{t_k, t_i}f(Y_{t_k}^{x,\tau})-\mu(f)\big| \nn \\
			&\quad +\sum_{i=k-1}^{\infty}\tau_i  \big|P^{\tau}_{t_{k-1}, t_i}f(Y_{t_{k-1}}^{x,\tau})-\mu(f)\big| \nn\\
			&\leq K\|f\|_{p,\gamma}\tau_{k-1}\big(1+\|Y_{t_{k-1}}^{x,\tau}\|^{\frac{p}{2}}\big) \nn \\
			&\quad +  K\|f\|_{p,\gamma} \Big(1+ \|Y_{t_k}^{x,\tau}\|^{\frac{p\tilde{q}}{2}+\tilde{d}\gamma} \Big)\sum_{i=k}^{\infty}\big(\tau_i  \max\big\{ \tau_i^{\gamma \alpha}, (\rho(t_i-t_k))^\gamma \big\} \big) \nn \\
			&\quad +K\|f\|_{p,\gamma} \Big(1+ \|Y_{t_{k-1}}^{x,\tau}\|^{\frac{p\tilde{q}}{2}+\tilde{d}\gamma} \Big)\sum_{i=k-1}^{\infty}\big(\tau_i  \max\big\{ \tau_i^{\gamma \alpha}, (\rho(t_i-t_{k-1}))^\gamma \big\} \big) \nn \\
			&\leq  K\|f\|_{p,\gamma} \Big(1+ \|Y_{t_{k-1}}^{x,\tau}\|^{\frac{p}{2}}+  \|Y_{t_{k-1}}^{x,\tau}\|^{\frac{p\tilde{q}}{2}+\tilde{d}\gamma} +\|Y_{t_{k}}^{x,\tau}\|^{\frac{p\tilde{q}}{2}+\tilde{d}\gamma} \Big).
		\end{align}
	Combining \eqref{equ_integrableZk} with
	$c\big(\frac{p\tilde{q}}{2}+\tilde{d}\gamma\big)\leq q$ and  Assumption \ref{a2}, we obtain 
	\begin{align*}
		\E\big[	|\mathcal Z^{x,\tau}_{k}|^c\big] &\leq   K\|f\|^c_{p,\gamma} \big(1+ \|x\|^{\frac{cp\tilde{q}}{2}}+  \|x\|^{c(\frac{p\tilde{q}^2}{2}+\tilde{q}\tilde{d}\gamma )} \big) \\
		&\leq   K\|f\|^c_{p,\gamma} \big(1+ \|x\|^{c(\frac{p\tilde{q}^2}{2}+\tilde{q}\tilde{d}\gamma )} \big),
	\end{align*}
	where we use the fact that $\frac p2 \leq \frac{p\tilde{q}}{2}+\tilde{d}\gamma$.  
	
Next we give the proof of \eqref{equ_Momentboundedness2}, which shows that under the condition $\sum^{m-1}_{k=n} \tau_{k-1} \leq \tilde K$, the partial sum $\sum^{m-1}_{k=n} \mathcal Z^{x,\tau}_k$ is also integrable. 
	Using \eqref{equ_splitOfZ}, we write out the sum as follows:
	\begin{align*}
		\sum^{m-1}_{k=n} \mathcal Z^{x,\tau}_k 
		&= \sum^{m-1}_{k=n}   \tau_{k-1} \big( f(Y_{t_{k-1}}^{x, \tau})-\mu(f)\big) \\ 
		&\quad +\sum_{i=m-1}^{\infty} \tau_i \Big(\E\big[f(Y_{t_i}^{x,\tau})|\mathcal F_{t_{m-1}}\big] -\mu(f) \Big) 	 -\sum_{i=n-1}^{\infty} \tau_i \Big(\E\big[f(Y_{t_i}^{x,\tau})|\mathcal F_{t_{n-1}}\big] -\mu(f)\Big). 
	\end{align*}
	Hence, as in \eqref{equ_integrableZk}, we obtain the estimate of the absolute value as 
		\begin{align*}
			\big|\sum^{m-1}_{k=n} \mathcal Z^{x,\tau}_k \big| 
			&\leq K\|f\|_{p,\gamma}    \Big( \Big[ \sum^{m-1}_{k=n}  \tau_{k-1} \big(1+\|Y_{t_{k-1}}^{x,\tau}\|^{\frac{p}{2}}\big) \Big] + 1+  \|Y_{t_{m-1}}^{x,\tau}\|^{\frac{p\tilde{q}}{2}+\tilde{d}\gamma} +\|Y_{t_{n-1}}^{x,\tau}\|^{\frac{p\tilde{q}}{2}+\tilde{d}\gamma} \Big).
		\end{align*}
	For the case  $c=1$, it holds that $ \E[		|\sum^{m-1}_{k=n} \mathcal Z^{x,\tau}_k |] \leq K \tilde K \|f\|_{p,\gamma}(1+  \|x\|^{\frac{p\tilde{q}^2}{2}+\tilde{q}\tilde{d}\gamma} )$.
	For the case  $c>1$, we use the H\"{o}lder inequality to obtain 
		\begin{align} \label{equ_momentofsumZ}
			\E \Big[ \Big( \sum^{m-1}_{k=n}  \tau_{k-1} \big(1+\|Y_{t_{k-1}}^{x,\tau}\|^{\frac{p}{2}}\big) \Big)^c \Big] 
			&\leq   \Big( \sum^{m-1}_{k=n}  \tau_{k-1} \Big)^{c-1} \E \Big[ \sum^{m-1}_{k=n}  \tau_{k-1} (1+\| Y_{t_{k-1}}^{x,\tau}\|^{\frac{cp}{2}})\Big] \nn \\
			&\leq \Big( \sum^{m-1}_{k=n}  \tau_{k-1} \Big)^{c} \big(1+\sup_{k\geq 0}\E[\| Y_{t_{k}}^{x,\tau}\|^{\frac{cp}{2}}] \big) 
			\leq \tilde K^c\big(1+\|x\|^{\frac{cp\tilde{q}}{2}}\big),
		\end{align}
	where we use the condition that $\sum^{m-1}_{k=n}  \tau_{k-1} \leq \tilde K$. 
	Therefore, 
	using \eqref{equ_momentofsumZ}, we see that for any $n,m\in\N$ such that $\sum^{m-1}_{k=n}  \tau_{k-1} \leq \tilde K$ and $c>1$, 
		\begin{align*} 
			&\quad \E\big[		|\sum^{m-1}_{k=n} \mathcal Z^{x,\tau}_k | ^c\big] \\
			&\leq  K\|f\|^c_{p,\gamma}  \Big(1+ \E \Big[ \Big( \sum^{m-1}_{k=n}  \tau_{k-1} \big(1+\|Y_{t_{k-1}}^{x,\tau}\|^{\frac{p}{2}}\big) \Big)^c \Big] + \E\big[  \|Y_{t_{m-1}}^{x,\tau}\|^{c(\frac{p\tilde{q}}{2}+\tilde{d}\gamma)} \big] +\E\big[ \|Y_{t_{n-1}}^{x,\tau}\|^{c(\frac{p\tilde{q}}{2}+\tilde{d}\gamma)} \big]\Big) \\
			& \leq   K \tilde K^c \|f\|^c_{p,\gamma} \Big(1+  \|x\|^{c(\frac{p\tilde{q}^2}{2}+\tilde{q}\tilde{d}\gamma)} \Big). 
		\end{align*}
	The proof is complete. 
\end{proof}
In the following proposition, building on the approximation of the invariant measure $\mu$ 
and the moment boundedness of the martingale differences in \Cref{prop_ZkIntegrable}, we give the $L^2$-limit of the martingale difference sequence $\{\tilde{\mathcal Z}^{x,\tau}_{k} \}_{k\in\N}$.
\begin{prop}\label{prop_ZkSquareMoment}
Let Assumptions \ref{a1}--\ref{a3} hold. Suppose that there exist constants $p\in [1,r \wedge q]$ and $\gamma\in[\gamma_1 ,1]$ such that $\frac p2(1+\tilde{r})+(1+\beta)\gamma_1\leq r$, $p\tilde{q}+2\tilde{d}\gamma\leq q$, $\frac p2+\frac{p\tilde{q}}{2}+ \gamma(\tilde{d}  \vee \kappa)\leq  r \wedge q$ and $\frac p2+ ([(\frac{p}{2}+\gamma )(\tilde{r} \vee \tilde{q})]\vee [( \frac{p}{2} + l \gamma)\tilde{r}])\leq r$. 
  If the step-size sequence $\tau$  satisfies conditions (\romannumeral 1)--(\romannumeral 4) in Theorem \ref{thm_numerLIL}, 
 then for $f\in \mathcal{C}_{p,\gamma}$, we have
 \begin{align} \label{equ_SquareZLimit}
 \lim_{n\to \infty}\frac{1}{\tilde{t}_N}\sum^N_{k=1} \E  \Big[\big| \tilde{\mathcal Z}^{x,\tau}_{k} \big|^2\Big] = v^2.
 \end{align}
\end{prop}
\begin{proof}
Using the property of  martingale differences that 
\begin{align} \label{equ_orthnormalZ}
\E [\mathcal{Z}^{x,\tau}_k \mathcal{Z}^{x,\tau}_j ]= \E [ \E[\mathcal{Z}^{x,\tau}_k \mathcal{Z}^{x,\tau}_j | \mathcal{F}_{j-1}]]=0, \quad  k<j,
\end{align} 
we obtain $\E [\mathcal{Z}^{x,\tau}_k \mathcal{Z}^{x,\tau}_j ]=0$ 
 for $k\neq j$. 
Since $\tilde{\mathcal{Z}}^{x,\tau}_{k}= \sum^{n_{(k)}}_{j=n_{(k-1)}+1} \mathcal{Z}^{x,\tau}_{j}$,
it follows that  $\sum^N_{k=1} \E  \big[\big| \tilde{\mathcal Z}^{x,\tau}_{k} \big|^2\big] = \sum^{n_{(N)}}_{k=1} \E  \big[\big| \mathcal Z^{x,\tau}_{k} \big|^2\big]$. Hence, to give the $L^2$-limit for  $\{ \tilde{\mathcal Z}^{x,\tau}_{k}\}_{k\in\N}$ as in \eqref{equ_SquareZLimit}, 
it suffices to show that 
\begin{align} \label{equ_SquareGeneralZLimit}
	 \lim_{n\to \infty} \frac{1}{t_n} \sum^n_{k=1} \E  \Big[\big| \mathcal Z^{x,\tau}_{k} \big|^2\Big]  = v^2.
\end{align}
	Based on \eqref{equ_splitOfZ}, 
	we obtain
	\begin{align} \label{equ_GuExpression}
	\E\Big[\big| \mathcal{Z}^{x,\tau}_{k} \big|^2 \Big\vert \mathcal{F}_{t_{k-1}}\Big] 
		&= -\tau_{k-1}^2|f(Y_{t_{k-1}}^{x,\tau})-\mu(f)|^2 +	\E \Big[ \Big| \sum^\infty_{i=k} \tau_i \big( \E \big[ f(Y_{t_{i}}^{x,\tau}) \vert \mathcal{F}_{t_k} \big] -\mu(f)\big)\Big|^2  \Big\vert \mathcal{F}_{t_{k-1}} \Big]  \nn\\
		&\quad - \Big| \sum^\infty_{i=k-1} \tau_i \big( \E \big[ f(Y_{t_{i}}^{x,\tau}) \big\vert \mathcal{F}_{t_{k-1}} \big] -\mu(f)\big)\Big|^2 \nn \\
		&\quad + 2\Big[\tau_{k-1} \big( f(Y_{t_{k-1}}^{x,\tau})-\mu(f)\big) \Big]   \Big[ \sum^\infty_{i=k-1} \tau_i \big( \E \big[ f(Y_{t_{i}}^{x,\tau}) \big\vert \mathcal{F}_{t_{k-1}} \big] -\mu(f)\big) \Big],
	\end{align}
which yields that 
		\begin{align*} 
	\E \Big[	\big|\mathcal Z^{x,\tau}_{k} \big|^2 \Big] &= \E \Big[ \E \Big[	\big|\mathcal Z^{x,\tau}_{k} \big|^2 \Big\vert  \mathcal{F}_{t_{k-1}}  \Big]	\Big] \\
	&=   -\tau_{k-1}^2 \E \Big[ \big| f(Y_{t_{k-1}}^{x,\tau})-\mu(f) \big|^2 \Big] +	\E \Big[ \Big| \sum^\infty_{i=k} \tau_i \big( \E \big[ f(Y_{t_{i}}^{x,\tau}) \vert \mathcal{F}_{t_k} \big] -\mu(f)\big)\Big|^2 \Big] \\
	&\quad - 	\E \Big[ \Big| \sum^\infty_{i=k-1} \tau_i \big( \E \big[ f(Y_{t_{i}}^{x,\tau}) \big\vert \mathcal{F}_{t_{k-1}}  \big] -\mu(f)\big)\Big|^2 \Big] \\
	&\quad + 2\tau_{k-1} \E\Big[ \big( f(Y_{t_{k-1}}^{x,\tau})-\mu(f)\big)  \Big( \sum^\infty_{i=k-1} \tau_i \big( \E \big[ f(Y_{t_{i}}^{x,\tau}) \big\vert \mathcal{F}_{t_{k-1}}  \big] -\mu(f)\big) \Big) \Big]. 
		\end{align*}
Therefore,
	\begin{align*} 
 &\quad \Big(\sum^n_{k=1} \tau_k\Big)^{-1} \Big( \sum^n_{k=1} \E \big[ |\mathcal Z^{x,\tau}_{k}|^2 \big]\Big) \\
 &=
  - \Big(\sum^n_{k=1} \tau_k\Big)^{-1} \Big( \sum^{n}_{k=1} \tau_{k-1}^2 \E \Big[ \big| f(Y_{t_{k-1}}^{x,\tau})-\mu(f) \big|^2 \Big] \Big) \\
  &\quad +\Big(\sum^n_{k=1} \tau_k\Big)^{-1}  \sum^{n}_{k=1}\Big(	\E \Big[ \Big| \sum^\infty_{i=k} \tau_i \big( P^\tau_{t_k, t_i}  f(Y_{t_{k}}^{x,\tau})-\mu(f)\big)\Big|^2 \Big] 
- 	\E \Big[ \Big| \sum^\infty_{i=k-1} \tau_i \big( P^\tau_{t_{k-1}, t_i}  f(Y_{t_{k-1}}^{x,\tau}) -\mu(f)\big)\Big|^2 \Big]  \Big)\\
 &\quad + 2 \Big(\sum^n_{k=1} \tau_k\Big)^{-1} \sum^{n}_{k=1} \Big( \tau_{k-1} \E\Big[ \big( f(Y_{t_{k-1}}^{x,\tau})-\mu(f)\big)  \big[ \sum^\infty_{i=k-1} \tau_i \big( P^\tau_{t_{k-1}, t_i}  f(Y_{t_{k-1}}^{x,\tau})-\mu(f)\big) \big] \Big] \Big) \\
 &=: -\uppercase\expandafter{\romannumeral1}_n+ \uppercase\expandafter{\romannumeral2}_n+\uppercase\expandafter{\romannumeral3}_n.
	\end{align*}
In the following, we split the proof into three steps.

{\textit {Step 1: Proving  $\lim_{n\to \infty}\uppercase\expandafter{\romannumeral1}_n = 0$.}}
For $\uppercase\expandafter{\romannumeral1}_n$, we have
	\begin{align}  \label{equ_estimateOfMainPart1}
		|	\uppercase\expandafter{\romannumeral1}_n |
		&=  \Big(\sum^n_{k=1} \tau_k\Big)^{-1} \Big( \sum^{n}_{k=1} \tau_{k-1}^2 \E \Big[ \big| f(Y_{t_{k-1}}^{x,\tau})-\mu(f) \big|^2 \Big] \Big) \nn \\
		&\leq K \|f\|^2_{p,\gamma} (1+\|x\|^{p\tilde{q}})   \Big(\sum^n_{k=1} \tau_k\Big)^{-1} \Big( \sum^{n}_{k=1} \tau_{k-1}^2  \Big),
	\end{align}
where we use
\begin{align} \label{equ_apprftomu}
 \E \Big[ \big| f(Y_{t_{k-1}}^{x,\tau})-\mu(f) \big|^2 \Big]  &\leq  K \E \Big[\|f\|^2_{p,\gamma} (1+\|Y_{t_{k-1}}^{x,\tau}\|^p) \Big] 
  \leq K \|f\|^2_{p,\gamma} (1+\|x\|^{p\tilde{q}}).
\end{align}
We notice that condition (\romannumeral 1) in \Cref{thm_numerLIL} implies that $\sum^{\infty}_{k=1} \tau_{k-1}^2  \leq  \bar{\tau}^{1-\gamma\alpha } \sum^{\infty}_{k=1} \tau_{k-1}^{1+\gamma\alpha} <\infty$.
Combining this with \eqref{equ_estimateOfMainPart1},  we prove that $\lim_{n\to \infty} |	\uppercase\expandafter{\romannumeral1}_n |=0$.

{\textit {Step 2: Proving  $\lim_{n\to \infty}\uppercase\expandafter{\romannumeral2}_n = 0$.}}
For $\uppercase\expandafter{\romannumeral2}_n$, we see that the summation part in  $\uppercase\expandafter{\romannumeral2}_n$ can be rewritten as
\begin{align*}
	 &\quad \sum^{n}_{k=1}\Big(	\E \Big[ \Big| \sum^\infty_{i=k} \tau_i \big( P^\tau_{t_{k}, t_i}  f(Y_{t_{k}}^{x,\tau})  -\mu(f)\big)\Big|^2 \Big] 
	- 	\E \Big[ \Big| \sum^\infty_{i=k-1} \tau_i \big( P^\tau_{t_{k-1}, t_i}  f(Y_{t_{k-1}}^{x,\tau})  -\mu(f)\big)\Big|^2 \Big]  \Big) \\
	&=  \E \Big[ \Big| \sum^\infty_{i=n} \tau_i \big(P^\tau_{t_{n}, t_i}  f(Y_{t_{n}}^{x,\tau})  -\mu(f)\big)\Big|^2 \Big] - \E \Big[ \Big| \sum^\infty_{i=0} \tau_i \big( P^\tau_{t_i} f(x) -\mu(f)\big)\Big|^2 \Big].
\end{align*}
Using \eqref{equ_PtkApprmuFromtm},  \eqref{equ_PtkApprmu}, and assumptions that $p\tilde{q}+2\tilde{d}\gamma \leq q$, 
	$\sum^\infty_{i=1}  \tau_{i}^{1+\gamma \alpha} <\infty $,  $\sum^\infty_{i=1}  \tau_{i} (\rho(t_{i}))^\gamma<\infty $,  and 
	$\sum^\infty_{i=n}  \tau_{i} (\rho(t_{i}-t_n))^\gamma<\infty$, we obtain 
	\begin{align*}  
	&\quad	\E \Big[ \Big| \sum^\infty_{i=0} \tau_i \big( P^\tau_{t_i} f(x) -\mu(f)\big)\Big|^2 \Big] \\
	&\leq   \E \Big[ \Big| \sum^\infty_{i=1} \big(\tau_i  K  \|f\|_{p,\gamma}  (1+\|x \|^{\frac{p\tilde{q}}{2}+\tilde{d}\gamma})  \max\big\{ \tau_{i}^{\gamma \alpha}, (\rho(t_{i}))^\gamma \big\} \big) \Big|^2 \Big] 
	\leq   K  \|f\|^2_{p,\gamma} \big(1+\|x \|^{p\tilde{q}^2+2\tilde{q}\tilde{d}\gamma} \big)
	\end{align*}
and
	\begin{align} \label{equ_summationOfSquareFromtm}
		&\quad	\E \Big[ \Big| \sum^\infty_{i=n} \tau_i \big( P^\tau_{t_{n}, t_i}  f(Y_{t_{n}}^{x,\tau})  -\mu(f)\big)\Big|^2 \Big] \nn \\
		&\leq   \E \Big[ \Big| \sum^\infty_{i=n} \big(\tau_i  K  \|f\|_{p,\gamma}  (1+\|Y_{t_{n}}^{x,\tau} \|^{\frac{p\tilde{q}}{2}+\tilde{d}\gamma})  \max\big\{ \tau_{i}^{\gamma \alpha}, (\rho(t_{i}-t_n))^\gamma \big\} \big) \Big|^2 \Big] \nn \\
		&\leq   K  \|f\|^2_{p,\gamma}  \big(1+\|x \|^{p\tilde{q}^2+2\tilde{q}\tilde{d}\gamma} \big).
	\end{align}
Hence, 
\begin{align*}
	 | \uppercase\expandafter{\romannumeral2}_n | 
	&\leq \Big(\sum^n_{k=1} \tau_k\Big)^{-1}\Big(   \E \Big[ \Big| \sum^\infty_{i=n} \tau_i \big(P^\tau_{t_{n}, t_i}  f(Y_{t_{n}}^{x,\tau})  -\mu(f)\big)\Big|^2 \Big] +  \E \Big[ \Big| \sum^\infty_{i=0} \tau_i \big( \E \big[ f(Y_{t_{i}}^{x,\tau}) \big] -\mu(f)\big)\Big|^2 \Big] \Big) \\
	&\leq  K  \|f\|^2_{p,\gamma}  (1+\|x \|^{p\tilde{q}^2+2\tilde{q}\tilde{d}\gamma}) \Big(\sum^n_{k=1} \tau_k\Big)^{-1}.
\end{align*}
Since the step-size sequence $\{\tau_k\}_{k\in\N}$ satisfies $\sum^n_{k=1} \tau_k \to \infty$ as $n\to \infty$, it shows that
$ | \uppercase\expandafter{\romannumeral2}_n |  \to 0 $ as $n\to \infty$.

{\textit {Step 3: Proving  $\lim_{n\to \infty}\uppercase\expandafter{\romannumeral3}_n = v^2$.}}
Note that
	\begin{align}  \label{equ_Mainpart3}
&\quad	\bigg|  2\E\Big[ \big( f(Y_{t_{k-1}}^{x,\tau})-\mu(f)\big)  \big[ \sum^\infty_{i=k-1} \tau_i \big( P^\tau_{t_{k-1}, t_i}  f(Y_{t_{k-1}}^{x,\tau})  -\mu(f)\big) \big] \Big] -v^2\bigg|\nn	\\
&\leq  	2\bigg|  \E\Big[ \big( f(Y_{t_{k-1}}^{x,\tau})-\mu(f)\big)  \big[ \sum^\infty_{i=k-1} \tau_i \big( P^\tau_{t_{k-1}, t_i}  f(Y_{t_{k-1}}^{x,\tau})  -\mu(f)\big) \big] \Big]\nn  \\
&\quad \quad -\mu\big((f-\mu(f)) \big[ \sum^\infty_{i=k-1} \tau_i \big( P^\tau_{t_{k-1},t_i} f-\mu(f)\big) \big]\big) \bigg| \nn\\
&\quad + 2\bigg| \mu\big((f-\mu(f)) \big[ \sum^\infty_{i=k-1} \tau_i \big( P^\tau_{t_{k-1},t_i} f-\mu(f)\big) \big]\big) - \mu\big((f-\mu(f))\int_{0}^{\infty}(P_{t}f-\mu(f))\mathrm dt\big)  \bigg|  \nn \\
&=: \I_{1,k}+\I_{2,k}.
	\end{align}

{ \bf{Estimate of $\I_{1,k}$.}}
For $\I_{1,k}$, we define the function 
\begin{align} \label{equ_definitionF}
	F^\tau_{k-1} (u) :=
	 \big( f(u)-\mu(f) \big) \big[ \sum^\infty_{i=k-1} \tau_i \big(P^\tau_{t_{k-1},t_i} f(u )-\mu(f)\big) \big]
\end{align}
for any $u\in E$. We claim that there exists some $\tilde{p}_F\geq 1$ such that $\|F^\tau_{k-1}\|_{\tilde{p}_F,\gamma} \leq K\|f\|^2_{p,\gamma}$. 
To prove this claim, we first apply \eqref{equ_PtkApprmuFromtm} and use the same technique as in \eqref{equ_integrableZk} to obtain
\begin{align*}
	| F^\tau_{k-1} (u)| &\leq  K\|f\|_{p,\gamma}^2 (1+\|u\|^{\frac p2+ \frac{p\tilde{q}}{2}+\tilde{d}\gamma}).
\end{align*}
This implies
\begin{align} \label{equ_FinCpgamma1}
	\frac{	| F^\tau_{k-1} (u)| }{1+\|u\|^{\frac p2+\frac{p\tilde{q}}{2}+\tilde{d}\gamma}} \leq K \|f\|_{p,\gamma}^2.
\end{align}
Using Assumption \ref{a2} (\romannumeral 2), we derive that for any $u_1, u_2 \in E$, $f\in \mathcal{C}_{p,\gamma}$ and $i\geq k-1$, 
	\begin{align}  \label{equ_PtkxApprPtky}
	&\quad \big|P^\tau_{t_{k-1},t_i} f(u_1 )-P^\tau_{t_{k-1},t_i} f(u_2)\big|  \nn\\
	&\leq   K \|f\|_{p,\gamma} \big( 1+\E \big[ \|Y^{u_1, \tau}_{t_{k-1}, t_i}\|^p \big]+\E\big[ \|Y^{u_2, \tau}_{t_{k-1}, t_i}\|^p \big]  \big)^{\frac 12} (\mathbb{W}_2 ( \mu^{u_1,\tau}_{ t_{k-1}, t_i},  \mu^{u_2,\tau}_{ t_{k-1}, t_i}) )^\gamma \nn \\
		&\leq K \|f\|_{p,\gamma} \|u_1-u_2\|^\gamma \big(1+ \|u_1\|^{\frac{p\tilde{q}}{2} +\gamma \kappa }+\|u_2\|^{\frac{p\tilde{q}}{2} +\gamma \kappa }\big) (\rho^\tau (t_i-t_{k-1}))^\gamma,
\end{align}
where we use
\begin{align*}
\mathbb{W}_2 ( \mu^{u_1,\tau}_{ t_{k-1}, t_i},  \mu^{u_2,\tau}_{ t_{k-1}, t_i}) &\leq\big( \E\big[ \|Y^{u_1, \tau}_{t_{k-1}, t_i}-Y^{u_2, \tau}_{t_{k-1}, t_i}\|^2 \big]\big)^{\frac 12} \\
&\leq K\| u_1-u_2\| \big(1+\|u_1\|^\kappa+ \|u_2\|^\kappa \big) \rho^\tau (t_i-t_{k-1}).
\end{align*} 
For $u_1, u_2 \in E$, using \eqref{equ_PtkApprmuFromtm}, \eqref{equ_PtkxApprPtky}, and conditions $\sum^\infty_{i=k-1} \tau_i(\rho(t_i-t_{k-1}))^\gamma <\infty$ and $\sum^\infty_{i=k-1} \tau_i(\rho^\tau(t_i-t_{k-1}))^\gamma <\infty$, we obtain
\begin{align*}  
 | F^\tau_{k-1} (u_1) -  F^\tau_{k-1} (u_2) | 
	&\leq \big| f(u_1)-f(u_2)\big| \Big(  \sum^\infty_{i=k-1} \tau_i \big|P^\tau_{t_{k-1},t_i} f(u_1 )-\mu(f)\big|  \Big) \\
	&\quad +   \big| (f(u_2)-\mu(f))\big| \Big(  \sum^\infty_{i=k-1} \tau_i \big|P^\tau_{t_{k-1},t_i} f(u_1 )-P^\tau_{t_{k-1},t_i} f(u_2)\big| \Big) \\
	&\leq  K \|f\|^2_{p,\gamma} \|u_1-u_2\|^\gamma\big(1+ \|u_1\|^{(\frac p2+\frac{p\tilde{q}}{2}+\tilde{d}\gamma ) \vee (\frac{p\tilde{q}}{2} +\gamma \kappa ) }+\|u_2\|^{\frac p2 +\frac{p\tilde{q}}{2} +\gamma \kappa }\big),
\end{align*}
which proves the claim by setting
$\tilde{p}_F:= \frac p2+\frac{p\tilde{q}}{2}+ \gamma(\tilde{d}  \vee \kappa)$.

With the assumption that $\tilde{p}_F\leq r\wedge q$, we use \eqref{equ_PtkApprmu} and arrive at
	\begin{align} \label{equ_Mainpart3_2}
		\I_{1,k}
		&= 2 \Big| P^\tau_{t_{k-1}}	F^\tau_{k-1} (x )  -\mu (F^\tau_{k-1}  )\Big| \nn\\
		&\leq  K  \|F^\tau_{k-1}\|_{\tilde{p}_F,\gamma}  (1+\|x\|^{\frac{\tilde{p}_F\tilde{q}}{2}+\tilde{d}\gamma})  \max\big\{ \tau_{k-1}^{\gamma \alpha}, (\rho(t_{k-1}))^\gamma \big\} \nn \\
		&\leq K\|f\|^2_{p,\gamma} (1+\|x\|^{\frac{\tilde{p}_F\tilde{q}}{2}+\tilde{d}\gamma})  \max\big\{ \tau_{k-1}^{\gamma \alpha}, (\rho(t_{k-1}))^\gamma \big\}.
	\end{align}
	
{ \bf{Estimate of $\I_{2,k}$.}}
For $\I_{2,k}$, we rewrite it as follows
\begin{align}\label{equ_expressI2k}
	\I_{2,k} &=2\bigg| \mu\big((f-\mu(f)) \big[ \sum^\infty_{i=k-1} \tau_i \big( P^\tau_{t_{k-1},t_i} f-\mu(f)\big) \big]\big) - \mu\big((f-\mu(f))\int_{t_{k-1}}^{\infty}(P_{t_{k-1},t} f-\mu(f))\mathrm dt\big)  \bigg| \nn \\
	&\leq 2 \mu \big(|f-\mu(f)| \sum^\infty_{i=k-1}\big[  \int^{t_{i+1}}_{t_{i}} \big| P^\tau_{t_{k-1},t_i} f-P_{t_{k-1},t_i}f  \big| +\big| P_{t_{k-1},t_i} f-P_{t_{k-1},t}f  \big|  \ \rd t \big] \big).
\end{align}
For any given $u\in E$, we know that  for $p \in [1, r \wedge q]$ and $\gamma \in [\gamma_1,1]$, 
\begin{align} \label{equ_apprPdeltaPfromtm}
	&\quad  \big| P^\tau_{t_{k-1},t_i} f(u)-P_{t_{k-1},t_i}f (u) \big| \nn \\
	 &\leq  \|f\|_{p,\gamma} \big(1+  \E\big[\|Y_{t_{k-1},t_i}^{u,\tau}\|^p\big] + \E\big[\|X_{t_{k-1},t_i}^{u}\|^p\big] \big)^{\frac{1}{2}}(\mathbb W_{2}(\mu_{t_{k-1},t_i}^{u,\tau},\mu_{t_{k-1},t_i}^{u}))^{\gamma}  \nn \\
	 &\leq K  \|f\|_{p,\gamma}  \big(1+\|u\|^{(\frac{p}{2}+\gamma) (\tilde{r} \vee \tilde{q})} \big)  \tau_i^{\gamma \alpha}. 
\end{align}
With Assumption \ref{a1} (\romannumeral 3) and the time-homogeneous Markov property of $\{X^{u}_{t}\}_{t\geq 0}$, we deduce that for any $t\geq t_i$ and $u\in E$, 
\begin{align*}
		\mathbb{W}_2 (\mu_{t_{k-1},t_i}^{u},  \mu_{t_{k-1},t}^{u}) &\leq \big(\E [ \| X^u_{t_{k-1},t_i}-X^u_{t_{k-1},t}\|^2] \big)^{\frac 12} \\
		&=\big(\E \big[ \| X^x_{t_{i},t}-x\|^2 \big|_{x=X^u_{t_{k-1},t_i}} \big]\big)^{\frac 12} \\
		&\leq K(1+\E[\|X^u_{t_{k-1},t_i}\|^l]) (t-t_i)^{\tilde{l}} 
		\leq K(1+\|u\|^{l\tilde{r}}) (t-t_i)^{\tilde{l}}.
\end{align*}
Then we have
	\begin{align} \label{equ_apprPsmalltimechange}
	&\quad  \big| P_{t_{k-1},t} f(u)-P_{t_{k-1},t_i}f (u) \big| \nn \\
	&\leq  \|f\|_{p,\gamma} \big(1+  \E\big[\|X^{u}_{t_{k-1},t}\|^p\big] + \E\big[\|X^{u}_{t_{k-1},t_i}\|^p\big] \big)^{\frac{1}{2}}(\mathbb W_{2}(\mu_{t_{k-1},t}^{u},\mu_{t_{k-1},t_i}^{u}))^{\gamma} \nn \\
	&\leq K   \|f\|_{p,\gamma} \big(1+\|u\|^{(\frac{p}{2}+l\gamma)\tilde{r}} \big)   (t-t_i)^{\tilde{l}\gamma}.  
\end{align}
For simplicity, we denote $\bar{p}:= [(\frac{p}{2}+\gamma) (\tilde{r} \vee \tilde{q}) ]\vee[(\frac{p}{2}+l\gamma)\tilde{r}]$.
By plugging  \eqref{equ_apprPdeltaPfromtm} and \eqref{equ_apprPsmalltimechange} into \eqref{equ_expressI2k}, 
we arrive at
\begin{align} \label{equ_Mainpart3_3}
	\I_{2,k} 
	&\leq   2  \mu \Big(|f-\mu(f)|  \big[ K  \|f\|_{p,\gamma} \big(1+\|\cdot \|^{\bar{p}}  \big)  \big] \sum^\infty_{i=k-1}\big[  \int^{t_{i+1}}_{t_{i}} \big(  \tau_i^{\gamma \alpha}+ (t-t_i)^{\tilde{l}\gamma}  \big)  \rd t \big] \Big)\nn \\
	&\leq    2 \mu \Big(\|f\|_{p,\gamma}(1+\|\cdot\|^{\frac p2})  \big[K  \|f\|_{p,\gamma} \big(1+\|\cdot \|^{\bar{p}}  \big) \big]    \Big( \sum^\infty_{i=k-1}  \tau_{i+1} \tau_{i}^{\gamma \alpha} +  \sum^\infty_{i=k} \tau_i^ {1+\tilde{l}\gamma} \Big)  \Big)\nn \\
	&\leq  K  \|f\|^2_{p,\gamma}   \Big( \sum^\infty_{i=k-1} \tau_{i+1} \tau_{i}^{\gamma \alpha} +  \sum^\infty_{i=k} \tau_i^ {1+\tilde{l}\gamma} \Big),
\end{align}
where  in the last step we use  $\frac p2+\bar{p} \leq r$ and  \eqref{equ_boundOfNormOfmu}.

Combining \eqref{equ_Mainpart3}, \eqref{equ_Mainpart3_2}  and  \eqref{equ_Mainpart3_3}, it holds that 

\begin{align*}
&\quad 	| \uppercase\expandafter{\romannumeral3}_n-v^2 | \\
	&\leq \Big(\sum^n_{k=1} \tau_k\Big)^{-1} \sum^{n}_{k=1} \big[ \tau_{k-1} \big( \I_{1,k}+\I_{2,k} \big) \big] \\
	&\leq K  \|f\|^2_{p,\gamma}  (1+\| x \|^{\frac{\tilde{p}_F\tilde{q}}{2}+\tilde{d}\gamma} )  \\
	&\quad \times \Big(\sum^n_{k=1} \tau_k\Big)^{-1} \Big( \sum^{n}_{k=1}  \tau_{k-1}^{1+\gamma \alpha} +  \sum^{n}_{k=1}  \tau_{k-1}(\rho(t_{k-1}))^\gamma+ \sum^{n}_{k=1} \tau_{k-1}\big( \sum^\infty_{i=k-1} \tau_{i+1} \tau_{i}^{\gamma \alpha} +  \sum^\infty_{i=k} \tau_i^ {1+\tilde{l}\gamma} \big)\Big).
\end{align*}
With assumptions of the step-size sequence $\tau$, we see that 
	$\lim_{n\to \infty} (\sum^n_{k=1} \tau_k)^{-1}  (\sum^{n}_{k=1}  \tau_{k-1}^{1+ \gamma \alpha}) =0$ and 
	$\lim_{n\to \infty} (\sum^n_{k=1} \tau_k)^{-1}  (\sum^{n}_{k=1}  \tau_{k-1}(\rho(t_{k-1}))^\gamma) =0$. 
	Combining with condition (\romannumeral 4) in \Cref{thm_numerLIL},
	we have
$
\lim_{n\to \infty}  	|\uppercase\expandafter{\romannumeral3}_n -v^2| =0.
$

Combining the above estimates, we finish the proof. 
\end{proof}
Next, based on Propositions  \ref{prop_ZkIntegrable} and \ref{prop_ZkSquareMoment}, we prove the almost sure limit of $\{\tilde{\mathcal{Z}}^{x,\tau}_{N} \}_{N\in\N}$ in the following proposition.
\begin{prop} \label{prop_almostsureZk}
	Let Assmptions \ref{a1}--\ref{a3} hold. 
	Suppose that there exists an  integer $c\geq 2$ and constants $p\in[1, r\wedge q]$, $\gamma \in[\gamma_1,1]$ such that $\frac p2(1+\tilde{r})+(1+\beta)\gamma_1\leq r$, 
	$4c(\frac{p\tilde{q}}{2}+\tilde{d}\gamma) \leq q$  and $\frac p2+ ([(\frac{p}{2}+\gamma )(\tilde{r} \vee \tilde{q})]\vee [( \frac{p}{2} + l \gamma)\tilde{r}])\leq r$. 
	Let $\tilde{p}_F:= \frac p2+\frac{p\tilde{q}}{2}+ \gamma(\tilde{d}  \vee \kappa)$ 
	and assume that  $\tilde{p}_F\leq r \wedge q$ and
	$2c (\frac{\tilde{p}_F\tilde{q}}{2}+\tilde{d}\gamma) \leq q$.
	If the step-size sequence $\tau$ satisfies conditions (\romannumeral 1)--(\romannumeral 4) in \Cref{thm_numerLIL}, 
	then, for
	$f\in \mathcal{C}_{p,\gamma}$, we have 
	\begin{align*}
			\lim_{N\to \infty} \frac{1}{\tilde{t}_N} \Big(\sum^{N}_{k=1} |\tilde{\mathcal Z}^{x,\tau}_k|^2\Big) =v^2,\quad \text{a.s.}
	\end{align*}
\end{prop}
\begin{proof}
Recalling that $\{\tilde{\tau}_N\}_{N\in\N}$ is the time-step sequence corresponding to $\{\tilde{t}_N\}_{N\in\N}$ and that $n_{(k)}$ is the index satisfying $\tilde{t}_k=t_{n_{(k)}}$. 
	We first expand $	\E [ | \frac{1}{\tilde{t}_m} \sum^m_{k=1} | \tilde{\mathcal Z}^{x,\tau}_k|^2-v^2|^{2c}]$ for $1\leq m\leq N$ and $c\in\N^+$  and obtain
	\begin{align} \label{equ_Zalmostsure1}
	&\quad \E \Big[ \Big| \frac{1}{\tilde{t}_m} \sum^m_{k=1} \big| \tilde{\mathcal Z}^{x,\tau}_k\big|^2-v^2\Big|^{2c}\Big] \nn\\
	&=\frac{1}{(\tilde{t}_m)^{2c}}\E\Big[\sum^m_{k_{2c}=1}\cdots\sum^m_{k_{2}=1}\sum^m_{k_{1}=1} \Big(\big| \tilde{\mathcal Z}^{x,\tau}_{k_1}\big|^2-\tilde{\tau}_{k_1}v^2 \Big)\cdots \Big(\big| \tilde{\mathcal Z}^{x,\tau}_{k_{2c}}\big|^2-\tilde{\tau}_{k_{2c}}v^2 \Big)  \Big] \nn \\
	&=(2c)! \frac{1}{(\tilde{t}_m)^{2c}}\E\Big[\sum^m_{k_{2c}=1}\sum^{k_{2c}}_{k_{2c-1}=1}\Big( \big(\big|\tilde{\mathcal Z}^{x,\tau}_{k_{2c-1}}\big|^2-\tilde{\tau}_{k_{2c-1}}v^2\big) \E\big[\big(\big|\tilde{\mathcal Z}^{x,\tau}_{k_{2c}} \big|^2-\tilde{\tau}_{k_{2c}}v^2 \big) \big| \mathcal{F}_{t_{n_{(k_{2c-1})}}}\big]\Big) \nn \\
	&\quad \times \Big(\sum^{k_{2c-1}}_{k_{2c-2}=1} \cdots\sum^{k_2}_{k_{1}=1} \big(\big| \tilde{\mathcal Z}^{x,\tau}_{k_1}\big|^2-\tilde{\tau}_{k_1}v^2 \big)\cdots \big(\big| \tilde{\mathcal Z}^{x,\tau}_{k_{2c-2}}\big|^2-\tilde{\tau}_{k_{2c-2}}v^2 \big) \Big) \Big] \nn \\
	&=(2c)(2c-1)  \frac{1}{(\tilde{t}_m)^{2c}}\E\bigg[\sum^{m}_{k_{2c-1}=1} \Big[ \big(
	\big|\tilde{\mathcal Z}^{x,\tau}_{k_{2c-1}}\big|^2-\tilde{\tau}_{k_{2c-1}}v^2\big) \Big( \sum^m_{k_{2c}=k_{2c-1}} \E\big[\big(\big|\tilde{\mathcal Z}^{x,\tau}_{k_{2c}} \big|^2-\tilde{\tau}_{k_{2c}}v^2 \big) \big| \mathcal{F}_{t_{n_{(k_{2c-1})}}}\big]\Big) \nn \\
	&\quad \times \Big| \sum^{k_{2c-1}}_{k=1} \big(|\tilde{\mathcal Z}^{x,\tau}_k|^2-\tilde{\tau}_kv^2 \big)\Big|^{2(c-1)} \Big] \bigg],
	\end{align}
	where we use the properties that $1\leq k_1 \leq k_2 \leq \cdots \leq k_{2c}\leq m$ and the martingale difference $\tilde{\mathcal Z}^{x,\tau}_{k}$ is $\mathcal{F}_{t_{n_{(k)}}}$-measuable. 
	By using the H\"{o}lder inequality for \eqref{equ_Zalmostsure1}, we have 
	\begin{align} \label{equ_tildeZsplit}
	&\quad \E \Big[ \Big| \frac{1}{\tilde{t}_m} \sum^m_{k=1} \big| \tilde{\mathcal Z}^{x,\tau}_k\big|^2-v^2\Big|^{2c}\Big]\nn \\
	&\leq K \frac{1}{(\tilde{t}_m)^{2c}}   \sum^m_{k_{2c-1}=1}\bigg[ \Big(\E\Big[ \big|\big(\big|\tilde{\mathcal Z}^{x,\tau}_{k_{2c-1}}\big|^2-\tilde{\tau}_{k_{2c-1}}v^2\big) \big( \sum^m_{k_{2c}=k_{2c-1}}  \E\big[\big(\big|\tilde{\mathcal Z}^{x,\tau}_{k_{2c}} \big|^2-\tilde{\tau}_{k_{2c}}v^2 \big) \big|\mathcal{F}_{t_{n_{(k_{2c-1})}}} \big] \big)\big|^{c}   \Big] \Big)^{\frac 1c}\nn \\
	&\quad \times
	 \Big(\E\Big[\big| \sum^{k_{2c-1}}_{k=1} \big(|\tilde{\mathcal Z}^{x,\tau}_k|^2-\tilde{\tau}_kv^2 \big)\big|^{2c} \Big]\Big)^{\frac{c-1}{c}} \bigg].
	\end{align}
	Here, the constant $K$ only depends on the order $c$.
	Define
	$$\mathcal{I}^{x,\tau}_c(N):=\sup_{1\leq m\leq N}\E \big[ \big|\sum^{m}_{k=1} \big(|\tilde{\mathcal Z}^{x,\tau}_k|^2-\tilde{\tau}_kv^2 \big) \big|^{2c} \big].$$  By using  \eqref{equ_tildeZsplit}, we obtain
	\begin{align*}
	 &\quad \mathcal{I}^{x,\tau}_c(N) \\
	 &=\sup_{1\leq m\leq N}  \bigg[(\tilde{t}_m)^{2c} \ \E \Big[ \Big| \frac{1}{\tilde{t}_m} \sum^m_{k=1} \big| \tilde{\mathcal Z}^{x,\tau}_k\big|^2-v^2\Big|^{2c}\Big] \bigg] \\
	 &\leq K  \sum^N_{k_{2c-1}=1} \sup_{1\leq m\leq N} \bigg[ \Big(\E\Big[ \big|\big(\big|\tilde{\mathcal Z}^{x,\tau}_{k_{2c-1}}\big|^2-\tilde{\tau}_{k_{2c-1}}v^2\big) \big( \sum^m_{k_{2c}=k_{2c-1}}  \E\big[\big(\big|\tilde{\mathcal Z}^{x,\tau}_{k_{2c}} \big|^2-\tilde{\tau}_{k_{2c}}v^2 \big) \big| \mathcal{F}_{t_{n_{(k_{2c-1})}}}\big] \big)\big|^{c}   \Big] \Big)^{\frac 1c} \\
	 &\quad \times
		\Big(\mathcal{I}^{x,\tau}_c(N)\Big)^{\frac{c-1}{c}} \bigg],
	\end{align*}
	which leads to
	\begin{align*} 
	&\quad \E \Big[ \Big|\sum^{N}_{k=1} \big(|\tilde{\mathcal Z}^{x,\tau}_k|^2-\tilde{\tau}_kv^2 \big) \Big|^{2c} \Big] \leq 	\mathcal{I}^{x,\tau}_c(N) 
	\leq  K \Big[\sum^N_{k_{2c-1}=1} G^{x,\tau,c}_{k_{2c-1}} \Big]^{c}
	\end{align*}
	with 
	\begin{align*}
	G^{x,\tau,c}_{k_{2c-1}} &:= \sup_{1\leq m\leq N}\Big(\E\Big[ \big||\tilde{\mathcal Z}^{x,\tau}_{k_{2c-1}}|^2-\tilde{\tau}_{k_{2c-1}}v^2\big|^{2c} \Big] \\
	&\quad + \E\Big[ \big||\tilde{\mathcal Z}^{x,\tau}_{k_{2c-1}}|^2-\tilde{\tau}_{k_{2c-1}}v^2\big|^c  \big| \sum^m_{k_{2c}=k_{2c-1}+1}  \E\big[\big(|\tilde{\mathcal Z}^{x,\tau}_{k_{2c}} |^2-\tilde{\tau}_{k_{2c}}v^2 \big) \big| \mathcal{F}_{t_{n_{(k_{2c-1})}}}\big] \big|^{c}   \Big] \Big)^{\frac 1c}.
	\end{align*}
	Here we claim that the expression on the right-hand side grows as $(\tilde{t}_N)^c$, which suffices to show that $G^{x,\tau,c}_{k_{2c-1}}$ is uniformly bounded. 
 We divide the proof of this claim into three steps.
 
 {\textit{Step 1:  Estimate of   $\, \sum^m_{k_{2c}=k_{2c-1}+1}  \E\big[\big(\big|\tilde{\mathcal Z}^{x,\tau}_{k_{2c}} \big|^2-\tilde{\tau}_{k_{2c}}v^2 \big) \big| \mathcal{F}_{t_{n_{(k_{2c-1})}}}\big] $.}}
By applying \eqref{equ_GuExpression} together with the tower property of conditional expectation, for $k\geq n_{ (k_{2c-1})}+1$, we express $\E[| \mathcal Z^{x,\tau}_{k}|^2 | \mathcal{F}_{t_{n_{(k_{2c-1})}}}] $  as
		\begin{align} \label{equ_ZkExpressionCondition}
			&\quad \E \Big[	\big|\mathcal Z^{x,\tau}_{k} \big|^2 \Big|\mathcal{F}_{t_{n_{(k_{2c-1})}}} \Big] \nn\\
			&=   -\tau_{k-1}^2 \E \Big[ \big| f(Y_{t_{k-1}}^{x,\tau})-\mu(f) \big|^2 \Big|\mathcal{F}_{t_{n_{(k_{2c-1})}}}	\Big] +	\E \Big[ \Big| \sum^\infty_{i=k} \tau_i \big( \E \big[ f(Y_{t_{i}}^{x,\tau}) \vert \mathcal{F}_{t_k} \big] -\mu(f)\big)\Big|^2 \Big| \mathcal{F}_{t_{n_{(k_{2c-1})}}}	\Big]  \nn \\
			&\quad - 	\E \Big[ \Big| \sum^\infty_{i=k-1} \tau_i \big( \E \big[ f(Y_{t_{i}}^{x,\tau}) \big\vert \mathcal{F}_{t_{k-1}}  \big] -\mu(f)\big)\Big|^2\Big|\mathcal{F}_{t_{n_{(k_{2c-1})}}}	 \Big]\nn \\
			&\quad + 2\tau_{k-1} \E\Big[ \big( f(Y_{t_{k-1}}^{x,\tau})-\mu(f)\big)  \Big( \sum^\infty_{i=k-1} \tau_i \big( \E \big[ f(Y_{t_{i}}^{x,\tau}) \big\vert \mathcal{F}_{t_{k-1}}  \big] -\mu(f)\big) \Big) \Big|\mathcal{F}_{t_{n_{(k_{2c-1})}}}	\Big]. 
		\end{align}
	Using \eqref{equ_ZkExpressionCondition}, for each fixed $k_{2c-1}$, we rewrite the sum as
		\begin{align} \label{equ_splitofZconditional}
		&\quad 	\sum^m_{k_{2c}=k_{2c-1}+1}  \E\big[\big(\big|\tilde{\mathcal Z}^{x,\tau}_{k_{2c}} \big|^2-\tilde{\tau}_{k_{2c}}v^2 \big) \big| \mathcal{F}_{t_{n_{(k_{2c-1})}}} \big] \nn\\
		&= - \bigg(\sum^{n_{(m)}}_{k=n_{(k_{2c-1})}+1} \tau_{k-1}^2 \E \Big[ \big| f(Y_{t_{k-1}}^{x,\tau})-\mu(f) \big|^2 \Big|\mathcal{F}_{t_{n_{(k_{2c-1})}}}	\Big] \bigg) \nn\\
		&\quad+ \bigg(	\E \Big[ \Big| \sum^\infty_{i=n_{(m)}} \tau_i \big( \E \big[ f(Y_{t_{i}}^{x,\tau}) \vert \mathcal{F}_{t_{n_{(m)}}} \big] -\mu(f)\big)\Big|^2 \Big|\mathcal{F}_{t_{n_{(k_{2c-1})}}}	\Big] \nn\\
		&\quad - 	\E \Big[ \Big| \sum^\infty_{i=n_{(k_{2c-1})}} \tau_i \big( \E \big[ f(Y_{t_{i}}^{x,\tau}) \big\vert \mathcal{F}_{t_{n_{(k_{2c-1})}} } \big] -\mu(f)\big)\Big|^2\Big|\mathcal{F}_{t_{n_{(k_{2c-1})}}} \Big] \bigg) \nn\\
		&\quad +  \sum^{n_{(m)}}_{k=n_{(k_{2c-1})}+1}  \tau_{k-1}\bigg(2 \E\Big[ \big( f(Y_{t_{k-1}}^{x,\tau})-\mu(f)\big)  \Big( \sum^\infty_{i=k-1} \tau_i \big( \E \big[ f(Y_{t_{i}}^{x,\tau}) \big\vert \mathcal{F}_{t_{k-1}}  \big] -\mu(f)\big) \Big) \Big| \mathcal{F}_{t_{n_{(k_{2c-1})}}}	\Big] -v^2 \bigg) \nn\\
		&\quad +  (\tau_{n_{(k_{2c-1})}}- \tau_{n_{(m)}}  ) v^2 \nn\\
		&=: -\uppercase\expandafter{\romannumeral1}_{k_{2c-1},m }+ \uppercase\expandafter{\romannumeral2}_{k_{2c-1},m }+\uppercase\expandafter{\romannumeral3}_{k_{2c-1},m }+  (\tau_{n_{(k_{2c-1})}}- \tau_{n_{(m)}}  ) v^2.
		\end{align}
	
	{\bf{Estimate of $\uppercase\expandafter{\romannumeral1}_{k_{2c-1},m}.$}}
	It follows from \eqref{equ_apprftomu} that  
	\begin{align} \label{equ_I1m}
		|\uppercase\expandafter{\romannumeral1}_{k_{2c-1},m }| &\leq K\|f\|^2_{p,\gamma} \sum^{n_{(m)}}_{k=n_{(k_{2c-1})}+1} \tau_{k-1}^2  \big(1+\E \big[ \| Y_{t_{k-1}}^{x,\tau}\|^p \big|\mathcal{F}_{t_{n_{(k_{2c-1})}}}	\big] \big).
	\end{align}
	
		{\bf{Estimate of $\uppercase\expandafter{\romannumeral2}_{k_{2c-1},m}.$}}
	In the same manner as \eqref{equ_summationOfSquareFromtm}, we derive  
		\begin{align*}
		&\quad	\E \Big[ \Big| \sum^\infty_{i=n_{(m)}} \tau_i \big( P^\tau_{t_{n_{(m)}}, t_i}  f(Y_{t_{n_{(m)}}}^{x,\tau})  -\mu(f)\big)\Big|^2  \Big| \mathcal{F}_{t_{n_{(k_{2c-1})}}}\Big] \\
		&\leq   \E \Big[ \Big| \sum^\infty_{i=n_{(m)}} \big(\tau_i  K  \|f\|_{p,\gamma}  (1+\|Y_{t_{n_{(m)}}}^{x,\tau} \|^{\frac{p\tilde{q}}{2}+\tilde{d}\gamma})  \max\big\{ \tau_{i}^{\gamma \alpha}, (\rho(t_{i}-t_{n_(m)}))^\gamma \big\} \big) \Big|^2  \Big| \mathcal{F}_{t_{n_{(k_{2c-1})}}}\Big] \\
		&\leq K  \|f\|^2_{p,\gamma}  (1+ \E[\| Y_{t_{n_{(m)}}}^{x,\tau} \|^{p\tilde{q}+2\tilde{d}\gamma}  |\mathcal{F}_{t_{n_{(k_{2c-1})}}} ])
	\end{align*}
	and
	\begin{align*}
			&\quad	\E \Big[ \Big| \sum^\infty_{i=n_{(k_{2c-1})}} \tau_i \big( P^\tau_{t_{n_{(k_{2c-1})}}, t_i}  f(Y_{t_{n_{(k_{2c-1})}}}^{x,\tau})  -\mu(f)\big)\Big|^2  \Big| \mathcal{F}_{t_{n_{(k_{2c-1})}}}\Big] \\
		&\leq K  \|f\|^2_{p,\gamma}  \big(1+ \| Y_{t_{n_{(k_{2c-1})}}}^{x,\tau} \|^{p\tilde{q}+2\tilde{d}\gamma} \big).
	\end{align*}
	These two estimates yield
	\begin{align}  \label{equ_I2m}
		|\uppercase\expandafter{\romannumeral2}_{k_{2c-1},m }| \leq  K  \|f\|^2_{p,\gamma}  \big(1+ \| Y_{t_{n_{(k_{2c-1})}}}^{x,\tau} \|^{p\tilde{q}+2\tilde{d}\gamma}+ \E[\| Y_{t_{n_{(m)}}}^{x,\tau} \|^{p\tilde{q}+2\tilde{d}\gamma}  | \mathcal{F}_{t_{n_{(k_{2c-1})}}}] \big).
	\end{align}
	
{\bf{Estimate of $\uppercase\expandafter{\romannumeral3}_{k_{2c-1},m}.$}}
For $\uppercase\expandafter{\romannumeral3}_{k_{2c-1},m }$, we apply the same decomposition as in \eqref{equ_Mainpart3} for each $k\in[n_{(k_{2c-1})}+1, n_{(m)}]$ and obtain
\begin{align*}
		|\uppercase\expandafter{\romannumeral3}_{k_{2c-1},m }| 
	\leq  \sum^{n_{(m)}}_{k=n_{(k_{2c-1})}+1} [ \tau_{k-1}( \tilde{\I}_{1,k} +\I_{2,k})], 
\end{align*}
where
$$	\tilde{\I}_{1,k}
:=2 \Big| \E\big[ 	F^\tau_{k-1} (Y_{t_{k-1}}^{x,\tau} ) \big| \mathcal{F}_{t_{n_{(k_{2c-1})}}}\big] -\mu (F^\tau_{k-1}  )\Big| ,$$ with $F^\tau_{k-1}$  defined in \eqref{equ_definitionF}, and $\I_{2,k}$ as given in \eqref{equ_Mainpart3}.
Combining
	\begin{align*}
			\tilde{\I}_{1,k}
			&\leq  K \|f\|^2_{p,\gamma}  (1+\|Y_{t_{n_{(k_{2c-1})}}}^{x,\tau}\|^{\frac{\tilde{p}_F\tilde{q}}{2}+\tilde{d}\gamma})  \max\big\{ \tau_{k-1}^{\gamma \alpha}, (\rho(t_{k-1}-t_{n_{(k_{2c-1})}}))^\gamma \big\}
	\end{align*}
and \eqref{equ_Mainpart3_3}, we have 
	\begin{align}  \label{equ_I3m}
	|\uppercase\expandafter{\romannumeral3}_{k_{2c-1},m }| 
	\leq  K \|f\|^2_{p,\gamma}  (1+\|Y_{t_{n_{(k_{2c-1})}}}^{x,\tau}\|^{\frac{\tilde{p}_F\tilde{q}}{2}+\tilde{d}\gamma}).
	\end{align}
	
By plugging \eqref{equ_I1m}--\eqref{equ_I3m} into \eqref{equ_splitofZconditional}, we arrive at
	\begin{align*}
	&\quad 	\sum^m_{k_{2c}=k_{2c-1}+1}  \E\big[\big(\big|\tilde{\mathcal Z}^{x,\tau}_{k_{2c}} \big|^2-\tilde{\tau}_{k_{2c}}v^2 \big) \big|\mathcal{F}_{t_{n_{(k_{2c-1})}}} \big] \\
		&\leq K\|f\|^2_{p,\gamma} \sum^{n_{(m)}}_{k=n_{(k_{2c-1})}+1} \tau_{k-1}^2  \big(1+\E \big[ \| Y_{t_{k-1}}^{x,\tau}\|^p \big|\mathcal{F}_{t_{n_{(k_{2c-1})}}}	\big] \big) \\
		&\quad +K  \|f\|^2_{p,\gamma}  \big(1+ \| Y_{t_{n_{(k_{2c-1})}}}^{x,\tau} \|^{(p\tilde{q}+2\tilde{d}\gamma) \vee (\frac{\tilde{p}_F\tilde{q}}{2}+\tilde{d}\gamma) }+ \E\big[\| Y_{t_{n_{(m)}}}^{x,\tau} \|^{p\tilde{q}+2\tilde{d}\gamma}  | \mathcal{F}_{t_{n_{(k_{2c-1})}}} \big] \big),
	\end{align*}
	which leads to 
	\begin{align*}
\E \Big[\Big| \sum^m_{k_{2c}=k_{2c-1}+1}  \E\big[\big(\big|\tilde{\mathcal Z}^{x,\tau}_{k_{2c}} \big|^2-\tilde{\tau}_{k_{2c}}v^2 \big) \big| \mathcal{F}_{t_{n_{(k_{2c-1})}}} \big]  \Big|^{2c} \Big] 
		\leq K \|f\|^{4c}_{p,\gamma} (1+\|x\|^{2c\tilde{q}[(p\tilde{q}+2\tilde{d}\gamma) \vee (\frac{\tilde{p}_F\tilde{q}}{2}+\tilde{d}\gamma)] }).
	\end{align*}
	Here, we use  $2c[(p\tilde{q}+2\tilde{d}\gamma) \vee (\frac{\tilde{p}_F\tilde{q}}{2}+\tilde{d}\gamma)] \leq q$ and 
	\begin{align*}
	&\quad	\E \big[ \big|\sum^{n_{(m)}}_{k=n_{(k_{2c-1})}+1} \tau_{k-1}^2  \big(1+\E \big[ \| Y_{t_{k-1}}^{x,\tau}\|^p \big| \mathcal{F}_{t_{n_{(k_{2c-1})}}} \big] \big) \big|^{2c} \big]  \\
	&\leq \Big(\sum^{n_{(m)}}_{k=n_{(k_{2c-1})}+1} \tau^2_{k-1} \Big)^{2c} \sup_{k\geq 0} \E\big[  \| Y_{t_{k-1}}^{x,\tau}\|^{2cp} \big] 
	\leq K (1+\|x\|^{2cp\tilde{q}}).
	\end{align*}

{\textit{Step 2: Estimate of $ \E \Big[\big|	|\tilde{\mathcal Z}^{x,\tau}_{i}|^2-\tilde{\tau}_{i}v^2 \big|^{2c} \Big] $.}}
Applying  \eqref{equ_MomentboundednessTildeZ} with the order $4c$, we obtain
	\begin{align*}
		 \E \Big[\big|	|\tilde{\mathcal Z}^{x,\tau}_{i}|^2-\tilde{\tau}_{i}v^2 \big|^{2c} \Big] 
		&\leq \E \big[ 	|\tilde{\mathcal Z}^{x,\tau}_{i}|^{4c} \big] + \tilde{\tau}_{i}^{2c}v^{4c} 
		\leq K\|f\|^{4c}_{p,\gamma}  (1+\|x\|^{4c(\frac{p\tilde{q}^2}{2}+\tilde{q}\tilde{d}\gamma)})
	\end{align*}
		for $4c(\frac{p\tilde{q}}{2}+\tilde{d}\gamma) \leq q$. 
		
{\textit{Step 3: Uniformly boundedness of $G^{x,\tau,c}_{k_{2c-1}}$.}}
		By combining  the estimates obtained in \textit{Steps 1--2}, we deduce that for $1\leq i\leq j\leq N$ and $c\geq 1$,
		the following estimate holds:
		\begin{align*} 
&\quad G^{x,\tau,c}_{k_{2c-1}}  \\
	&\leq \bigg(	\E\Big[ \big||\tilde{\mathcal Z}^{x,\tau}_{k_{2c-1}}|^2-\tilde{\tau}_{k_{2c-1}}v^2\big|^{2c} \Big] \\
			&\quad + \sup_{1\leq m\leq N} \Big(	\E\Big[ \big||\tilde{\mathcal Z}^{x,\tau}_{k_{2c-1}}|^2-\tilde{\tau}_{k_{2c-1}}v^2\big|^{2c} \Big] \Big)^{\frac 12} \Big( \E \Big[ \Big| \sum^m_{k_{2c}=k_{2c-1}+1}  \E\big[\big(|\tilde{\mathcal Z}^{x,\tau}_{k_{2c}} |^2-\tilde{\tau}_{k_{2c}}v^2 \big) \big| \mathcal{F}_{t_{n_{(k_{2c-1})}}} \big] \Big|^{2c}   \Big] \Big)^{\frac 12} \bigg)^{\frac 1c}\\
			&\leq K \|f\|^{4}_{p,\gamma} (1+\|x\|^{4(\frac{p\tilde{q}^2}{2}+\tilde{q}\tilde{d}\gamma) \vee (2(\frac{p\tilde{q}^2}{2}+\tilde{q}\tilde{d}\gamma)+\tilde{q}(\frac{\tilde{p}_F\tilde{q}}{2}+\tilde{d}\gamma))}).
		\end{align*}
		
Combining \textit{Steps 1--3} proves the claim, based on which we have
	\begin{align*}
		\E \Big[ \Big| \frac{1}{\tilde{t}_N} \sum^{N}_{k=1} \big(|\tilde{\mathcal Z}^{x,\tau}_k|^2-v^2 \big) \Big|^{2c} \Big] &=\frac{1}{(\tilde{t}_N)^{2c}} \E \Big[ \Big|\sum^{N}_{k=1} \big(|\tilde{\mathcal  Z}^{x,\tau}_k|^2-\tilde{\tau}_kv^2 \big) \Big|^{2c} \Big] \leq \frac{1}{(\tilde{t}_N)^{2c}} \Big( \sum^N_{k_{2c-1}=1} G^{x,\tau,c}_{k_{2c-1}}\Big)^c  \\
		&\leq K \frac{N^{c}}{(\tilde{t}_N)^{2c}}\|f\|^{4c}_{p,\gamma} (1+\|x\|^{4c(\frac{p\tilde{q}^2}{2}+\tilde{q}\tilde{d}\gamma) \vee (2c(\frac{p\tilde{q}^2}{2}+\tilde{q}\tilde{d}\gamma)+c\tilde{q}(\frac{\tilde{p}_F\tilde{q}}{2}+\tilde{d}\gamma))}),
	\end{align*}
	and 
	for any $\varepsilon>0$,
	\begin{align*}
		&\quad \sum^\infty_{N=1} \mathbb{P} \Big\{ \Big| \frac{1}{\tilde{t}_N} \Big(\sum^{N}_{k=1} |\tilde{\mathcal Z}^{x,\tau}_k|^2\Big) -v^2  \Big|>\varepsilon \Big\}  \\
		&\leq  \frac{1}{\varepsilon^{2c}}\sum^\infty_{N=1} 	\E \Big[ \Big| \frac{1}{\tilde{t}_N} \sum^{N}_{k=1} \big(|\tilde{\mathcal Z}^{x,\tau}_k|^2-v^2 \big) \Big|^{2c} \Big] \\
		&\leq  K \frac{1}{\varepsilon^{2c}}\sum^\infty_{N=1}  \Big( \frac{N^{c}}{(\tilde{t}_N)^{2c}} \|f\|^{4c}_{p,\gamma} (1+\|x\|^{4c(\frac{p\tilde{q}^2}{2}+\tilde{q}\tilde{d}\gamma) \vee (2c(\frac{p\tilde{q}^2}{2}+\tilde{q}\tilde{d}\gamma)+c\tilde{q}(\frac{\tilde{p}_F\tilde{q}}{2}+\tilde{d}\gamma))}) \Big) \\
		&= K \frac{1}{\varepsilon^{2c}}\|f\|^{4c}_{p,\gamma} (1+\|x\|^{4c(\frac{p\tilde{q}^2}{2}+\tilde{q}\tilde{d}\gamma) \vee (2c(\frac{p\tilde{q}^2}{2}+\tilde{q}\tilde{d}\gamma)+c\tilde{q}(\frac{\tilde{p}_F\tilde{q}}{2}+\tilde{d}\gamma))}) \Big( \sum^\infty_{N=1}  \frac{N^{c}}{(\tilde{t}_N)^{2c}} \Big) <\infty.
	\end{align*}
	Here \eqref{equ_sumofTildeT} and  the fact that 
	\begin{align*}
		 \sum^\infty_{N=1}  \frac{N^c}{(\tilde{t}_N)^{2c}} = \sum^\infty_{N=1}  \frac{(N-\bar{\tau})^c}{(\tilde{t}_N)^{2c}} \frac{N^c}{(N-\bar{\tau})^c} \leq K \sum^\infty_{N=1}  \frac{1}{(\tilde{t}_N)^{c}} <\infty
	\end{align*}
	for any $c> 1$ are used.
	Therefore, the Borel--Cantelli lemma leads to 
$	\lim_{N\to \infty} \frac{1}{\tilde{t}_N} \big(\sum^{N}_{k=1} |\tilde{Z}^{x,\tau}_k|^2\big) =v^2$ almost surely, which completes the proof.
\end{proof}
 \section{Proofs of Propositions in \Cref{sec_numLIL}} \label{sec_ProofofChap3}
Based on the estimates in  \Cref{sec_martingaleProperties}, we now prove Propositions~\ref{prop_LILforMartingale}--\ref{prop_RemainderLimit2}.
For \Cref{prop_LILforMartingale}, we apply the LIL criterion for martingales given in \cite[Theorem 1]{Heyde1973},  which is stated as follows.
\begin{lemma} \label{lemma_Heyde1973}
	Assume that $\{M_n\}_{n\in\N}$ is a square integrable martingale, $\{Z_n\}_{n\in\N}$ 
	is the martingale difference sequence associated with $\{M_n\}_{n\in\N}$, and $S_n := \big(\mathbb{E}[(M_n)^2] \big)^{\frac 12} \to \infty $ as $n \to \infty$. If 
	\begin{equation} \label{equ_LILCondition1}
		\sum_{n=1}^{\infty} S_n^{-4} \mathbb{E}(|Z_n|^4 \mathbf{1}_{\{|Z_n|< \delta S_n\}}) < \infty, \qquad 	\text{for some}\  \delta>0,
	\end{equation} 
	\begin{equation}\label{equ_LILCondition1_1}
		\sum_{n=1}^{\infty} S_n^{-1} \mathbb{E}(|Z_n| \mathbf{1}_{\{|Z_n| \geq \varepsilon S_n\}}) < \infty, \qquad 	\text{for all}\  \varepsilon>0,
	\end{equation}
	and 
	\begin{equation}  \label{equ_LILCondition2}
		\lim_{n \to \infty} S_n^{-2} \sum_{k=1}^{n} Z_k^2 = 1, \qquad \text{a.s}.
	\end{equation}
	Then the sequence $\{\Lambda_n(\cdot)\}_{n > g(e)}$ of random variables on $\mathcal{C}([0,1]; \mathbb{R})$ defined by
	\[
	\Lambda_n(t) = \sum_{k=0}^{n-1} \mathbf{1}_{\{S_k^2 \leq t S_n^2 \leq S_{k+1}^2\}} \frac{M_k + (S_n^2 t - S_k^2)(S_{k+1}^2 - S_k^2)^{-1} Z_{k+1}}{\sqrt{2 S_n^2 \log \log S_n^2}}, \quad t \in [0,1]
	\]
 with $g(t):=\sup_{n\in\N^+} \{n: S_n^2 \leq t\}$ 
is almost surely relatively compact, and the set of its limit points coincides with $\mathcal{H}$. Here,
	$$
	\mathcal{H} := \Big\{  h\in \mathcal{C}([0,1]; \mathbb{R}) :h(0)=0,  \ h_t' \text{\ exists a.e. } t\in[0,1], \int^1_0 |h'_t|^2 \rd t \leq 1\Big\}.
	$$
\end{lemma}
The proof of \Cref{prop_LILforMartingale} is given below.
\begin{proof}[Proof of \Cref{prop_LILforMartingale}]
We first verify that $\{\tilde{\mathscr{M}}^{x,\tau}_{N}\}_{N\in\N}$ satisfies the conditions of Lemma \ref{lemma_Heyde1973}.
From \eqref{equ_orthnormalZ}, it follows that  $
	\E [\tilde{\mathcal{Z}}^{x,\tau}_{k} \tilde{\mathcal{Z}}^{x,\tau}_{j} ] = 0
	$   for $k\neq j$.  Together with
	 Proposition \ref{prop_ZkSquareMoment}, this yields
	\begin{align} \label{equ_L2limitofMn}
		\lim_{N \to \infty} \frac{1}{\tilde{t}_{N}} (\tilde{S}^{x,\tau}_{N} )^2 &=  \lim_{N \to \infty} \frac{1}{\tilde{t}_{N}} \mathbb{E}[( \tilde{\mathscr{M}}^{x,\tau}_{N})^2] =  \lim_{N \to \infty}  \frac{1}{\tilde{t}_{N}} \sum^{N}_{k=1} \E [| \tilde{ \mathcal{Z}}^{x,\tau}_k |^2] =v^2,
	\end{align}
	where we define $\tilde{S}^{x,\tau}_{N}:= ( \E[(\tilde{\mathscr{M}}^{x,\tau}_{N} )^2])^{\frac 12}$. 
	Since $\tilde{t}_{N} \to \infty$ as ${N} \to \infty$, \eqref{equ_L2limitofMn} implies that  $\tilde{S}^{x,\tau}_{N} \to \infty$.
	Moreover, by applying  \eqref{equ_L2limitofMn} and using \Cref{prop_almostsureZk} with  $c=2$, 
		we obtain
		\begin{align*}
			\lim_{N \to \infty}  (\tilde{S}^{x,\tau}_{N})^{-2}    \sum^N_{k=1} |\tilde{\mathcal{Z}}^{x,\tau}_k |^2 &=  	\lim_{N \to \infty}  \Big( \frac{1}{\tilde{t}_N} (\tilde{S}^{x,\tau}_{N} )^{2}  \Big)^{-1} \Big( \frac{1}{\tilde{t}_N}  \sum^N_{k=1} |\tilde{\mathcal{Z}}^{x,\tau}_k |^2 \Big) =1, \quad \text{a.s.}
		\end{align*}
		This shows that $\{\tilde{\mathscr{M}}^{x,\tau}_{N} \}_{N\in\N}$ satisfies  condition \eqref{equ_LILCondition2}.
		Combining  \Cref{prop_ZkSquareMoment} with \eqref{equ_L2limitofMn}, we next show  that  \eqref{equ_LILCondition1} holds for $\delta=1$ and that  \eqref{equ_LILCondition1_1} is satisfied. 
	Applying  \eqref{equ_MomentboundednessTildeZ}  with  $c=4$ yields $  \E [ |\tilde{\mathcal{Z}}^{x,\tau}_N |^4]= \E [| \sum^{n_{(N)}}_{n_{(N-1)}+1}Z^{x,\tau}_N |^4 ]<\infty$, establishing the required moment boundedness. 
		Then we derive
		\begin{align*}
			\sum^\infty_{N=1}   (\tilde{S}^{x,\tau}_{N} )^{-4}  \E \Big[ |\tilde{\mathcal{Z}}^{x,\tau}_N |^4  \mathbf{1}_{\{|\tilde{\mathcal{Z}}^{x,\tau}_N| < \tilde{S}^{x,\tau}_{N}\}}\Big] &\leq  	\sum^\infty_{N=1}   (\tilde{S}^{x,\tau}_{N})^{-4}  \E \Big[ |\tilde{\mathcal{Z}}^{x,\tau}_{N} |^4  \Big] \\
			&\leq K+	K \sum^\infty_{N=\lceil \bar{\tau} \rceil}   (\tilde{t}_N)^{-2}  \E \Big[ |\tilde{\mathcal{Z}}^{x,\tau}_{N} |^4  \Big] < \infty,
		\end{align*}
		which verifies \eqref{equ_LILCondition1}.
		Using 
		$\sum^{\infty}_{N=1} (\tilde{t}_N)^{-\frac 32}<\infty$ and the Chebyshev inequality,  for any $\varepsilon>0$, we obtain
		\begin{align*}
			&\quad \sum_{N=1}^{\infty} (\tilde{S}^{x,\tau}_{N} )^{-1} \mathbb{E}\Big[|\tilde{\mathcal{Z}}^{x,\tau}_{N}| \mathbf{1}_{\{|\tilde{\mathcal{Z}}^{x,\tau}_N| \geq \varepsilon \tilde{S}^{x,\tau}_{N} \}}\Big] \\
			&\leq  	\sum_{N=1}^{\infty} (\tilde{S}^{x,\tau}_{N})^{-1}  \Big( \mathbb{E}\Big[ |\tilde{\mathcal{Z}}^{x,\tau}_N|^2 \Big]\E\Big[| \mathbf{1}_{\{|\tilde{\mathcal{Z}}^{x,\tau}_N| \geq \varepsilon \tilde{S}^{x,\tau}_{N} \}}|^2 \Big] \Big)^{\frac 12} \\
			&\leq  \sum_{N=1}^{\infty} \frac{1}{\varepsilon^2} (\tilde{S}^{x,\tau}_{N} )^{-1}  \Big( \mathbb{E}\Big[ |\tilde{\mathcal{Z}}^{x,\tau}_{N}|^4 \Big] \Big)^{\frac 14}  \Big( (\tilde{S}^{x,\tau}_{N} )^{-4}  \E[ |\tilde{\mathcal{Z}}^{x,\tau}_{N}|^4 ] \Big)^{\frac 12} \\
			&\leq   \sum_{N=1}^{\infty} \frac{1}{\varepsilon^2} (\tilde{S}^{x,\tau}_{N} )^{-3}     \Big(   \E \big[ |\tilde{\mathcal{Z}}^{x,\tau}_{N}|^4 \big] \Big)^{\frac 34} \leq K+K   \sum_{N=\lceil \bar{\tau} \rceil}^{\infty} (\tilde{t}_N)^{-\frac 32}  \Big(     \E[ |\tilde{\mathcal{Z}}^{x,\tau}_N|^4 ]  \Big)^{\frac 34} <\infty,
		\end{align*}
		which verifies \eqref{equ_LILCondition1_1}. Therefore, we use Lemma \ref{lemma_Heyde1973} and  deduce that for $t\in[0,1]$,
		\begin{equation} \label{equ_LambdatExpression}
			\begin{aligned}
				\Lambda^{x,\tau}_N(t) &= \sum_{k=0}^{N-1}\Bigg( \mathbf{1}_{\{(\tilde{S}^{x,\tau}_{k})^2 \leq t (\tilde{S}^{x,\tau}_{N})^2 \leq (\tilde{S}^{x,\tau}_{k+1})^2\}} \\
				&\qquad\quad  \cdot \frac{\tilde{\mathscr{M}}^{x,\tau}_{k} + ((\tilde{S}^{x,\tau}_{N})^2 t - (\tilde{S}^{x,\tau}_{k})^2)((\tilde{S}^{x,\tau}_{k+1})^2 - (\tilde{S}^{x,\tau}_{k})^2)^{-1} \tilde{\mathcal{Z}}^{x,\tau}_{k+1}}{\sqrt{2 (\tilde{S}^{x,\tau}_{N})^2 \log \log( \tilde{S}^{x,\tau}_{N})^2}} \Bigg),
			\end{aligned} 
		\end{equation}
		 is almost surely relatively compact,  and the set of its limit points coincides with $\mathcal{H}$. 	 Since 
		$$\| h\|_{\mathcal{H}} =\sup_{t\in[0,1]} |h(t)| \leq\sup_{t\in[0,1]} \Big\{\sqrt{t} \Big(\int^t_0 |h'(s)|^2 \rd s\Big)^{\frac 12} \Big\}\leq \sup_{t\in[0,1]} \sqrt{t}  \leq 1$$ for any $h\in\mathcal{H}$, it holds that
		$$
	 \sup_{t\in[0,1]} \limsup_{N\to \infty}  \Lambda^{x,\tau}_N(t)  \leq 1, \quad \text{a.s.}
		$$
From \eqref{equ_LambdatExpression}, it follows that
		\begin{align*}
			\Lambda^{x,\tau}_N(1)&= \frac{\tilde{\mathscr{M}}^{x,\tau}_{N-1} +  \tilde{\mathcal{Z}}^{x,\tau}_{N}}{\sqrt{2 (\tilde{S}^{x,\tau}_{N})^2 \log \log( \tilde{S}^{x,\tau}_{N})^2}} 
			=   \frac{\tilde{\mathscr{M}}^{x,\tau}_{N} }{\sqrt{2 (\tilde{S}^{x,\tau}_{N})^2 \log \log( \tilde{S}^{x,\tau}_{N})^2}},
		\end{align*}
		which implies
		\begin{equation} \label{equ_limsupOfMn}
			\begin{aligned}
				\limsup_{N\to \infty} \frac{\tilde{\mathscr{M}}^{x,\tau}_{N} }{\sqrt{2 \tilde{t}_N \log \log \tilde{t}_N}} 
				&=   	 \limsup_{N\to \infty}  \frac{\tilde{\mathscr{M}}^{x,\tau}_{N} }{\sqrt{2 (\tilde{S}^{x,\tau}_{N})^2 \log \log( \tilde{S}^{x,\tau}_{N})^2}} \frac{\sqrt{2 (\tilde{S}^{x,\tau}_{N})^2 \log \log( \tilde{S}^{x,\tau}_{N})^2} }{\sqrt{2 \tilde{t}_N \log \log \tilde{t}_N }} \\
				&= v  ( \limsup_{N\to \infty} \Lambda^{x,\tau}_N(1) )\leq v, \quad \text{a.s.}
			\end{aligned}
		\end{equation}
		On the other hand, since the set of limit points of $\{ 	\Lambda^{x,\tau}_N(\cdot)\}_{N\in\N}$ in  $\mathcal{C}([0,1];\mathbb{R})$ coincides with $\mathcal{H}$, 
		there exists a subsequence $\{N_k\}_{k\in\N}$ such that $N_k \to \infty$ as $k\to\infty$ and 
		$$
		\lim_{k \to \infty} \sup_{t\in [0,1]} |  	\Lambda^{x,\tau}_{N_k}(t) -t|=0, \quad \text{a.s.},
		$$
		where we use the fact that $\tilde{h}(t):=t$ for  $t\in[0,1]$ belongs to $\mathcal{H}$. 
		This and \eqref{equ_limsupOfMn} yield
		$$
		\limsup_{N\to \infty} \frac{ \tilde{\mathscr{M}}^{x,\tau}_{N} }{\sqrt{2  \tilde{t}_N\log \log  \tilde{t}_N }} =v, \quad \text{a.s.},
		$$
		which provides the $\limsup$ result in \Cref{prop_LILforMartingale}.
		Replacing $f$ by $-f$, this formula turns into
		$$
		\liminf_{N\to \infty} \frac{ \tilde{\mathscr{M}}^{x,\tau}_{N} }{\sqrt{2 \tilde{t}_N  \log \log \tilde{t}_N  }} =-v, \quad \text{a.s.},
		$$
		which provides the $\liminf$ result  in \Cref{prop_LILforMartingale}.
	\end{proof}
	
	\begin{proof}[Proof of Proposition \ref{prop_RemainTermLimit}]
		Using  \eqref{equ_PtkApprmuFromtm} and \eqref{equ_PtkApprmu}, we have 
			\begin{align*}
				| \mathscr R^{x,\tau}_{k} | &\leq \sum_{i=k+1}^{\infty}\tau_i \Big| P^\tau_{t_k,t_i} f(Y_{t_k}^{x,\tau})-\mu(f)\Big| + \sum_{i=0}^{\infty}\tau_i\Big|P_{t_i}^{\tau} f(x)-\mu(f)\Big| \\
				&\leq K\|f\|_{p,\gamma} \Big(1+ \|Y_{t_k}^{x,\tau}\|^{\frac{p\tilde{q}}{2}+\tilde{d}\gamma} \Big)\sum_{i=k+1}^{\infty}\big(\tau_i  \max\big\{ \tau_i^{\gamma \alpha}, (\rho(t_i-t_k))^\gamma \big\} \big) \\
				&\quad + K\|f\|_{p,\gamma} \Big(1+ \|x\|^{\frac{p\tilde{q}}{2}+\tilde{d}\gamma} \Big)\sum_{i=0}^{\infty}\big(\tau_i  \max\big\{ \tau_i^{\gamma \alpha}, (\rho(t_i))^\gamma \big\} \big) \\
				&\leq  K\|f\|_{p,\gamma} \Big(1+ \|Y_{t_k}^{x,\tau}\|^{\frac{p\tilde{q}}{2}+\tilde{d}\gamma}+ \|x\|^{\frac{p\tilde{q}}{2}+\tilde{d}\gamma} \Big),
			\end{align*}
which combining with $c( \frac{p\tilde{q}}{2}+\tilde{d}\gamma) \leq q$ yields 
		\begin{align*}
			\sup_{k\in\N}	\E [ 	| \mathscr R^{x,\tau}_{k} |^c ] \leq   K\|f\|^c_{p,\gamma}  \big(1+\|x\|^{c(\frac{p\tilde{q}^2}{2}+\tilde{q}\tilde{d}\gamma)} \big).
		\end{align*}
	This implies the moment boundedness of $ \tilde{\mathscr R}^{x,\tau}_{N}$, since $ \tilde{\mathscr R}^{x,\tau}_{N}= \mathscr R^{x,\tau}_{n_{(N)}}$. 
		Using \eqref{equ_sumofTildeT}, we know that  $\sum^\infty_{N=1} (\tilde{t}_N)^{-\frac c2}$ converges with any $c>2$. Hence for any $\varepsilon>0$,
		\begin{align*}
			\sum^{\infty}_{N=1} \mathbb{P} \{  \frac{1}{(\tilde{t}_N)^{\frac 12}}| \tilde{\mathscr R}^{x,\tau}_{N}| > \varepsilon \} &\leq \frac{1}{\varepsilon^c} 	\sum^{\infty}_{N=1}  \frac{1}{(\tilde{t}_N)^{\frac c2}}\E [ 	| \tilde{\mathscr R}^{x,\tau}_{N} |^c ] \\
			&\leq \frac{K}{\varepsilon^c} \|f\|^c_{p,\gamma}  \big(1+\|x\|^{c(\frac{p\tilde{q}^2}{2}+\tilde{q}\tilde{d}\gamma)} \big)  	\sum^{\infty}_{N=1}(\tilde{t}_N)^{-\frac c2} <\infty.
		\end{align*}
		By utilizing the Borel--Cantelli Lemma, we prove that $\lim_{N\to \infty}	(\tilde{t}_N)^{-\frac 12}	\tilde{\mathscr R}^{x,\tau}_{N} =0$ almost surely, which ends the proof.
	\end{proof}
	 \begin{proof}[Proof of \Cref{prop_RemainderLimit2}]
		Let $\bar{R}^{x,\tau}_{\tilde{k}}:= \sum^{n_{(\tilde{k}+1)}}_{i=n_{(\tilde{k})}+1} \tau_i | f(Y^{x,\tau}_{t_i}) - \mu(f)|$. Combining this with \eqref{equ_defR2}, we arrive at
		\begin{align*}
			| R^{x,\tau}_{k} | 
			\leq \sum^{k}_{i=n_{(\tilde{k})}+1}\tau_i \big| f(Y^{x,\tau}_{t_i}) - \mu(f) \big| 
			\leq \bar{R}^{x,\tau}_{\tilde{k}}. 
		\end{align*}
		Using $\frac{cp}{2}\leq q$, we have
		\begin{align*}
			\E \big[ \big|\bar{R}^{x,\tau}_{\tilde{k}}\big|^c \big]  &= \E \Big[ \Big(  \sum^{n_{(\tilde{k}+1)}}_{i=n_{(\tilde{k})}+1}\tau_i \big| f(Y^{x,\tau}_{t_i}) - \mu(f) \big|\Big)^c \Big] \\
			&\leq  \big(  \sum^{n_{(\tilde{k}+1)}}_{i=n_{(\tilde{k})}+1}\tau_i \big)^{c-1} \Big(K\|f\|_{p,\gamma}^c  \sum^{n_{(\tilde{k}+1)}}_{i=n_{(\tilde{k})}+1}\tau_i  \Big(1+\E \big[ \| Y^{x,\tau}_{t_i}\|^{\frac{cp}{2}} \big]\Big)\Big)  \\
			&\leq  K\|f\|_{p,\gamma}^c \big(1+\|x\|^{\frac{cp\tilde{q}}{2}} \big) \big(  \sum^{n_{(\tilde{k}+1)}}_{i=n_{(\tilde{k})}+1}\tau_i \big)^{c}  
			\leq K \|f\|_{p,\gamma}^c  (1+2\bar{\tau})^c  (1+\|x\|^{\frac{cp\tilde{q}}{2}} ),
		\end{align*}
	which implies that  for any $\varepsilon>0$, 
		\begin{align*}
			\mathbb{P} \Big\{  \frac{1 }{(t_{n_{(\tilde{k})}})^{\frac 12}}  \big| \bar{R}^{x,\tau}_{\tilde{k}} \big|> \varepsilon \Big\} &\leq \frac{1}{\varepsilon^c} \frac{1}{ (t_{n_{(\tilde{k})}})^{\frac c2}} \E \big[ \big|\bar{R}^{x,\tau}_{\tilde{k}} \big|^c \big] 
			\leq   \frac{1}{\varepsilon^c} \frac{1}{ (t_{n_{(\tilde{k})}})^{\frac c2}} K \|f\|_{p,\gamma}^c   (1+\|x\|^{\frac{cp\tilde{q}}{2} }).
		\end{align*}
	Using \eqref{equ_sumofTildeT}, we obtain 
		\begin{align*}
			\sum^\infty_{\tilde{k}=1} 	\mathbb{P} \Big\{  \frac{1 }{(t_{n_{(\tilde{k})}})^{\frac 12}}  \big| \bar{R}^{x,\tau}_{\tilde{k}}  \big|> \varepsilon \Big\} \leq  	\Big(\sum^\infty_{\tilde{k}=1} 	\frac{1}{ (t_{n_{(\tilde{k})}})^{\frac c2}} \Big)   \frac{1}{\varepsilon^c}K \|f\|_{p,\gamma}^c  (1+\|x\|^{\frac{cp\tilde{q}}{2} }) <\infty,
		\end{align*}
		which together with the Borel--Cantelli lemma yields 
		\begin{align}\label{equ_limitR=0}
			\lim_{\tilde{k}\to \infty}  \frac{1}{(t_{n_{(\tilde{k})}})^{\frac 12}}  \big| \bar{R}^{x,\tau}_{\tilde{k}} \big|  =0, \quad \text{a.s.}
		\end{align}
		Combining \eqref{equ_limitR=0} with the fact that $k\in [n_{(\tilde{k})}, n_{(\tilde{k}+1)})$, 
		we have
		\begin{align*}
			\lim_{k\to \infty}  \frac{1}{(t_{k})^{\frac 12}}  \big| R^{x,\tau}_{k} \big| \leq   	\lim_{\tilde{k}\to \infty}  \frac{1}{(t_{n_{(\tilde{k})}})^{\frac 12} }  \big| \bar{R}^{x,\tau}_{\tilde{k}} \big| \Big(\frac{ t_{n_{(\tilde{k})}}} {t_{n_{(\tilde{k+1})}} } \Big)^{\frac 12} =0.
		\end{align*}
		Here we use  $	\lim_{\tilde{k}\to \infty} t_{n_{(\tilde{k})}}/t_{n_{(\tilde{k}+1)}}= 1$ as in the proof of \Cref{thm_numerLIL}. 
	 	This completes the proof.
	\end{proof}
\section{Examples} \label{sec_examples}
Finally, we apply our theoretical results to specific SODE and SPDE,  illustrating the applicability  of our results. 
\subsection{SODE} \label{subsection_SODE}
In this subsection, we consider the case of  SODE  as a concrete example of our general result.
To be specific, we consider the SODE in the following form
\begin{align}\label{ex_sode}
	\mathrm d X_t=b(X_t)\mathrm dt+\sigma(X_t) \mathrm dW(t),\quad X_0=x \in \RR^d,
\end{align} 
where $\{W(t),t\ge 0\}$ is an $m$-dimensional standard Brownian motion on a filtered probability space $(\Omega,\mathcal F,\{\mathcal F_t\}_{t\ge 0},\mathbb P)$.  Here, $b: \RR^d \rightarrow \RR^d$ and  $\sigma:\RR^d \rightarrow  \RR^{d\times m}$ satisfy the following assumptions.
 
\begin{assp}\label{ode_a1}
	There exist positive constants $c_1, c_2$ such that
	\begin{align*}
		\|\sigma(u_1)-\sigma(u_2)\|_{HS}\le c_1|u_1-u_2|, \quad &\forall~u_1,u_2\in \RR^d,\\
		\|\sigma(u)\|_{HS}\le c_2, \quad &\forall~u\in\RR^d.
	\end{align*}
\end{assp}

\begin{assp}\label{ode_a2}
	There exist $c_3>\frac{15}{2} c_1^2, \ c_4>0$, and $\bar{q}\ge 1$ such that
	\begin{align*}
	\langle u_1-u_2,b(u_1)-b(u_2)\rangle \le -c_3|u_1-u_2|^2, \quad\forall~u_1,u_2\in\RR^d,\\
		|b(u_1)-b(u_2)|\le c_4(1+|u_1|^{\bar{q}-1}+|u_2|^{\bar{q}-1})|u_1-u_2|, \quad\forall~u_1,u_2\in\RR^d.
	\end{align*}
\end{assp} 
As mentioned in \cite{Jin2025}, Assumptions \ref{ode_a1}--\ref{ode_a2} lead to 
\begin{align*}
 2		\langle u_1-u_2,b(u_1)-b(u_2)\rangle + 15	\|\sigma(u_1)-\sigma(u_2)\|_{HS}^2 \le -c_5|u_1-u_2|^2, \quad &\forall~u_1, u_2\in\RR^d, \\
 2 \langle u,b(u)\rangle +  c_6 \| \sigma(u) \|^2_{HS} \leq -c_3|u|^2 + c_7, \quad &\forall~u\in\RR^d,
\end{align*}
for any $c_6>0$. Here, $c_5 :=2c_3-15c_1^2$,  $c_7 := \frac{1}{c_3} |b(0)|^2 +c_6c_2^2$, and
$\|A\|_{HS} := \sqrt{\langle A, A \rangle_{HS}}$ with $\langle A, B \rangle_{HS} :=\sum^d_{i=1} \sum^m_{j=1} A_{ij}B_{ij}$ for matrices $A=(A_{ij})_{d \times m} \in \RR^{d \times m}$ and $B=(B_{ij})_{d \times m} \in \RR^{d \times m}$.
We additionally assume that $c_5 \leq 13$.

Using  results in \cite[Proposition 2.2]{Jin2025}, we find that
Assumption \ref{a1} (\romannumeral1)--(\romannumeral 3)  are satisfied with $\tilde r=1$,  $r\ge 2, \beta=0$,  $\rho(t)=e^{-\frac{c_5}{2}t}$,  $\gamma_1\in (0,1]$, $l=\bar{q}$, and $\tilde{l}=\frac 12$. Then using Proposition \ref{thm_exactLIL},  we prove that the exact solution $\{X^x_t\}_{t\ge 0}$  admits a unique invariant measure $\mu$ and obeys the LIL for $f \in \mathcal{C}_{p,\gamma}$ with any $p\geq1$ and $\gamma\in(0,1]$. To be specific,
\begin{align*}
	\limsup_{t\to \infty} \frac{ \int^t_0 (f(X^x_s)-\mu(f)) \rd s}{\sqrt{2t\log\log t} } =v,\quad a.s. \quad \text{and} \quad 
	\liminf_{t\to \infty} \frac{ \int^t_{0} (f(X^x_s)-\mu(f)) \rd s}{\sqrt{2t\log\log t}  } = -v, \quad a.s.
\end{align*}

Let   $\tau=\{\tau_n=\frac 1n\}_{n\in\N^+}$ be the time-step  sequence. 
The backward EM (BEM) method of \eqref{ex_sode} with time-step sequence $\tau$  is given by 
\begin{align} \label{equ_SDEBEM}
Y^{x,\tau}_{t_{i+1}}=Y^{x,\tau}_{t_i}+b(Y^{x,\tau}_{t_{i+1}})\tau_{i+1}+\sigma(Y^{x,\tau}_{t_{i}})\Delta W_i,\quad \Delta W_i=W(t_{i+1})-W(t_i).
\end{align}
For the BEM method, we verify Assumptions \ref{a2}--\ref{a3}
 as follows. 
 
{ \textit{ \Cref{a2} (\romannumeral 1) holds for any $q\geq 2$ and $\tilde{q}=1$.}} 
For any fixed $m,n\in \N$ such that $m\leq n-1$ 
and any $u\in\RR^d$, 
\begin{equation} \label{sode_a1step1}
\begin{aligned}
&\quad  \E[ |Y^{u, \tau}_{t_m, t_{n}}|^2] - \E[ |Y^{u, \tau}_{t_{m}, t_{n-1}}|^2] +  \E[ |Y^{u, \tau}_{t_m, t_{n}} - Y^{u, \tau}_{t_{m}, t_{n-1}}|^2] \\
 &= 2 \E \langle   Y^{u, \tau}_{t_m, t_{n}} - Y^{u, \tau}_{t_{m}, t_{n-1}}, Y^{u, \tau}_{t_m, t_{n}} \rangle \\
	&\leq -\tau_n c_3 \E [ |Y^{u, \tau}_{t_m, t_{n}}|^2] +c_7\tau_n -c_6\tau_n  \E [ \|\sigma(Y^{u, \tau}_{t_m, t_{n}} )\|^2_{HS}] +  \tau_n  \E [ \|\sigma(Y^{u, \tau}_{t_{m}, t_{n-1}} )\|_{HS}^2] + \E[ |Y^{u, \tau}_{t_m, t_{n}} - Y^{u, \tau}_{t_{m}, t_{n-1}}|^2]. 
\end{aligned}
\end{equation}
This leads to 
\begin{align*}
	(1+c_3\tau_{n}) \E[ |Y^{u, \tau}_{t_m, t_{n}}|^2] +c_6\tau_n  \E [ \|\sigma(Y^{u, \tau}_{t_m, t_{n}} )\|_{HS}^2] 
&	\leq c_7\tau_n + \E [|Y^{u, \tau}_{t_m,t_{n-1}}|^2]+ \tau_{n} \E [ \|\sigma(Y^{u, \tau}_{t_{m}, t_{n-1}} )\|_{HS}^2]. 
\end{align*}
By the arbitrariness of $c_6>0$, choosing $c_6 \geq 1+c_3 \bar{\tau}$ yields
\begin{align*}
	&\quad 
		\E[ |Y^{u, \tau}_{t_m, t_{n}}|^2]  +  \tau_n  \E [ \|\sigma(Y^{u, \tau}_{t_m, t_{n}} )\|_{HS}^2] -\frac{c_7}{c_3} \\
		&\leq \frac{1}{1+c_3\tau_{n}} \Big( \E [|Y^{u, \tau}_{t_m,t_{n-1}}|^2]+ \tau_{n-1} \E [ \|\sigma(Y^{u, \tau}_{t_{m}, t_{n-1}} )\|_{HS}^2 ] -\frac{c_7}{c_3} \Big) \\
		&\leq  \frac{1}{\prod_{i=m+1}^{n}(1+c_3\tau_i)} \Big( |u|^2+ \tau_{m}  \|\sigma(u )\|_{HS}^2  -\frac{c_7}{c_3} \Big) 
		\leq K(1+|u|^2).
\end{align*}
Similar to \cite[Theorem 4.1]{Jin2025},  Assumption \ref{a2} (\romannumeral 1) holds for any $q\geq 2$ and $\tilde{q}=1$ by induction.

\textit{\Cref{a2} (\romannumeral 2) holds for $\kappa=0, \rho^\tau (t)=e^{-\frac{c_8}{2} t}$.}
We
define $D_{m,n}^{x,y}:=Y^{x, \tau}_{t_m, t_{n}}-Y^{y, \tau}_{t_m, t_{n}}$ for any   $n,m\in \N$ such that $n\geq m$, with initial value $D_{m,m}^{x,y}=x-y$. 
Using the same technique in \eqref{sode_a1step1}, we obtain
\begin{align*}
	&\quad (1+c_5 \tau_{n})\E\big[ |D_{m,n}^{x,y}|^2\big] +15 \tau_{n} \E \big[\big\|\sigma\big(Y^{x, \tau}_{t_m, t_{n}}\big)-\sigma\big(Y^{y, \tau}_{t_m, t_{n}}\big)\big\|_{HS}^2\big]\\
	&\leq \E\big[ |D_{m,n-1}^{x,y}|^2]+ \tau_{n} \E\big[\big\|\sigma\big(Y^{x, \tau}_{t_m, t_{n-1}}\big)-\sigma\big(Y^{y, \tau}_{t_m, t_{n-1}}\big)\big\|_{HS}^2\big],
\end{align*}
which together with  $1+c_5\bar{\tau}\leq 1+c_5 \leq 15$, $a<e^{-(1-a)}$ for $a\in (0,1)$, and  $c_8:=  \frac{c_5}{1+c_5\bar{\tau}}$ implies 
\begin{align*}
	&\quad \E[ |D_{m,n}^{x,y}|^2]  +  \tau_{n} \E \big[\big\|\sigma\big(Y^{x, \tau}_{t_m, t_{n}}\big)-\sigma\big(Y^{y, \tau}_{t_m, t_{n}}\big)\big\|_{HS}^2 \big] \\
	&\leq \frac{1}{1+c_5 \tau_{n}} \Big( \E[ |D_{m,n-1}^{x,y}|^2]+ \tau_{n-1} \E[\big\|\sigma\big(Y^{x, \tau}_{t_m, t_{n-1}}\big)-\sigma\big(Y^{y, \tau}_{t_m, t_{n-1}}\big)\big\|_{HS}^2] \Big) \\
	&\leq \exp\Big\{-\frac{c_5}{1+c_5\tau_n} \tau_n\Big\}  \Big( \E[ |D_{m,n-1}^{x,y}|^2]+ \tau_{n-1} \E[\big\|\sigma\big(Y^{x, \tau}_{t_m, t_{n-1}}\big)-\sigma\big(Y^{y, \tau}_{t_m, t_{n-1}}\big)\big\|_{HS}^2] \Big) \\
&\leq K \exp\Big\{-\frac{c_5}{1+c_5\bar{\tau}} \sum^n_{i=m+1} \tau_i \Big\}  |x-y|^2 
	\leq  K \exp\big\{-c_8 (t_n-t_m) \big\} |x-y|^2.
\end{align*}
Hence, Assumption \ref{a2} (\romannumeral 2) is verified with $\kappa=0, \rho^\tau (t)=e^{-\frac{c_8}{2} t}$. 

\textit{\Cref{a3} holds for $\alpha= \frac{c_5}{2}$ when $\frac{c_5}{2}<1$, $\alpha=1-\varepsilon$ for any $\varepsilon\in (0,1)$ when $\frac{c_5}{2}=1$,  and $\alpha=1$ when $\frac{c_5}{2}>1$.}
We define the error  by $e_{m,n}:= X^{x}_{t_m, t_{n}} - Y^{x, \tau}_{t_m, t_{n}}$ for $m,n\in\N$ such that $n>m$ and $e_{m,m}=0$. Then we obtain
\begin{align*}
	e_{m,n}-e_{m,n-1} 
	&=\int^{t_n}_{t_{n-1}} b(X_{t_m,s}^x) \rd s + \int^{t_n}_{t_{n-1}} \sigma(X_{t_m,s}^x) \rd W(s) 
- \tau_{n} b(Y^{x, \tau}_{t_m, t_{n}}) -\sigma(Y^{x, \tau}_{t_m, t_{n-1}}) \Delta W_{n-1},
\end{align*}
which combining the Young's inequality with $\varepsilon_n\in(0,1)$ implies
\begin{align} \label{equ_odea2decom}
	&\quad 2\E[ \langle e_{m,n}-e_{m,n-1}, e_{m,n} \rangle] \nn \\
	&=2\tau_n \E \langle (b(X^{x}_{t_m, t_n} )-b( Y^{x, \tau}_{t_m, t_{n}} )), e_{m,n} \rangle +2 \E \langle \int^{t_n}_{t_{n-1}} (b(X^{x}_{t_m,s} )-b( X^{x}_{t_m, t_{n}} )) \rd s, e_{m,n} \rangle \nn\\
	&\quad + 2\E \langle \int^{t_n}_{t_{n-1}} (\sigma(X^{x}_{t_m,s} )-\sigma( X^{x}_{ t_m, t_{n-1}} )) \rd W(s), e_{m,n} \rangle \nn\\
	&\quad +2 \E \langle (\sigma(X^{x}_{t_m, t_{n-1}} )-\sigma( Y^{x, \tau}_{t_m, t_{n-1}} )) \Delta W_{n-1}, e_{m,n} \rangle \nn \\
	&\leq -c_5\tau_n \E [|e_{m,n}|^2 ] -15\tau_n \E[ \| \sigma(X^{x}_{t_m, t_{n}} )-\sigma( Y^{x, \tau}_{t_m, t_{n}} )\|_{HS}^2] + \varepsilon_n  \E[ |e_{m,n}|^2]  \nn \\
	&\quad + \frac{1}{\varepsilon_n} \E \Big[\Big| \int^{t_n}_{t_{n-1}} (b(X^{x}_{t_m,s} )-b( X^{x}_{t_m,t_{n}} )) \rd s\Big|^2\Big] + 2\E\Big[\Big|\int^{t_n}_{t_{n-1}} \sigma(X^{x}_{t_m,s} )-\sigma( X^{x, \tau}_{t_m, t_{n-1}} ) \rd W_s \Big|^2 \Big] \nn \\
	&\quad + 2\tau_n \E[ \| \sigma(X^{x}_{t_m, t_{n-1}} )-\sigma( Y^{x, \tau}_{t_m, t_{n-1}} )\|_{HS}^2]  +\E[| e_{m,n}-e_{m,n-1} |^2].
\end{align}
Using \Cref{a1} (\romannumeral 1), \Cref{ode_a2} 
and the It\^{o} isometry, it holds that 
\begin{align*}
	&\quad \E \Big[\Big|  \int^{t_n}_{t_{n-1}} (b(X^{x}_{t_m,s} )-b( X^{x}_{t_m, t_{n}} ))\, \rd s\Big|^2\Big] \\
	&\leq K \tau_n\E \Big[ \int^{t_n}_{t_{n-1}} |b(X^{x}_{t_m,s} )-b( X^{x}_{t_m, t_{n}} )|^2\rd s\Big] \\
	&\leq K \tau_n\E \Big[ \int^{t_n}_{t_{n-1}} \big( 1+  |X^{x}_{t_m,s}|^{2(\bar{q}-1)}+  | X^{x}_{t_m, t_{n}} |^{2(\bar{q}-1)}  \big) | X^{x}_{t_m,s}-X^{x}_{t_m, t_{n}}|^2 \rd s\Big] 
	  \leq K (1+|x|^{4\bar{q}-2}) \tau_n^3
\end{align*}
and
\begin{align*}
	\E\Big[\Big|\int^{t_n}_{t_{n-1}} \sigma(X^{x}_{t_m,s} )-\sigma( X^{x, \tau}_{t_m, t_{n-1}} ) \rd W_s \Big|^2 \Big]  \leq K \tau_n^2,
\end{align*}
where the continuity property of $\{X^x_t\}_{t\geq 0}$ such that $(\E [ | X^x_t -X^x_s|^p])^{\frac 1p} \leq K(1+|x|^{\bar{q}} ) |t-s|^{1/2}$ for any $p\geq 2$ is used.
Using the identity $\E[ |e_{m,n}|^2] -\E[ |e_{m,n-1}|^2] +\E[| e_{m,n}-e_{m,n-1} |^2] =2\E[ \langle e_{m,n}-e_{m,n-1}, e_{m,n} \rangle] $ and cancelling the same terms on both sides of \eqref{equ_odea2decom}, we arrive at 
\begin{align*}
&\quad 	\E[ |e_{m,n}|^2] -\E[ |e_{m,n-1}|^2] + 15\tau_n \E [\| \sigma(X^{x}_{t_m,t_n} )-\sigma( Y^{x, \tau}_{ t_m, t_{n}} )\|_{HS}^2] \\
&\leq (-c_5\tau_n +\varepsilon_n) \E [|e_{m,n}|^2] +K\frac{1}{\varepsilon_n}(1+|x|^{4\bar{q}-2}) \tau_n^3 +K \tau_n^2 \\
&\quad + 2\tau_n \E [\| \sigma(X^{x}_{t_m,t_{n-1}} )-\sigma( Y^{x, \tau}_{ t_m, t_{n-1}} )\|_{HS}^2]. 
\end{align*}
By iteration, we have
\begin{align}  \label{equ_odeEmn1}
	&\quad \E[ |e_{m,n}|^2] +2 \tau_n \E [\| \sigma(X^{x}_{t_m, t_n} )-\sigma( Y^{x, \tau}_{t_m, t_{n}} )\|_{HS}^2] \nn\\
	&\leq \frac{1}{1+ c_5\tau_n -\varepsilon_n} ( \E[ |e_{m, n-1}|^2] +2 \tau_n \E [\| \sigma(X^{x}_{t_m, t_{n-1}} )-\sigma( Y^{x, \tau}_{t_m, t_{n-1}} )\|_{HS}^2] ) \nn \\
	&\quad +  K(1+|x|^{4\bar{q}-2})\frac{\tau_n^2}{1+ c_5\tau_n -\varepsilon_n} \nn \\
	&\leq \frac{1}{\prod^n_{i=m+1} (1+ c_5\tau_i -\varepsilon_i)} ( \E [|e_{m,m}|^2]+2\tau_m \E[\| \sigma(x)-\sigma(x) \|_{HS}^2]  ) \nn \\
	&\quad + K(1+|x|^{4\bar{q}-2})\sum^n_{k=m+1} \Big(\frac{1}{ \prod^n_{i=k} (1+ c_5\tau_i -\varepsilon_i)} \tau_k^2 \Big),
\end{align}
where we take $\varepsilon_n = \frac{c_5\tau_n}{2}$ such that  $1 \leq 1+ c_5\tau_n -\varepsilon_n\leq \frac{15}{2}$. 
Combining $2\tau_n \E [\| \sigma(X^{x}_{t_m, t_n} )-\sigma( Y^{x, \tau}_{t_m, t_{n}} )\|_{HS}^2]>0$ and \eqref{equ_odeEmn1}, we  arrive at 
\begin{align*}
	\E[ |e_{m,n}|^2]  \leq K(1+|x|^{4\bar{q}-2})\sum^n_{k=m+1} \Big(\frac{1}{ \prod^n_{i=k} (1+ \frac{c_5}{2}\frac 1i)} \frac{1}{k^2} \Big).
\end{align*}
Here,
\begin{align*}
	\prod^n_{i=k} \big(1+ \frac{c_5}{2}\frac 1i \big) &\geq \exp \Big\{ \frac{c_5}{2} (\log(2(n+1)+c_5)- \log(2k+c_5)) \Big\} 
	= \Big(\frac{2(n+1)+c_5}{2k+c_5} \Big)^{c_5/2},
\end{align*}
since 
\begin{align*}
	\log \Big( \prod^n_{i=k} \big(1+ \frac{c_5}{2}\frac 1i \big) \Big) &=\sum^n_{i=k} \log  \Big( 1+ \frac{c_5}{2}\frac 1i  \Big) \geq \sum^n_{i=k} \frac{  \frac{c_5}{2i}}{1+  \frac{c_5}{2i}} =\sum^n_{i=k} \frac{c_5}{2i+c_5} \\
	&\geq \frac{c_5}{2} (\log(2(n+1)+c_5)- \log(2k+c_5)).
\end{align*}
For the case $\frac{c_5}{2}=1$, for any $n> m$  and $\varepsilon \in(0,1)$, we have 
\begin{align} \label{equ_ode_emn=1}
		\E[ |e_{m,n}|^2]  &\leq K(1+|x|^{4\bar{q}-2}) \sum^n_{k=m+1} \Big( \frac{k+\frac{c_5}{2}}{(n+1)+\frac{c_5}{2}}  \frac{1}{k^2} \Big)\nn \\
		&= K(1+|x|^{4\bar{q}-2}) \Big( \frac{1}{n}   \Big(\sum^n_{k=m+1} \frac{1}{k} \Big) + \frac{\frac{c_5}{2}}{n}  \Big(\sum^n_{k=m+1} \frac{1}{k^2} \Big) \Big) \nn\\
		&\leq K (1+|x|^{4\bar{q}-2})\Big(\frac{1}{n}  (\log n+1) +\frac{1}{n} \Big)  \nn  \\
	&	\leq  K(1+|x|^{4\bar{q}-2})\frac{1}{n} (\log n \vee 1)
	\leq K(1+|x|^{4\bar{q}-2}) \tau_n^{1-\varepsilon}.
\end{align}
For the case  $\frac{c_5}{2}> 1$, we obtain
\begin{align}  \label{equ_ode_emn>1}
	\E[ |e_{m,n}|^2]  &\leq K(1+|x|^{4\bar{q}-2}) \sum^n_{k=m+1} \Big(\big( \frac{k+\frac{c_5}{2}}{(n+1)+\frac{c_5}{2}} \big)^{\frac{c_5}{2}}  \frac{1}{k^2} \Big) \nn \\
	&\leq K(1+|x|^{4\bar{q}-2}) \frac{1}{n^{\frac{c_5}{2}}} \Big(  \sum^n_{k=m+1} \big( k^{\frac{c_5}{2}}  +1 \big) \frac{1}{k^2} \Big) 
	\leq K(1+|x|^{4\bar{q}-2})\tau_n.
\end{align}
For the case $0<\frac{c_5}{2}<1$,  we similarly have
\begin{align}  \label{equ_ode_emn<1}
	\E[ |e_{m,n}|^2]  &\leq K(1+|x|^{4\bar{q}-2}) \sum^n_{k=m+1} \Big(\big( \frac{2k}{(n+1)+\frac{c_5}{2}} \big)^{\frac{c_5}{2}}  \frac{1}{k^2} \Big) \nn\\
	&\leq K(1+|x|^{4\bar{q}-2})  \frac{1}{n^{\frac{c_5}{2}}}    \Big(\sum^n_{k=m+1} \frac{1}{k^{2-\frac{c_5}{2}}} \Big)   
	\leq  K(1+|x|^{4\bar{q}-2})\tau_n^{\frac{c_5}{2}}.
\end{align}
Combining \eqref{equ_ode_emn=1}--\eqref{equ_ode_emn<1}, the numerical method satisfies Assumption \ref{a3} with $\alpha= \frac{c_5}{2}$ when $\frac{c_5}{2}<1$, $\alpha=1-\varepsilon$ for any $\varepsilon\in (0,1)$ when $\frac{c_5}{2}=1$ and $\alpha=1$ when $\frac{c_5}{2}>1$, respectively. 

Conditions (\romannumeral 1) and (\romannumeral 4) in  Theorem \ref{thm_numerLIL} hold since $\sum^\infty_{k=1} (\tau_k)^{1+\delta} <\infty$ for any $\delta > 0$. 
Conditions (\romannumeral 2) and (\romannumeral 3) hold since for any $\gamma\in(0,1]$ and $k\geq 0$, 
$$
\sum^\infty_{i=k} \tau_i (\rho^\tau(t_i-t_{k}))^\gamma \leq \sum^\infty_{i=k} \tau_i (\rho(t_i-t_{k}))^\gamma \leq 1+ \int^\infty_{t_{k}} e^{-\frac{c_5}{2}\gamma (t-t_{k})} \rd t =2,
$$
where the fact that $c_8\geq c_5$ is used.
Therefore,  the numerical solution $\{ Y^{x,\tau}_{t_n}\}_{n\in\N}$ obeys the  LIL in Theorem \ref{thm_numerLIL} with any $f\in \mathcal{C}_{p,\gamma}$ with $p\geq 1$ and $\gamma \in(0,1]$.

\subsection{SPDE} 
For SPDEs, one typically needs to analyze the properties of solutions in various spaces. To accommodate this need, we introduce a subspace refinement of the original Assumptions  \ref{a1}--\ref{a3}, under which Proposition \ref{thm_exactLIL} and Theorem \ref{thm_numerLIL} remain valid. We next formulate these assumptions and state the corresponding LIL results.

Denote by $\tilde{E}$ a subspace of $E$ and $\| \cdot \|_{\tilde{E}}$ the corresponding norm. We further assume that $\| x \| \leq \| x \|_{\tilde{E}}$ for any $x\in \tilde{E}$.
 \begin{assp}\label{a_general1}
	Assume that the time-homogeneous  Markov process $\{X^x_t\}_{t\ge 0}$ satisfies the following conditions.
	\begin{itemize}
		\item[(\romannumeral 1)] There exist  constants $r\ge 2,\tilde r\ge 1$, and $\tilde{L}_1>0$ such that for any $x\in \tilde{E}$,
		\begin{align*}
			&\sup_{t\geq0}\E [\|X^{x}_t\|_{\tilde{E}}^{r}]\leq \tilde{L}_1(1+\|x\|_{\tilde{E}}^{\tilde rr}).
		\end{align*}
		\item[(\romannumeral 2)] There exist constants $\gamma_1\in(0,1]$, $\beta\in[0,r-1],\, \tilde{L}_2>0$,  and a  continuous function $\rho:[0,+\infty)\to[0,+\infty)$  with $\int_{0}^{\infty} (\rho(t))^{\gamma_1} \mathrm dt<\infty$  and $\sum^\infty_{k=1} (\rho(k))^{\gamma_1} <\infty$ 
		such that for any $x,y\in \tilde{E},$
		\begin{align*}
			\big(\E[\|X^{x}_{t}-X^{y}_{t}\|^2]\big)^{\frac{1}{2}}\leq \tilde{L}_2\|x-y\|(1+\|x\|_{\tilde{E}}^{\beta}+\|y\|_{\tilde{E}}^{\beta})\rho(t).
		\end{align*}
		\item[(\romannumeral 3)] There exist constants $l\in[0, r]$ and $\tilde{l}>0$ such that for $t\geq 0$,
		$$
		\big(\E[\|X^{x}_{t}-x\|^2]\big)^{\frac{1}{2}}\leq  \tilde{L}_3 (1+\|x\|_{\tilde{E}}^l) t^{\tilde{l}}.
		$$
	\end{itemize}
\end{assp} 
 \begin{assp}\label{a_general2}
	Assume that $\{Y^{x,\tau}_{t_k}\}_{k\in\mathbb N}$ satisfies  the following conditions.
	\begin{itemize}
		\item[\textup{(i)}] 
		There exist  constants $q\ge 2,\tilde q\ge 1$, and $\tilde{L}_4>0$ such that
		\begin{align*}
			\sup_{k\geq m}\E [\|Y^{x,\tau}_{t_m, t_{k}}\|_{\tilde{E}}^{q}]\leq \tilde{L}_4 (1+\|x\|_{\tilde{E}}^{\tilde qq})
		\end{align*}
		for any fixed $m\in \N$.
		
		\item[\textup{(ii)}] There exist constants 
		$\kappa \in[0,q-1], \tilde{L}_5>0$, and a    function $\rho^{\tau}: [0,+\infty)\to[0,+\infty)$
		such that for any $x, y\in \tilde{E}$, $m,n \in\N$ with $n\geq m$,
		\begin{align*}
			\big(\E[\|Y^{x,\tau}_{t_m,t_n}-Y^{y,\tau}_{t_m,t_n}\|^2]\big)^{\frac{1}{2}}&\leq \tilde{L}_5\|x-y\|(1+\|x\|_{\tilde{E}}^{\kappa}+\|y\|_{\tilde{E}}^{\kappa})\rho^{\tau}(t_n-t_m).
		\end{align*}
	\end{itemize}
\end{assp} 
\begin{assp}\label{a_general3}
	Assume that there exist $\a\in\mathbb R_+$ and  $L_6>0$ such that for any $x\in \tilde{E}$ and $m,n \in \mathbb{N}^+$ such that $n\geq m$,  
	$$
	\big(\E[\|X^{x}_{t_m, t_{n}}-Y^{x,\tau}_{t_m, t_n}\|^2]\big)^{\frac{1}{2}}\leq L_6(1+\|x\|_{\tilde{E}}^{\tilde r\vee \tilde q}) 
	\tau_n^{\alpha}.$$
\end{assp}
Then Proposition \ref{thm_exactLIL} and Theorem \ref{thm_numerLIL} remain valid when replacing 
 Assumptions \ref{a1}--\ref{a3} by Assumptions \ref{a_general1}--\ref{a_general3}.

In this subsection, we check Assumptions \ref{a_general1}--\ref{a_general3} and show the LILs  for the case of SPDE. 
Consider the following infinite-dimensional stochastic evolution equation 
\begin{align}\label{ex_spde}
	\mathrm dX^x_t=AX^x_t\mathrm dt+F(X^x_t)\mathrm dt+\mathrm dW(t), \quad X^x_0=x \in H,
\end{align}
where  $H:=L^2((0,1);\mathbb R)$ endowed with the inner product $\langle \cdot, \cdot \rangle$ and the norm $\|\cdot \|$, 
$A:\textrm{Dom}(A)\subset H\to H$ is the Dirichlet Laplacian with homogeneous Dirichlet boundary conditions, and  $\{W(t)\}_{t\ge 0}$ is a general $Q$-Wiener process on a filtered probability space $(\Omega,\mathcal F,\{\mathcal F_t\}_{t\ge 0},\mathbb P)$.
Let $\| \cdot\|_{\mathcal L(H)}$ and $\| \cdot\|_{\mathcal L^2(H)}$ denote the usual operator norm and the Hilbert--Schmidt operator norm on $H$, respectively.
Assume that $Q$ and $A$ commute, and  that $\|(-A)^{\frac{\beta_1-1}{2}}Q^{\frac12}\|_{\mathcal L_2}<\infty$ for $\beta_1\in(0,1]$. 
Denote by  $\{\lambda_j\}_{j\in\N^+}=\{j^2 \pi^2\}_{j\in\N^+}$  the  eigenvalues of $-A$.
For any $x, y\in H$, let the drift coefficient  $F$  satisfy
\begin{align*}
	\langle x-y, F(x)-F(y)\rangle \leq c_9 \|x-y\|^2
\end{align*}
with $c_9< \lambda_1$ and $\|F(x)-F(y)\| \leq L_F\|x-y\|.$ 
For this equation, we choose the subspace $\dot{H}^{\beta_1}:= \text{Dom}((-A)^{\frac{\beta_1}{2}})$ and denote the corresponding norm by $\|\cdot\|_{\beta_1}$. 

Assumption \ref{a_general1} (\romannumeral1)--(\romannumeral3)  are satisfied with any $r\ge 2$,  $\tilde r=1$, $ \beta=0,$  $\rho(t)=e^{-(\lambda_1-c_9)t}$, $\gamma_1\in(0,1], l=1, \tilde{l}=\frac{\beta_1}{2}$; 
see \cite[Proposition 2.1]{Chen2023} for details.  
Using Proposition \ref{thm_exactLIL}, the exact solution 
$\{X^x_t\}_{t\ge 0}$ admits a unique invariant measure $\mu$ and fulfills the LIL for $f \in \mathcal{C}_{p,\gamma}$ with any $p\geq1$ and $\gamma\in(0,1]$.

Consider the temporal semi-discretization $\{Y^{x,\tau}_{t_i}\}_{i\in\mathbb N}$  based on the exponential Euler method:
\begin{align*}
	Y^{x,\tau}_{t_{k+1}}=S(\tau_{k+1})(Y^{x,\tau}_{t_k}+F(Y^{x,\tau}_{t_k})\tau_{k+1}+\Delta W_k),\quad Y^{x,\tau}_0=x,
\end{align*}
where 
$\Delta W_k=W(t_{k+1})-W(t_k)$ and $ S(t):=e^{At}$.
Here we further assume $\frac{\lambda_1-c_9}{4 L^2_F}\geq 1$ for simplicity and define 
$\{\tau_n \}_{n\in\N^+}=\{ \frac 1n\}_{n\in\N^+}$. For the case that $\frac{\lambda_1-c_9}{4 L^2_F} \in (0,1)$, we have similar results by defining $\{\tau_n\}_{n\in\N^+} = \{ \frac 1n \wedge (\frac{\lambda_1-c_9}{4 L^2_F})\}_{n\in\N^+}.$ 

Similar to  \cite[Propositions 2.3]{Chen2023}, 
we prove that
Assumption \ref{a_general2} (\romannumeral 1) is satisfied with any $q\ge 2$ and $\tilde q=1$. Combining
\begin{align*}
	\E [ \| Y^{x,\tau}_{t_m, t_n}-Y^{y,\tau}_{t_m, t_n}\|^{2}] &\leq e^{-2\lambda_1\tau_{n}} (1+2c_9\tau_{n}+L^2_F\tau_{n}^2) \E [ \| Y^{x,\tau}_{t_m, t_{n-1}}-Y^{y,\tau}_{t_m, t_{n-1}}\|^{2}] \\
		&\leq  e^{-2\lambda_1 (t_{n} -t_m)}  \prod_{i=m+1}^{n} ( 1+2c_9\tau_{i}+L^2_F\tau_{i}^2) \E[  \| x-y\|^{2} ] \\
		&\leq e^{(2c_9-2\lambda_1) (t_{n} -t_m) +L_F^2 (\sum^{n}_{i=m+1} \tau^2_i) }   \E[  \| x-y\|^{2} ]
	\end{align*}
with the facts that $c_9-\lambda_1<0$, $\sum^\infty_{i=1} \tau_i^2$ converges, and $t_n$ tends to infinity as $n\to \infty$, we prove Assumption \ref{a_general2} (\romannumeral 2) with
$\kappa=0$, $\rho^{\tau}(t)=e^{-ct}$ for some $c>0$. 

For Assumption \ref{a_general3}, we 
denote the continuous version of $Y^{x,\tau}_{t_m, t_n}$ by
$$
Y^{x,\tau}_{t_m, t} =  S(t-t_m) x + \int^t_{t_m} S(t-[s]_{\tau}) F(Y^{x,\tau}_{t_m, [s]_{\tau}}) \rd s +\int^t_{t_m} S(t-[s]_{\tau}) \rd W(s)
$$
and 
define an auxiliary process $\bar{Y}^{x,\tau}_{t_m, t}$ such that 
$$
\rd \bar{Y}^{x,\tau}_{t_m, t} =A \bar{Y}^{x,\tau}_{t_m, t} \rd t +S(t-[t]_\tau) F(Y^{x,\tau}_{t_m, [t]_\tau}) \rd t+\rd W(t), \quad \bar{Y}^{x,\tau}_{t_m, t_m}=x,
$$
where $[t]_\tau :=t_n$ denotes the largest point in the sequence $\{t_n\}_{n\in\N}$ such that $t_n\leq t$.
We use  \cite[Lemma B.9]{Kruse2014} and obtain the error estimate between the numercial solution $Y^{x,\tau}_{t_m, t_n}$ and $\bar{Y}^{x,\tau}_{t_m, t_n}$ as
\begin{align} \label{pde_dissipateError}
	&\quad	\E \big[ \| Y^{x,\tau}_{t_m,t_n} -\bar{Y}^{x,\tau}_{t_m, t_n} \|^2 \big] \nn\\
	& \leq \int^{t_n}_{t_m} \| S(t_n-s) (S(s-[s]_{\tau})-I) Q^{\frac 12} \|^2_{\mathcal{L}_2(H)} \rd s \nn\\
	&\leq   \int^{t_n}_{t_m} \| (-A)^{-\frac{\beta_1}{2}} ( S(s-[s]_{\tau})-I)\|^2_{\mathcal{L}(H)} \| (-A)^{\frac 12} S(t_n-s) (-A)^{\frac {\beta_1-1}{2}} Q^{\frac 12} \|^2_{\mathcal{L}_2(H)} \rd s \nn \\
	&\leq K \sum^{n-1}_{i=m} \int^{t_{i+1}}_{t_i} \tau_{i+1}^{\beta_1} \big\| \sum^\infty_{j=1} (-A)^{\frac 12} S(t_n-s) (-A)^{\frac {\beta_1-1}{2}} Q^{\frac 12} e_j  \big\|^2 \rd s  \nn \\
	&= K \sum^{n-1}_{i=m}  \sum^\infty_{j=1}    \int^{t_{i+1}}_{t_i} \tau_{i+1}^{\beta_1} \|(-A)^{\frac {\beta_1-1}{2}} Q^{\frac 12} e_j \|^2  \lambda_j e^{-2\lambda_j(t_n-s) } \rd s  \nn \\
	&\leq K  \sum^\infty_{j=1}   \|(-A)^{\frac {\beta_1-1}{2}} Q^{\frac 12} e_j \|^2  \Big(  \sum^{n-1}_{i=m} \big[   \tau_{i+1}^{\beta_1} \big(e^{-2\lambda_j(t_n-t_{i+1}) } - e^{-2\lambda_j(t_n-t_i) } \big) \big] \Big).
\end{align}
We have estimates that
\begin{align} \label{pde_eEstimate}
	e^{-2\lambda_j(t_n-t_{i+1}) } - e^{-2\lambda_j(t_n-t_i) } &\leq e^{-2\lambda_j(\log (n+1)-\log (i+2) ) } - e^{-2\lambda_j(\log n-\log i ) } \nn\\
	&= \Big(\frac{i+2}{n+1} \Big)^{2\lambda_j} - \Big(\frac{i}{n} \Big)^{2\lambda_j}. 
\end{align}
With the fact that for $\theta \in(0,1)$ and $x >0$,
\begin{align} \label{pde_Taylorexpand1}
	(1+x)^{\theta}  \leq 1+\theta x \quad\text{and} \quad  	(1-x)^{\theta} \geq 1-\theta x,
\end{align}
using \eqref{pde_eEstimate}, \eqref{pde_Taylorexpand1}, and the fact that $2\lambda_j>\beta_1+1$ for all $j\in\N$, we further obtain the upper bound of the sum over $i$ as
\begin{align} \label{pde_sumofi}
	&\quad \sum^{n-1}_{i=m}  \tau_{i+1}^{\beta_1} \big(e^{-2\lambda_j(t_n-t_{i+1}) } - e^{-2\lambda_j(t_n-t_i) } \big) \nn  \\
	&\leq   \sum^{n-1}_{i=m} \Big( \frac{1}{ (i+1)^{\beta_1}} \Big(\frac{i+2}{n+1} \Big)^{2\lambda_j} - \frac{1}{ (i+1)^{\beta_1}} \Big(\frac{i}{n} \Big)^{2\lambda_j}  \Big) \nn \\ 
	&=  \frac{1}{ (n+1)^{2\lambda_j }}  \sum^{n-1}_{i=m} \Big((i+2)^{2\lambda_j-\beta_1} \Big(\frac{i+2}{i+1}\Big)^{\beta_1} -  i^{2\lambda_j-\beta_1} \Big(\frac{i}{i+1}\Big)^{\beta_1} \Big(\frac{n+1}{n}\Big)^{2\lambda_j }\Big) \nn \\
	&\leq   \frac{1}{ (n+1)^{2\lambda_j }}  \sum^{n-1}_{i=m} \Big( (i+2)^{2\lambda_j-\beta_1} ( 1+\beta_1\frac{1}{i+1}) - i^{2\lambda_j-\beta_1}  ( 1-\beta_1\frac{1}{i+1})  \Big) \nn \\
	&\leq   \frac{1}{ (n+1)^{2\lambda_j }}  \Big( \sum^{n-1}_{i=m} \big( (i+2)^{2\lambda_j-\beta_1}-  i^{2\lambda_j-\beta_1}  \big)+ 2\beta_1 \sum^{n-1}_{i=m}   (i+2)^{2\lambda_j-\beta_1-1} +\beta_1 \sum^{n-1}_{i=m}  (i+1)^{2\lambda_j-\beta_1-1}  \Big) \nn \\
	&\leq  \frac{1}{ (n+1)^{2\lambda_j }} \Big((n+1)^{2\lambda_j-\beta_1}+n^{2\lambda_j-\beta_1} +2\beta_1 \big( (n+1)^{ 2\lambda_j-\beta_1-1} + (n+1)^{2\lambda_j-\beta_1 } \big) + \beta_1 (n+1)^{2\lambda_j-\beta_1 }  \Big) \nn \\
	&\leq (2+5\beta_1)\tau_n^{\beta_1}. 
\end{align}
By plugging \eqref{pde_sumofi} into \eqref{pde_dissipateError}, we have
\begin{align} \label{pde_DissipationTermErrortn}
	\E \big[ \| Y^{x,\tau}_{t_m,t_n} -\bar{Y}^{x,\tau}_{t_m, t_n} \|^2 \big] 
	\leq K\sum^\infty_{j=1}   \|(-A)^{\frac {\beta_1-1}{2}} Q^{\frac 12} e_j \|^2 \tau_n^{\beta_1}=  K  \| (-A)^{\frac {\beta_1-1}{2}} Q^{\frac 12}\|^2_{\mathcal{L}_2(H)}  \tau_n^{\beta_1}.
\end{align}
Next, we show that for any $t \in (t_{n-1}, t_{n})$, a similar estimate holds as above. By applying the decomposition in \eqref{pde_dissipateError} together with the technique used in \eqref{pde_sumofi}, we obtain
\begin{align} \label{pde_DissipationTermErrort}
	&\quad \E \big[ \| Y^{x,\tau}_{t_m,t} -\bar{Y}^{x,\tau}_{t_m, t} \|^2 \big] \nn\\
	&\leq K \sum^{n-2}_{i=m} \int^{t_{i+1}}_{t_i} \tau_{i+1}^{\beta_1} \big\| \sum^\infty_{j=1} (-A)^{\frac 12} S(t-s) (-A)^{\frac {\beta_1-1}{2}} Q^{\frac 12} e_j  \big\|^2 \rd s \nn\\
	&\quad + K \int^{t}_{t_{n-1}} (t-t_{n-1})^{\beta_1} \big\| \sum^\infty_{j=1} (-A)^{\frac 12} S(t-s) (-A)^{\frac {\beta_1-1}{2}} Q^{\frac 12} e_j  \big\|^2 \rd s  \nn\\
	&\leq K  \sum^\infty_{j=1}   \|(-A)^{\frac {\beta_1-1}{2}} Q^{\frac 12} e_j \|^2  \Big(  \sum^{n-2}_{i=m} \big[   \tau_{i+1}^{\beta_1} \big(e^{-2\lambda_j(t-t_{i+1}) } - e^{-2\lambda_j(t-t_i) } \big) \big] + (t-t_{n-1})^{\beta_1} \big(1-e^{-2\lambda_j(t-t_{n-1})}\big)  \Big)\nn \\
	&\leq K  \| (-A)^{\frac {\beta_1-1}{2}} Q^{\frac 12}\|^2_{\mathcal{L}_2(H)}  \tau_n^{\beta_1}.
\end{align}
Here, 
the facts that
 \begin{align*}
 	&\quad \sum^{n-2}_{i=m}  \tau_{i+1}^{\beta_1} \big(e^{-2\lambda_j(t-t_{i+1}) } - e^{-2\lambda_j(t-t_i) } \big)  \\
 	&\leq   \sum^{n-2}_{i=m} \Big[ \frac{1}{ (i+1)^{\beta_1}} \Big(\frac{i+2}{n} \Big)^{2\lambda_j} - \frac{1}{ (i+1)^{\beta_1}} \Big(\frac{i}{n} \Big)^{2\lambda_j}  \Big] 
 	\leq K \tau_n^{\beta_1}
 \end{align*}
and
 \begin{align*}
 	(t-t_{n-1})^{\beta_1} \big(1-e^{-2\lambda_j(t-t_{n-1})}\big) \leq \tau_n^{\beta_1}
 \end{align*}
are used. 
For the remaining estimate,  we first consider $\beta_1\in(0,1)$ and have
\begin{align}  \label{pde_decompoOfDrift}
	&\quad \rd \left(e^{b(\lambda_1-c_9)(t-t_m)}\|\bar{Y}^{x,\tau}_{t_m, t} - X^x_{t_m, t}\|^2\right) \nn\\ 
	&\le \big[(b-2)(\lambda_1-c_9) + \varepsilon_1\big] e^{b(\lambda_1-c_9)(t-t_m)}\|\bar{Y}^{x,\tau}_{t_m, t} - X^x_{t_m, t}\|^2 \rd t \nn\\
	&\quad + \gamma e^{b(\lambda_1-c_9)(t-t_m)} \int_{t_m}^t (t-r)^{-\tfrac{\beta_1}{2}} e^{-\tfrac{\lambda_1 (t-r)}{2}}
	\big(2\| \bar{Y}^{x,\tau}_{t_m, r} - X^x_{t_m, r}\|^2 + 2\|Y^{x,\tau}_{t_m, r} - \bar{Y}^{x,\tau}_{t_m, r} \|^2\big)\,\rd r \rd t \nn\\
	&\quad + K e^{b(\lambda_1-c_9)(t-t_m)} \int_{t_m}^t  (t-r)^{-\frac{\beta_1}{2}} e^{-\frac{\lambda_1 (t-r)}{2}}
	\|Y^{x,\tau}_{t_m, r} - Y^{x,\tau}_{t_m,[r]_\tau}\|^2\,\rd r \rd t \nn\\
	&\quad + K (t-[t]_{\tau})^{\beta_1} L_F^2 e^{b(\lambda_1-c_9)(t-t_m)}\big(1+\|Y^{x,\tau}_{t_m, [t]_\tau}\|^2\big) \rd t \nn\\
	&\quad + K e^{b(\lambda_1-c_9)(t-t_m)}  (t-[t]_\tau)^{\frac{\beta_1}{2}}  \nn \\
	&\qquad \times\int_{t_m}^t (r-[r]_\tau)^{\frac{\beta_1}{2}} (t-r)^{-\beta_1} e^{-\frac{\lambda_1 (t-r)}{2}} \big(\|F(Y^{x,\tau}_{t_m, [t]_\tau})\|^2 + \|F(Y^{x,\tau}_{t_m, [r]_\tau})\|^2\big)\,\rd r\rd t\nn \\
	&\quad + K(\varepsilon_1) e^{b(\lambda_1-c_9)(t-t_m)}\big(\|Y^{x,\tau}_{t_m, t} - Y^{x,\tau}_{t_m, [t]_\tau}\|^2 + \|Y^{x,\tau}_{t_m, t}- \bar{Y}^{x,\tau}_{t_m, t}\|^2\big) \rd t \nn\\
	&=: \sum^6_{i=1} J_i (t_m, t) \rd t
\end{align}
for $\varepsilon_1 \in (0,1)$, 
whose proof is similar to \cite[Eq. (B.3)]{Chen2023}.  
Using the Fubini theorem and \eqref{pde_DissipationTermErrort}, by further assuming  that $b<\frac{\lambda_1}{2(\lambda_1-c_9)}$, we have
\begin{align*}
	\E\Big[ \int^{t_n}_{t_m} J_2(t_m, t) \rd t \Big] & \leq 2\gamma \big( \frac{\lambda_1}{2} -b(\lambda_1-c_9) \big)^{-1+\frac{\beta_1}{2}} \Gamma(\frac{2-\beta_1}{2}) \\
	&\quad \times\int^{t_n}_{t_m} e^{b(\lambda_1-c_9)(t-t_m)}  	\big( \E [\| \bar{Y}^{x,\tau}_{t_m, t} - X^x_{t_m, t}\|^2] + \E [ \|Y^{x,\tau}_{t_m, t} - \bar{Y}^{x,\tau}_{t_m, t} \|^2]\big) \, \rd t.
\end{align*}
By taking $\gamma>0$, $\varepsilon_1>0$ small enough, and $0<b< \frac{\lambda_1}{2(\lambda_1-c_9)} \wedge 2$ such that $2\gamma \big( \frac{\lambda_1}{2} -b(\lambda_1-c_9) \big)^{-1+\frac{\beta_1}{2}}\Gamma(\frac{2-\beta_1}{2})+(b-2)(\lambda_1-c_9)+\varepsilon_1\leq 0$, we arrive at
\begin{align*}
		\E\Big[ \int^{t_n}_{t_m} J_1(t_m,t)+  J_2(t_m, t) \rd t \Big] 
		&\leq K  \int^{t_n}_{t_m} e^{b(\lambda_1-c_9)(t-t_m)}  	 \E [ \|Y^{x,\tau}_{t_m, t} - \bar{Y}^{x,\tau}_{t_m, t} \|^2] \, \rd t\\
		& \leq K\sum^{n-1}_{i=m}\int^{t_{i+1}}_{t_i} e^{b(\lambda_1-c_9)(t-t_m)}   \tau_{i+1}^{\beta_1} \, \rd t.
\end{align*}
For $J_3$, we first obtain the H\"older continuity of $\{Y^{x,\tau}_{t_m, t}\}_{t\geq t_m}$ similarly to the proof of \cite[Eq. (2.14)]{Chen2023}, which states as
\begin{align} \label{pde_continuityofY}
	\E [ \|Y^{x,\tau}_{t_m, t} - Y^{x,\tau}_{t_m, s}  \|^2 ] \leq K (1+\|x\|_{\beta_1}^2) (t-s)^{\beta_1} 
\end{align}
for $t_m\leq s\leq t$. Then we see that 
\begin{align*}
	\E\Big[	\int^{t_n}_{t_m} J_3(t_m, t) \rd t  \Big] &\leq  K \int^{t_n}_{t_m} \int_{t_m}^t  e^{b(\lambda_1-c_9)(t-t_m)} (t-r)^{-\frac{\beta_1}{2}} e^{-\frac{\lambda_1 (t-r)}{2}}
\E[	\|Y^{x,\tau}_{t_m, r} - Y^{x,\tau}_{t_m,[r]_\tau}\|^2]\,\rd r \rd t \\
	&\leq K \int^{t_n}_{t_m} e^{b(\lambda_1-c_9)(t-t_m)} 	\E[\|Y^{x,\tau}_{t_m, t} - Y^{x,\tau}_{t_m,[t]_\tau}\|^2]\,\rd t \\
	&\leq K(1+\|x\|_{\beta_1}^2) \sum^{n-1}_{i=m}\int^{t_{i+1}}_{t_i} e^{b(\lambda_1-c_9)(t-t_m)}   \tau_{i+1}^{\beta_1 }\, \rd t.
\end{align*}
For $J_4$, we similarly obtain
\begin{align*}
		\E\Big[	\int^{t_n}_{t_m} J_4(t_m, t) \rd t  \Big]\leq K (1+\|x\|^2)    \sum^{n-1}_{i=m}\int^{t_{i+1}}_{t_i} e^{b(\lambda_1-c_9)(t-t_m)}   \tau_{i+1}^{\beta_1} \, \rd t.
\end{align*}
For $J_5$, we first use the fact that for $t\in (t_{i},t_{i+1}]$,
\begin{align*}
	\int^t_{t_m} (r-[r]_{\tau})^{\frac{\beta_1}{2}} e^{-\frac{\lambda_1(t-r)}{2}} \rd r \leq K \sum^{i}_{j=m} \tau_{j+1}^{\frac{\beta_1}{2}} \big( e^{-\frac{\lambda_1(t-t_{j+1})}{2}} - e^{-\frac{\lambda_1(t-t_{j})}{2}}  \big) \leq K \tau_{i+1}^{\frac{\beta_1}{2} }
\end{align*}
and then obtain
\begin{align} \label{pde_decomposeIntoJ}
	&\quad 	\int^t_{t_m} (r-[r]_{\tau})^{\frac{\beta_1}{2}} (t-r)^{-\beta_1} e^{-\frac{\lambda_1(t-r)}{2}} \rd r \nn\\
		&\leq \big(	\int^t_{t_m}  (t-r)^{-\beta_1(1+\varepsilon_2)} \rd r \big)^{\frac{1}{1+\varepsilon_2} } \big( \int^t_{t_m} (r-[r]_{\tau})^{\frac{\beta_1(1+\varepsilon_2) }{2\varepsilon_2}} e^{-\frac{\lambda_1(1+\varepsilon_2)}{2\varepsilon_2} (t-r)}  \rd r\big)^{\frac{\varepsilon_2}{1+\varepsilon_2}} \nn \\
		&\leq K\tau_{i+1}^{\frac{\beta_1}{2} } 
\end{align}
for any $\varepsilon_2\in(0,1)$.
Here, we choose $\varepsilon_2>0$ small enough such that $\beta_1(1+\varepsilon_2)<1$.
Using  $ \E[\|Y^{x,\tau}_{t_m, r}\|^2 ]\leq K (1+\|x\|^2)$ for any $r\geq t_m$, we have
\begin{align*} 
		\E\Big[	\int^{t_n}_{t_m} J_5(t_m, t) \rd t  \Big] 
		&\leq K  (1+\|x\|^2) \sum^{n-1}_{i=m}\int^{t_{i+1}}_{t_i} e^{b(\lambda_1-c_9)(t-t_m)}   \tau_{i+1}^{\beta_1} \, \rd t.
\end{align*}
For $J_6$, it holds that 
\begin{align*}
	\E\Big[	\int^{t_n}_{t_m} J_6(t_m, t) \rd t \Big] &\leq K (1+\|x\|^2_{\beta_1}) \sum^{n-1}_{i=m}\int^{t_{i+1}}_{t_i} e^{b(\lambda_1-c_9)(t-t_m)}   \tau_{i+1}^{\beta_1} \, \rd t.
\end{align*}
Next we give estimate for $\sum^{n-1}_{i=m}\int^{t_{i+1}}_{t_i} e^{b(\lambda_1-c_9)(t-t_m)}   \tau_{i+1}^{\beta_1} \, \rd t$. 
With the assumption that $b(\lambda_1-c_9) <{\beta_1}$, we come to 
\begin{align} \label{pde_estimateSuminJ}
	\sum^{n-1}_{i=m}\tau_{i+1}^{\beta_1}  \int^{t_{i+1}}_{t_i} e^{b(\lambda_1-c_9)(t-t_m)}   \, \rd t &=  \frac{1}{b(\lambda_1-c_9)}	\sum^{n-1}_{i=m}\tau_{i+1}^{\beta_1} \big( e^{b(\lambda_1-c_9)(t_{i+1}-t_m)} -e^{b(\lambda_1-c_9)(t_{i}-t_m)}  \big) \nn\\
	&\leq  \frac{1}{b(\lambda_1-c_9)}	 \sum^{n-1}_{i=m} \tau_{i+1}^{\beta_1}  e^{b(\lambda_1-c_9)(t_{i}-t_m)} \big(b(\lambda_1-c_9) \tau_{i+1}  e^{b(\lambda_1-c_9)\tau_{i+1}} \big) \nn\\
	&=  \sum^{n-1}_{i=m} \Big(\frac{1}{i+1}\Big)^{1+\beta_1} e^{b(\lambda_1-c_9)(t_{i+1}-t_m)}.
\end{align}
We substitute the estimates of $J_1$ to $J_6$  into \eqref{pde_decompoOfDrift} 
and arrive at
\begin{align*}
	\E [ \|\bar{Y}^{x,\tau}_{t_m, t} - X^x_{t_m, t}\|^2 ] &\leq K  (1+\|x\|^2_{\beta_1}) e^{-b(\lambda_1-c_9)(t_n-t_m)}  \sum^{n-1}_{i=m} \Big(\frac{1}{i+1}\Big)^{1+\beta_1} e^{b(\lambda_1-c_9)(t_{i+1}-t_m)} \\
	&\leq K  (1+\|x\|^2_{\beta_1}) \sum^{n-1}_{i=m} \Big(\frac{1}{i+1}\Big)^{1+\beta_1} e^{-b(\lambda_1-c_9)(t_{n}-t_{i+1})} \\
	&\leq K (1+\|x\|^2_{\beta_1}) \Big[ \Big(\frac{1}{n}\Big)^{1+\beta_1} + \sum^{n-2}_{i=m} \Big(\frac{1}{i+1}\Big)^{1+\beta_1} \Big(\frac{i+2}{n} \Big)^{b(\lambda_1-c_9)} \Big] \\
	&\leq K (1+\|x\|^2_{\beta_1}) \Big[ \Big(\frac{1}{n}\Big)^{1+\beta_1} +  \Big(\frac{1}{n}\Big)^{b(\lambda_1-c_9)} \Big] \\
	&\leq K  (1+\|x\|^2_{\beta_1})\tau_n^{b(\lambda_1-c_9)}.
\end{align*}
Here, the estimate
\begin{align*}
	\sum^{n-2}_{i=m} \Big(\frac{1}{i+1}\Big)^{1+\beta_1} \Big(\frac{i+2}{n} \Big)^{b(\lambda_1-c_9)} &\leq 2 \tau_n^{b(\lambda_1-c_9)}  	\sum^{n-2}_{i=m}   \big(i+1\big)^{b(\lambda_1-c_9)-1-\beta_1} \leq  K  \tau_n^{b(\lambda_1-c_9)}  
\end{align*}
is used. 
Due to the arbitrariness of  $b\in(0,\frac{\beta_1}{\lambda_1-c_9} \wedge2)$, 
we arrive at
\begin{align} \label{pde_DriftTermErrortn}
 \E [\| \bar{Y}^{x,\tau}_{t_m, t_n} -X^x_{t_m,t_n}\|^2 ]\leq K(1+\|x\|_{\beta_1}^2)\tau_n^{(\beta_1 \wedge[ 2(\lambda_1-c_9)] )- \varepsilon} 
\end{align}
for $n\geq m$, where $\varepsilon
 \in (0, \beta_1 \wedge[ 2(\lambda_1-c_9)])$. 

Combining \eqref{pde_DissipationTermErrortn} with \eqref{pde_DriftTermErrortn} shows that 
$$
\E [\| Y^{x,\tau}_{t_m, t_n} -X^x_{t_m,t_n}\|^2 ]\leq K(1+\|x\|_{\beta_1}^2)\tau_n^{(\beta_1 \wedge  [ 2(\lambda_1-c_9)]) -\varepsilon}
$$
for any $n\geq m$ and $\varepsilon\in(0,\beta_1 \wedge  [2(\lambda_1-c_9)])$.
For the case $\beta_1=1$, the only difference in \eqref{pde_decomposeIntoJ} lies in the expression of $J_5(t_m,t)$. Specifically, we have
\begin{align*}
J_5(t_m,t):= K e^{b(\lambda_1-c_9)(t-t_m)}  (t-[t]_\tau)^{\frac{1}{2}}  \int_{t_m}^t (r-[r]_\tau)^{\frac{1}{2}}  e^{-\lambda_1 (t-r)}
\big(\|Y^{x,\tau}_{t_m, [t]_\tau}\|_{1}^2 + \|Y^{x,\tau}_{t_m, [r]_\tau}\|_{1}^2\big)\,\rd r\rd t
\end{align*}
and subsequently,
\begin{align*}
		\E\Big[	\int^{t_n}_{t_m} J_5(t_m, t) \rd t  \Big] 
&\leq K  (1+\|x\|_{\beta_1}^2) \sum^{n-1}_{i=m}\int^{t_{i+1}}_{t_i} e^{b(\lambda_1-c_9)(t-t_m)}   \tau_{i+1}^{\beta_1} \, \rd t,	
\end{align*}
while the remaining estimates proceed in the same way as above.
Hence, Assumption \ref{a_general3} is satisfied with $\alpha=(\beta_1 \wedge  [ 2(\lambda_1-c_9)]) -\varepsilon$. Using Theorem \ref{thm_numerLIL}, we prove that for any $f\in \mathcal{C}_{p,\gamma}$ with $p\geq 1$ and $\gamma \in(0,1]$, the numerical solution $\{Y^{x,\tau}_{t_n}\}_{n\in\N}$ obeys the LIL.

\appendix
\section{Proof of Proposition \ref{thm_exactLIL}} \label{Appendix1}
 Under Assumption \ref{a1}, 
the existence and uniqueness of the invariant measure
follow from  standard arguments, which are therefore omitted here. We denote  the invariant measure by $\mu$.
 
For the time-homogeneous Markov process $\{X_t\}_{t\geq 0}$, we decompose
$\int^t_0 f(X^x_s)-\mu(f) \rd s$ into $R^x_t+\mathscr M^{x}_{\lfloor t \rfloor}+\mathscr R^{x}_{\lfloor t \rfloor}$ for $t \geq 0$, where
\begin{align*}
	R^x_t:=& \int^t_{\lfloor t \rfloor} \big(f(X^x_s)-\mu(f) \big) \rd s, \\
	\mathscr M^{x}_{\lfloor t \rfloor}:=&\int_{0}^{\infty}\Big(\E\big[f(X_{s}^{x})|\mathcal F_{\lfloor t \rfloor}\big]-\mu(f)\Big)\mathrm ds-\int_{0}^{\infty}\Big(\E\big[f(X_{s}^{x})|\mathcal F_{0}\big]-\mu(f)\Big)\mathrm ds,
\end{align*}
and 
\begin{align*}
		\mathscr R^{x}_{\lfloor t \rfloor}:=&-\int_{\lfloor t \rfloor}^{\infty}\Big(\E\big[f(X_{s}^{x})|\mathcal F_{\lfloor t \rfloor}\big]-\mu(f)\Big)\mathrm ds+\int_{0}^{\infty}\Big(\E\big[f(X_{s}^{x})|\mathcal F_{0}\big]-\mu(f)\Big)\mathrm ds.
\end{align*}
Then we have the integrability result as
\begin{align*}
	\E\big[ |\mathscr M^{x}_{k} | \big]
	&\leq Kk\|f\|_{p,\gamma_1}\Big(1+ \|x\|^{\frac{p\tilde{r}^2}{2}+\tilde{r}(1+\beta)\gamma_1}\Big)<\infty,
\end{align*}
where we use  $\frac{p\tilde{r}}{2}+(1+\beta)\gamma_1 \leq r$,
Assumption \ref{a1}, and the estimate
\begin{align*}
	|\mathscr M^{x}_{k}|
	&\leq\int_{0}^{k}  \big|f(X_{t}^{x})-\mu(f)\big| \rd t 
	+ \int_{k}^{\infty}  \big| P_{t-k}f(X_{k}^{x})-\mu(f)\big| \rd t + \int_{0}^{\infty} \big| P_{t}f(x)-\mu(f)\big|  \rd t \\
	&\leq K\|f\|_{p,\gamma_1}  \int^k_0 (1+\| X_{t}^{x}\|^{\frac p2}) \ \rd t
	+K\|f\|_{p,\gamma_1}(1+\|X_{k}^{x}\|^{\frac{p\tilde{r}}{2}+(1+\beta)\gamma_1}) \int^\infty_{k} (\rho(t-k))^{\gamma_1} \rd t\\ &\quad+K\|f\|_{p,\gamma_1}(1+\|x\|^{\frac{p\tilde{r}}{2}+(1+\beta)\gamma_1} ) \int^\infty_{0} (\rho(t))^{\gamma_1} \rd t.
\end{align*}
For any $s\in \N$ such that $s\leq k$, we see that
\begin{align*}
&\quad 	\E \Big[ \mathscr M^{x,\tau}_{k} \Big\vert \mathcal{F}_{s} \Big] \\ 
&=   \E \bigg[ \int_{0}^{\infty}\Big(\E\big[f(X_{t}^{x})|\mathcal F_{k}\big]-\mu(f)\Big)\mathrm dt-\int_{0}^{\infty}\Big(\E\big[f(X_{t}^{x})|\mathcal F_{0}\big]-\mu(f)\Big)\mathrm dt \bigg\vert  \mathcal{F}_{s} \bigg] \\
&=   \int_{0}^{\infty}\Big(\E\big[f(X_{t}^{x})|\mathcal F_{s}\big]-\mu(f)\Big)\mathrm dt-\int_{0}^{\infty}\Big(\E\big[f(X_{t}^{x})|\mathcal F_{0}\big]-\mu(f)\Big)\mathrm dt = \mathscr M^{x,\tau}_{s}. 
\end{align*}
Then we find that $\{\mathscr M^{x}_{n}\}_{n\in\N}$ is a martingale. 
 We define the corresponding martingale difference sequence by $\mathcal{Z}^x_0=0$ and $\mathcal{Z}^x_n :=\mathscr M^{x}_{n}-\mathscr M^{x}_{n-1}$ for $n\in\N^+$ and verify  \eqref{equ_LILCondition1}--\eqref{equ_LILCondition2} in Lemma \ref{lemma_Heyde1973}. 

With the expression
\begin{align*}
	\mathcal Z^{x}_{n+1}
	&=\int_{n}^{n+1}\Big(f(X^{x}_t)-\mu(f)\Big)\mathrm dt+\int_{0}^{\infty}\Big(P_{t}f(X^{x}_{n+1})-\mu(f)\Big)\mathrm dt\\
	&\quad -\int_{0}^{\infty}\Big(P_{t}f(X^{x}_n)-\mu(f)\Big)\mathrm dt
\end{align*}
for all $n\in\N$ 
and by an argument similar to that in Proposition \ref{prop_ZkIntegrable}, we deduce that for any $p\geq 1$  satisfying  $2p\tilde{r} +4(1+\beta)\gamma_1 \leq r$,  it holds that
$\sup_{n\in\N} \E[ | \mathcal{Z}^x_n|^4] \leq K$.
Furthermore, using the time-homogeneous property and expressing the conditional expectation of $\mathcal{Z}^x_n$ as in \eqref{equ_GuExpression}, we obtain
$$
\E \big[ \big|\mathcal{Z}^x_{n} \big|^2 \big]=\E\big[
\E \big[ \big|\mathcal{Z}^x_{n}|^2 \big| \mathcal{F}_{n-1}\big]\big]=\E\big[H(X^x_{n-1})\big],
$$ 
where $$H(x):=  \E\Big[\big|\int_{0}^{\infty}P_{t}f(X_{1}^{x})\mathrm dt\big|^2\Big]-\big|\int_{0}^{\infty}P_{t}f(x)\mathrm dt\big|^2+2\E\Big[\int_{0}^{1}f(X^{x}_t) \int_{0}^{\infty}P_{s}f(X_{t}^{x})\mathrm ds\mathrm dt\Big].$$ 
Similar to the proof of $F^\tau_{k-1} \in \mathcal C_{\tilde{p}_F, \gamma}$ in \Cref{prop_ZkSquareMoment},  by  defining  $\bar{p}_{\gamma_1}:=2\tilde{r}(p\tilde{r}+(2+3\beta)\gamma_1)$ and examining the definition of $\mathcal C_{\bar{p}_{\gamma_1},\gamma_1}$, we can show that $H \in \mathcal C_{\bar{p}_{\gamma_1}, \gamma_1}$ with $\|H\|_{\bar{p}_{\gamma_1}, \gamma_1} \leq K\|f\|_{p,\gamma_1}^2$.
Since $H$ can be rewritten as
\begin{align*}
	H=P_1\Big|\int_0^{\infty}P_tf\mathrm dt\Big|^2-\Big|\int_0^{\infty}P_tf\mathrm dt\Big|^2+2\int_0^1P_t\Big(\int_0^{\infty}fP_sf\mathrm ds\Big)\mathrm dt,
\end{align*}
and noting that $P_t^*\mu=\mu$ for $t>0$, it follows that $\mu(H)=v^2$. 
Hence, with the aid of   \eqref{equ_apprExactPttoMu}, 
we arrive at 
\begin{align} \label{equ_exactlimitS}
	\lim_{n\to \infty} \frac{\E[| \mathscr{M}^x_n|^2]}{n} = \lim_{n\to\infty} \frac{1}{n} \sum^n_{k=1} P_{k-1} H(x) =\mu(H)=v^2
\end{align}
for $\bar{p}_{\gamma_1}\leq r$, 
which shows that $\E[| \mathscr{M}^x_n|^2] \to \infty$ as $n\to \infty$. Combining the boundness of $\E[|\mathcal{Z}^x_n|^4]$ and \eqref{equ_exactlimitS}, \eqref{equ_LILCondition1} and \eqref{equ_LILCondition1_1} hold. In the following proposition, we give the result of the almost sure convergence property of $\frac 1n \sum^n_{k=1} (\mathcal{Z}_k^x)^2$, which leads to \eqref{equ_LILCondition2} in a similar manner as the proof for Proposition \ref{prop_LILforMartingale}.

\begin{prop} \label{prop_exactalmostsureZ}
	Under assumptions of  \Cref{thm_exactLIL},
	\begin{align*}
		\lim_{n\to \infty} \frac{1}{n} \sum^n_{k=1} (\mathcal{Z}^x_k)^2 =v^2, \quad \text{a.s.}
	\end{align*}
\end{prop}
\begin{proof}
	Since the process $\{X^x_t\}_{t\geq 0}$ 
	is time-homogeneous and $\mu$  is its invariant measure, the proof of Proposition \ref{prop_exactalmostsureZ} follows from Birkhoff’s individual ergodic theorem.
	 According to the proof of \cite[Lemma 3.2]{Bolt2012}, we only need to show the continuity of 
	\begin{align*}
		\Lambda_1(x) := \E \Big[ \big| \limsup_{n\to \infty} \big(\frac{1}{n} \sum^n_{k=1} (\mathcal{Z}^x_k)^2\big) -v^2 \big| \wedge 1\Big] 
	\end{align*}
	and 
	\begin{align*}
		\Lambda_2(x) := \E \Big[ \big| \liminf_{n\to \infty} \big(\frac{1}{n} \sum^n_{k=1} (\mathcal{Z}^x_k)^2\big) -v^2 \big| \wedge 1\Big].
	\end{align*}
By applying the same technique as in \cite[Lemma 4.2]{Wu2022}, we deduce that
\begin{align*}
	\big|	\Lambda_1(x)- 	\Lambda_1(y) \big| \leq 2(1+v^2) \sum^\infty_{k=K_0} \E[ | \mathcal{Z}_k^x -\mathcal{Z}_k^y|],
\end{align*}
where $K_0>1$ is a fixed integer. In the following, we only need to estimate $\sum^\infty_{k=K_0} \E[ | \mathcal{Z}_k^x -\mathcal{Z}_k^y|]$.
From Assumption \ref{a1} (ii), we first obtain 
\begin{align*}
	\big|f(X^x_t) -f(X^y_t) \big| 
	&\leq \|f\|_{p,\gamma_1} (1\wedge\|X^x_t-X^y_t\|^{\gamma_1}) \big(1+\|X^x_t\|^p+\|X^y_t\|^p\big)^{\frac 12}
\end{align*}
and 
\begin{align*}
	\big|P_tf(x) -P_tf(y) \big| & \leq \|f\|_{p,\gamma_1} \big(1+\E[\|X^x_{t}\|^p]+\E[\|X^y_{t}\|^p]\big)^{\frac 12} \big(\mathbb{W}_{2} (\mu^x_{t},\mu^y_{t}) \big)^{\gamma_1} \\
	&\leq \|f\|_{p,\gamma_1} \|x-y\|^{\gamma_1} (1+\|x\|^{\frac{p\tilde{r}}{2}+\beta\gamma_1}+\|y\|^{\frac{p\tilde{r}}{2}+\beta\gamma_1}) (\rho(t))^{\gamma_1}. 
\end{align*}
For the martingale difference sequence $\{\mathcal{Z}^x_n\}_{n\in\N}$, 
	we have  
	\begin{align*}
		\big| \mathcal{Z}^x_{n}-\mathcal{Z}^y_{n}\big| 
		&\leq   \int^n_{n-1} \|f\|_{p,\gamma_1} \|X^x_t-X^y_t\|^{\gamma_1} \big(1+\|X^x_t\|^{\frac p2}+\|X^y_t\|^{\frac p2}\big) \rd t \\
		&\quad+ K\|f\|_{p,\gamma_1}\|X^x_n-X^y_n\|^{\gamma_1} (1+\|X^x_n\|^{\frac{p\tilde{r}}{2}+\beta\gamma_1}+\|X^y_n\|^{\frac{p\tilde{r}}{2}+\beta\gamma_1}) \\
		&\quad +K\|f\|_{p,\gamma_1}\|X^x_{n-1}-X^y_{n-1}\|^{\gamma_1} (1+\|X^x_{n-1}\|^{\frac{p\tilde{r}}{2}+\beta\gamma_1}+\|X^y_{n-1}\|^{\frac{p\tilde{r}}{2}+\beta\gamma_1}).
	\end{align*}
		With the conditions that $p\leq r$ and $p\tilde{r}+2\beta\gamma_1 \leq r$, it follows that
	\begin{align*}
		&\quad \E \big[  |\mathcal{Z}^x_n - \mathcal{Z}^y_n| \big] \\
		&\leq K\|f\|_{p,{\gamma_1}} \Big( \int^n_{n-1} (\E[\|X^x_t-X^y_t\|^{2{\gamma_1}}] )^{\frac 12} \big(1+\E[\|X^x_t\|^{p}]+\E[\|X^y_t\|^{p}] \big)^{\frac 12} \rd t \\
		&\quad + (\E[\|X^x_n-X^y_n\|^{2{\gamma_1}}] )^{\frac 12} \big(1+\E[\|X^x_n\|^{p\tilde{r}+2\beta{\gamma_1}}]+\E[\|X^y_n\|^{p\tilde{r}+2\beta{\gamma_1}}] \big)^{\frac 12} \\
		&\quad + (\E[\|X^x_{n-1}-X^y_{n-1}\|^{2{\gamma_1}}] )^{\frac 12} \big(1+\E[\|X^x_{n-1}\|^{p\tilde{r}+2\beta{\gamma_1}}]+\E[\|X^y_{n-1}\|^{p\tilde{r}+2\beta{\gamma_1}}] \big)^{\frac 12} \Big) \\
		&\leq  K\|f\|_{p,{\gamma_1}}  \|x-y\|^{\gamma_1} (1+\|x\|^\beta+\|y\|^\beta)^{\gamma_1} \Big( (1+\|x\|^{p\tilde{r}}+\|y\|^{p\tilde{r}})^{\frac 12} \int^n_{n-1} (\rho(t))^{\gamma_1} \rd t \\
		&\quad + (1+\|x\|^{\tilde{r}(p\tilde{r}+2\beta{\gamma_1})}+\|y\|^{\tilde{r}(p\tilde{r}+2\beta{\gamma_1})})^{\frac 12} \big[(\rho(n))^{\gamma_1} +(\rho(n-1))^{\gamma_1} \big]\Big).
	\end{align*}
	This yields
	\begin{align*}
		&\quad \sum^\infty_{k=K_0} \E\big[ \big|\mathcal{Z}^x_k-\mathcal{Z}^y_k\big|\big]\\
		& \leq K\|f\|_{p,{\gamma_1}}  \|x-y\|^{\gamma_1} (1+\|x\|^{{\gamma_1}\beta}+\|y\|^{{\gamma_1}\beta}) \big( 1+\|x\|^{\frac{p\tilde{r}}{2}}+\|y\|^{\frac{p\tilde{r}}{2}} +\|x\|^{\tilde{r}(\frac{p\tilde{r}}{2}+\beta{\gamma_1})}+\|y\|^{\tilde{r}(\frac{p\tilde{r}}{2}+\beta{\gamma_1})}\big)
	\end{align*}
	for $f\in \mathcal{C}_{p,\gamma_1}$. 
	Therefore, we obtain $\limsup_{x\to y} |\Lambda_1(x)-\Lambda_1(y)|=0$. In a similar manner, we also have $\limsup_{x\to y} |\Lambda_2(x)-\Lambda_2(y)|=0$, which completes the proof.
\end{proof}

Similar to Proposition \ref{prop_LILforMartingale}, by applying \eqref{equ_exactlimitS} together with Proposition \ref{prop_exactalmostsureZ}, we verify conditions \eqref{equ_LILCondition1}–\eqref{equ_LILCondition2} in Lemma \ref{lemma_Heyde1973} and thereby establish the LIL for the martingale $\{\mathscr{M}^{x}_{n}\}_{n\in\N}$:
$$
\limsup_{n \to \infty} \frac{\mathscr{M}^{x}_{n} }{\sqrt{2 {n} \log \log {n}}}  = v, \quad \text{a.s.} \qquad \text{and} \qquad   \liminf_{n\to \infty} \frac{\mathscr{M}^{x}_{n} }{\sqrt{2 {n} \log \log {n}}}  = -v, \quad \text{a.s.}
$$
For the remaining part $\{\mathscr{R}^x_{n}\}_{n\in\N}$, by applying \eqref{equ_apprExactPttoMu}, we obtain 
\begin{align*}
		| \mathscr R^{x}_{n} | &\leq \int^\infty_n \Big| P_{t-n} f(X_{n}^{x})-\mu(f)\Big|  \rd t+ \int_{0}^{\infty}\Big|P_{t} f(x)-\mu(f)\Big|  \rd t\\
	&\leq  K\|f\|_{p,{\gamma_1}}\Big(1+\|x\|^{\frac{p\tilde r}{2}+(1+\beta){\gamma_1}}+\|X_n^x\|^{\frac{p\tilde r}{2}+(1+\beta){\gamma_1}} \Big).
\end{align*}
This implies that, for any $\varepsilon>0$,
\begin{align*}
	\sum^{\infty}_{n=1} \mathbb{P} \{  \frac{1}{\sqrt{n}}|\mathscr R^{x}_{n}| > \varepsilon \} &\leq \frac{1}{\varepsilon^4} 	\sum^{\infty}_{n=1}  \frac{1}{n^2}\E [ 	| \mathscr R^{x}_{n} |^4 ] \\
	&\leq \frac{K}{\varepsilon^4} \|f\|_{p,{\gamma_1}}  \big(1+\|x\|^{\tilde{r}(2p\tilde{r} +4(1+\beta){\gamma_1})} \big)  	\sum^{\infty}_{n=1}\frac{1}{n^2} <\infty.
\end{align*}
Applying the Borel--Cantelli lemma, we conclude that  $\lim_{n\to\infty} \frac{1}{\sqrt{n}} |\mathscr R^{x}_{n}| =0$ almost surely.

By denoting $\bar{R}^{x}_{k}:= \int^{k+1}_{k} | f(X^{x}_{s}) - \mu(f)| \rd s$,
we have
\begin{align} \label{equ_estimateOfRk}
\E\big[\big|	R^{x}_{t}\big|^4\big]\leq \E\big[\big|	\bar{R}^{x}_{\lfloor t \rfloor}\big|^4\big] \leq\int^{\lfloor t \rfloor+1}_{\lfloor t \rfloor}  \E \big[| f(X^{x}_{s}) - \mu(f)|^4 \big] \rd s \leq K\|f\|^4_{p,{\gamma_1}} (1+\|x\|^{2p\tilde{r}}).
\end{align}
Combining \eqref{equ_estimateOfRk} with the Borel--Cantelli lemma, we find that 
$$
\lim_{t\to\infty} \frac{1}{\sqrt{t}} |R^x_t|=0, \quad \text{a.s.}
$$
The proof of \Cref{thm_exactLIL} is completed in a similar manner to that of Theorem \ref{thm_numerLIL}.
\bibliographystyle{plain}
\bibliography{reference.bib}

\end{document}